\documentclass[11pt]{amsart}
\usepackage[T1]{fontenc}
\usepackage{lmodern}
\usepackage{amsmath,amssymb,mathtools}
\usepackage{graphicx,tikz}
\usepackage[hypertexnames=false,colorlinks=true,linkcolor=black,citecolor=green!40!black,urlcolor=black]{hyperref}
\usetikzlibrary{arrows.meta}
\numberwithin{equation}{section}
\theoremstyle{plain}
\newtheorem{theorem}{Theorem}[section]
\newtheorem{proposition}[theorem]{Proposition}
\newtheorem{lemma}[theorem]{Lemma}
\newtheorem{corollary}[theorem]{Corollary}
\theoremstyle{remark}
\newtheorem{remark}[theorem]{Remark}
\newcommand{\R}{\mathbb R}

\newcommand{\I}{\mathbb I}
\newcommand{\M}{\mathcal M}
\newcommand{\A}{\mathcal A}

\DeclareMathOperator{\supp}{supp}
\DeclareMathOperator{\diag}{diag}
\DeclareMathOperator{\conv}{conv}

\allowdisplaybreaks[2]
\makeatletter
\@ifundefined{subjclassname@2020}{%
  \@namedef{subjclassname@2020}{%
    \textup{2020} Mathematics Subject Classification}%
}{}
\makeatother

\begin{document}
\title[The strong circular maximal operator]
{Sharp $L^p$--$L^q$ bounds for the strong circular maximal operator}
\author{Jin Bong Lee}
\address{Research Institute of Mathematics, Seoul National University, Seoul 08826, Republic of Korea}
\email{jinblee@snu.ac.kr}
\author{Sanghyuk Lee}
\address{Department of Mathematical Sciences and RIM, Seoul National University, Seoul 08826, Republic of Korea}
\email{shklee@snu.ac.kr}
\author{Seunghwan Lee}
\address{Department of Mathematical Sciences, Seoul National University, Seoul 08826, Republic of Korea}
\email{swlee6720@snu.ac.kr}
\author{Jeongtae Oh}
\address{Research Institute of Mathematics, Seoul National University, Seoul 08826, Republic of Korea}
\email{ojt0117@snu.ac.kr}
\date{}
\begin{abstract}
We study $L^p$--$L^q$ boundedness of the strong circular maximal operator $\M$, the maximal function over axis-parallel ellipses centered at the origin whose two semiaxes range independently over $[1,2]$. 
We prove that $\M$ is bounded from $L^p$ to $L^q$ whenever $(1/p,1/q)$ lies in the interior of the triangle with vertices $(0,0)$, $(2/7,1/7)$, and $(1/3,1/3)$, the two-parameter analogue of the triangle with vertices $(0,0)$, $(2/5,1/5)$, and $(1/2,1/2)$ obtained by Schlag and Schlag--Sogge for the circular maximal function. 
The range is sharp up to the boundary; in particular, weak type $(p,q)$ bounds fail for every $p\le3$. 
The proof is based on a bilinear estimate for frequencies in separated angular sectors, uniform as the sectors approach the coordinate axes. 
We also prove the analogous sharp result for axis dilations of a translated circle $c+\mathbb S^1$ with fixed $c\ne0$, where the vertex $(2/7,1/7)$ is replaced by $(1/4,1/8)$.
\end{abstract}
\maketitle

\section{Introduction}\label{sec:introduction}

Let $d\sigma$ be the normalized arclength measure on the unit circle
$\mathbb S^1$. For $t,s>0$, define the probability measure $\sigma_{t,s}$ on
the ellipse $\{(t\cos\theta,s\sin\theta)\}$ by
\[
 \int g(y)\,d\sigma_{t,s}(y)=\int_{\mathbb S^1}g(ty_1,sy_2)\,d\sigma(y).
\]
In other words, $\sigma_{t,s}$ is the pushforward of $d\sigma$ under the
linear map $y\mapsto(ty_1,sy_2)$; it is not the normalized arclength
measure of the ellipse. For $t,s\in[1,2]$, the two measures have mutually
bounded densities, so this choice does not affect the results below.

\subsection{The strong circular maximal operator}\label{sec:operator}
Put $\I=[1,2]$, $\A_{t,s}f=f*\sigma_{t,s}$, and
\begin{equation}\label{eq:maximal-operator}
 \M f(x)=\sup_{(t,s)\in\I^2}|\A_{t,s}f(x)|.
\end{equation}
Following~\cite{LLOelliptic}, we call $\M$ the strong circular maximal operator: the
ellipses are axis-parallel and the two semiaxes vary independently.
Throughout the paper $(t,s)$ ranges over the fixed square $\I^2$; the
global operator, with $(t,s)\in(0,\infty)^2$, is discussed below.

For $c=(c_1,c_2)\in\R^2$, the corresponding averages over axis dilations
of $c+\mathbb S^1$ and their maximal operator are given by 
\[
 \begin{aligned}
 \A^c_{t,s}f(x)&=\A_{t,s}f\bigl(x-(tc_1,sc_2)\bigr),\\[3pt]
 \M^cf(x)&=\sup_{(t,s)\in\I^2}|\A^c_{t,s}f(x)|.
 \end{aligned}
\]
We refer to $c=0$ as the centered case and to a fixed $c\ne0$ as the
translated case. Here $c$ is the center of the circle before dilation,
so for $c\ne0$ the center $(tc_1,sc_2)$ of the dilated ellipse moves
with $(t,s)$. Note that $\A^0_{t,s}=\A_{t,s}$ and $\M^0=\M$.

To place our results in context, we first consider the circular
maximal function
\[
 \M_{\mathrm{circ}} f(x)=\sup_{1<t<2}|\A_{t,t}f(x)|,
\]
which corresponds to the diagonal $t=s$. Bourgain~\cite{Bourgain} proved that the global
circular maximal function, with $t>0$, is bounded on $L^p$ for $p>2$;
see also Mockenhaupt--Seeger--Sogge~\cite{MSS} for a proof based on
local smoothing. For the local operator, Schlag~\cite{Schlag} showed that $\M_{\mathrm{circ}}$ is bounded from $L^p$ to
$L^q$ when $(1/p,1/q)$ lies in the interior of the triangle with
vertices $(0,0)$, $(1/2,1/2)$, and $(2/5,1/5)$, and Schlag--Sogge~\cite{SS} gave a proof based on local smoothing estimates. These
estimates fail outside the closed triangle and at $(1/2,1/2)$, and Lee~\cite{Lee} proved them on the remaining boundary except at
$(2/5,1/5)$, where the restricted weak type estimate
$L^{5/2,1}\to L^{5,\infty}$ holds. Recently, Lee--Zhao~\cite{LZ} proved
the strong type estimate $L^{5/2}\to L^5$ at this vertex, so the
$L^p$--$L^q$ bounds for $\M_{\mathrm{circ}}$ are now completely
determined.

As with $\M_{\mathrm{circ}}$, $L^p$ bounds for $\M$ were obtained first.
Lee--Lee--Oh~\cite[Theorem 1.2]{LLOelliptic} proved $L^p$ boundedness
for $p>4$ of the global strong circular maximal operator, with
$(t,s)\in(0,\infty)^2$, and
hence of $\M$; they observed that a modification of an example of
Erdo\u gan~\cite{Erdogan} rules out $p\le3$. Their Fourier
representation, parameter Sobolev reduction, and local smoothing
estimates provide the analytic framework used here. The sharp range
$p>3$ for $\M$ follows by combining two-parameter local smoothing at
large exponents with the incidence estimate of Pramanik--Yang--Zahl~\cite{PYZ}; Lemma~\ref{lem:diagonal} below uses the local smoothing
estimate~\cite[Proposition 2.3]{LLOelliptic}. Chen--Guo--Yang~\cite[Corollary 1.2 and Section 3]{CGY} obtained the same range from
their general local smoothing theorem and~\cite{PYZ}, and the
added-in-proof note in~\cite{LLOelliptic} states that it holds for the
global operator. Following the coordinatewise Littlewood--Paley
decomposition and summation argument in~\cite[Section 2.1]{LLOelliptic},
with Lemma~\ref{lem:translated-diagonal} providing the local frequency
decay, we obtain $L^p$ boundedness for $p>3$ of the global operator
$\sup_{(t,s)\in(0,\infty)^2}|\A^c_{t,s}f|$. The anisotropic rescaling
preserves the fixed translation parameter $c$, and the implicit
constants may depend on $c$.

In this paper we determine, for each
fixed $c\in\R^2$ and up to the boundary, the set of $(1/p,1/q)$ for which $\M^c$ is bounded
from $L^p(\R^2)$ to $L^q(\R^2)$.

\subsection{The centered case}\label{sec:main-result}
Define
\begin{equation}\label{eq:triangle}
 \Delta=\conv\left\{(0,0),\left(\frac27,\frac17\right),
                         \left(\frac13,\frac13\right)\right\}.
\end{equation}
Our first result is the following, which establishes $L^p$--$L^q$
boundedness of $\M$ on a range of $(1/p,1/q)$ that is sharp up to the
boundary.

\begin{theorem}\label{thm:main}
If $(1/p,1/q)\in\operatorname{int}\Delta$, or if $3<p=q\le\infty$,
then
\begin{equation}\label{eq:main-bound}
 \|\M f\|_{L^q(\R^2)}\le C_{p,q}\|f\|_{L^p(\R^2)}.
\end{equation}
Conversely, if \eqref{eq:main-bound} holds for some $1\le p,q\le\infty$, then
$(1/p,1/q)\in\Delta$ and $p>3$.
\end{theorem}

Since $\M$ dominates $\M_{\mathrm{circ}}$, the region $\Delta$ lies in
the triangle for $\M_{\mathrm{circ}}$ recalled in
Section~\ref{sec:operator}: the vertices $(2/5,1/5)$ and $(1/2,1/2)$ are
replaced by $(2/7,1/7)$ and $(1/3,1/3)$. The diagonal part of
Theorem~\ref{thm:main} is the $L^p$ range recalled there, together with
the $L^\infty$ bound from the unit mass of $\sigma_{t,s}$; the new part
is the off-diagonal range.

The condition $p>3$ cannot be relaxed even to weak type. For
$1\le p\le3$, the weak type estimate
\[
 \|\M f\|_{L^{q,\infty}(\R^2)}\lesssim\|f\|_{L^p(\R^2)}
\]
fails for every $1\le q\le\infty$, even on nonnegative smooth compactly
supported inputs; here $L^{\infty,\infty}=L^\infty$. This is proved by a
smooth cusp construction at the end of Section~\ref{sec:necessary-conditions}.

Theorem~\ref{thm:main} rests on the following frequency-localized estimate at the
vertex $(2/7,1/7)$. Let $P_\lambda$ be the Littlewood--Paley projection
to $|\xi|\sim\lambda$ (see Section~\ref{sec:frequency-reduction}), and set
\[
 \M^c_\lambda f=\M^c(P_\lambda f),
\]
with $\M_\lambda=\M^0_\lambda$.

\begin{theorem}\label{thm:frequency}
For every $\varepsilon>0$ there is a constant $C_\varepsilon$ such that
\begin{equation}\label{eq:frequency-bound}
 \|\M_\lambda f\|_7\le C_\varepsilon\lambda^\varepsilon\|f\|_{7/2}
\end{equation}
for all dyadic $\lambda\ge2$.
\end{theorem}

This is the two-parameter analogue of the $L^{5/2}\to L^5$ estimate at
the vertex $(2/5,1/5)$ for the circular maximal function. Interpolation
with the diagonal decay of Lemma~\ref{lem:diagonal} gives the open triangle.

\begin{figure}[htbp]
\centering
\begin{tikzpicture}[x=14cm,y=14cm,>=Stealth]
 \coordinate (O) at (0,0);
 \coordinate (V) at (2/7,1/7);
 \coordinate (D) at (1/3,1/3);
 \fill[black!12] (O)--(V)--(D)--cycle;
 \draw[->] (-.025,0)--(.40,0) node[right] {$1/p$};
 \draw[->] (0,-.025)--(0,.39) node[above] {$1/q$};
 \draw[thick] (O)--(D);
 \draw[thick,densely dotted] (O)--(V)--(D);
 \fill (O) circle (1.5pt);
 \draw[fill=white] (V) circle (1.8pt);
 \draw[fill=white] (D) circle (1.8pt);
 \node[below left] at (O) {$(0,0)$};
 \node[right=3pt] at (V) {$\left(\frac27,\frac17\right)$};
 \node[above right=1pt] at (D) {$\left(\frac13,\frac13\right)$};
\end{tikzpicture}
\caption{The shaded interior is the $L^p$-improving range proved here.
Boundedness also holds on the solid diagonal, including the origin and
excluding $(1/3,1/3)$. The dotted edges and their common vertex remain
open; weak type fails at $(1/3,1/3)$.}\label{fig:mapping-region}
\end{figure}
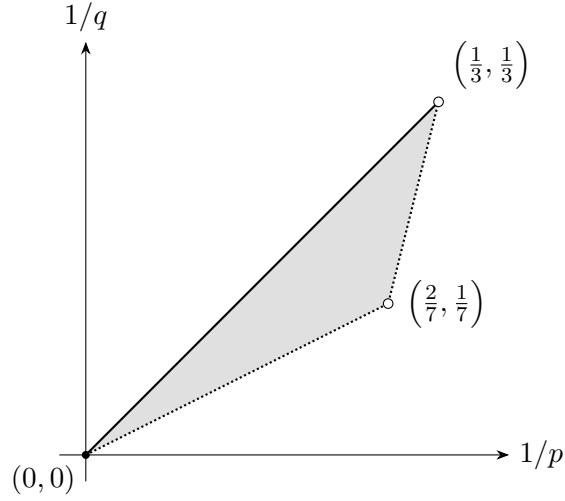

The following questions remain open. (i) Restricted weak type and
strong type at the vertex $(2/7,1/7)$; for the circular maximal
function, the corresponding estimates at $(2/5,1/5)$ were proved in~\cite{Lee} and~\cite{LZ}, respectively. The loss $\lambda^\varepsilon$
in Theorem~\ref{thm:frequency} is the obstacle. (ii) Strong type on the two off-diagonal
edges of $\Delta$, the open segments joining $(2/7,1/7)$ to $(0,0)$ and
to $(1/3,1/3)$, dotted in Figure~\ref{fig:mapping-region}. (iii) Restricted weak type at $(1/3,1/3)$: the cusp
construction in Section~\ref{sec:necessary-conditions} uses inputs that are not characteristic
functions, so it does not decide the restricted weak type estimate
$L^{3,1}\to L^{3,\infty}$.

\subsection{The translated case}\label{sec:translated-result}
We next consider axis dilations of a circle centered away from the origin.
Fix $c\in\R^2\setminus\{0\}$ and use the operators
$\A^c_{t,s}$ and $\M^c$ defined in Section~\ref{sec:operator}. Define
\[
 \Delta^\ast=\conv\left\{(0,0),\left(\frac14,\frac18\right),
                              \left(\frac13,\frac13\right)\right\},
\]
which does not depend on $c$.
Our second result is the analogue of Theorem~\ref{thm:main} for $\M^c$;
the range is again sharp up to the boundary.

\begin{theorem}\label{thm:translated}
If $(1/p,1/q)\in\operatorname{int}\Delta^\ast$, or if $3<p=q\le\infty$,
then
\begin{equation}\label{eq:cor-bound}
 \|\M^cf\|_{L^q(\R^2)}\le C_{c,p,q}\|f\|_{L^p(\R^2)}.
\end{equation}
Conversely, if \eqref{eq:cor-bound} holds for some $1\le p,q\le\infty$, then
$(1/p,1/q)\in\Delta^\ast$ and $p>3$.
\end{theorem}

Since $(1/4,1/8)$ lies on the edge of $\Delta$ joining $(0,0)$ and
$(2/7,1/7)$, and the line $5/p-2/q=1$ passes through $(1/4,1/8)$ and
$(1/3,1/3)$, we have
\[
 \Delta^\ast=\Delta\cap\bigl\{(1/p,1/q):5/p-2/q\le1\bigr\}.
\]
The additional condition comes from higher-order contact between the
dilated translated circles and short parabolic arcs, which the
construction in Section~\ref{sec:translated-necessity} detects. On the
sufficiency side, the same feature produces a larger Sobolev loss when
the supremum over $(t,s)$ is taken. See Section~\ref{sec:strategy} for
both points.

The condition $p>3$ cannot be relaxed even to weak type. The cusp
construction in Section~\ref{sec:translated-necessity} shows that, for
$1\le p\le3$, the estimate
\[
 \|\M^c f\|_{L^{q,\infty}(\R^2)}\lesssim\|f\|_{L^p(\R^2)}
\]
fails for every $1\le q\le\infty$, even on nonnegative smooth compactly
supported inputs. The other necessary conditions are also proved in
Section~\ref{sec:translated-necessity}.

Theorem~\ref{thm:translated} is proved in Section~\ref{sec:translated-case}.
It rests on the frequency-localized estimate of
Proposition~\ref{prop:translated}: for every $A>0$, $\varepsilon>0$,
and dyadic $\lambda\ge2$,
\[
 \|\M^c_\lambda f\|_8
 \lesssim_{A,\varepsilon}\lambda^\varepsilon\|f\|_4,
 \qquad |c|\le A.
\]
Interpolation with the diagonal decay of
Lemma~\ref{lem:translated-diagonal} gives the open triangle, and the
unit mass of the averages gives the $L^\infty$ bound.
As in the centered case, the following remain open: restricted weak
type and strong type at the vertex $(1/4,1/8)$, strong type on the open
segments joining $(1/4,1/8)$ to $(0,0)$ and to $(1/3,1/3)$, and
restricted weak type at $(1/3,1/3)$.

\subsection{Strategy of the proof}\label{sec:strategy}
The proofs of Theorems~\ref{thm:main} and~\ref{thm:translated} follow
the same two reductions. First, interpolating the vertex frequency
estimate with diagonal frequency decay gives summable decay in the open
triangle: Theorem~\ref{thm:frequency} and Lemma~\ref{lem:diagonal} in
Section~\ref{sec:frequency-reduction}, and
Proposition~\ref{prop:translated} and Lemma~\ref{lem:translated-diagonal}
in Sections~\ref{sec:translated-sufficiency} and~\ref{sec:translated-diagonal}, respectively. Second, an angular
decomposition based on the pointwise inequality of
Lemma~\ref{lem:brd-nrw}, a two-term variant of Ham--Lee~\cite[Lemma 2.8]{HL}, and C\'ordoba's square-function estimate in
Lemma~\ref{lem:cordoba} reduces the vertex estimates to separated-support
and short-interval estimates: Propositions~\ref{prop:separated}
and~\ref{prop:short} in the centered case, and
\eqref{eq:translated-separated} and \eqref{eq:translated-short} in the
translated case.

\subsubsection*{The bilinear estimate near the axes}
At the vertex exponents, the loss from the suprema over the two
parameters is balanced by the gain in the bilinear estimate of
Proposition~\ref{prop:spacetime} for separated angular supports.
Bilinear estimates for separated frequency supports go back to~\cite{TVV,TV}; for the circular maximal function, Lee~\cite{Lee} used
the bilinear cone restriction estimates of Wolff~\cite{Wolff} and Tao~\cite{TaoCone} to obtain endpoint $L^p$--$L^q$ bounds. The main
difficulty is uniformity as the supports approach a coordinate axis.
Let $b$ be the angular distance of the two frequency sectors from the
axis and $\delta$ their angular separation. The mixed Hessian of the
bilinear phase degenerates there: its determinant is comparable to
$b\delta^4$. Output coordinates given by the coefficients of an exact
third-order expansion of the phase in the angle, suitably rescaled,
cancel this factor at the cost of the factor $b^{-1/2}\delta^{-2}$ in
\eqref{eq:spacetime-bound}. The factor $b$ reflects the angular
derivative of the radius of curvature of the ellipse, which vanishes on
the axes by reflection symmetry; see
Section~\ref{sec:input-jet-coordinates}.

\subsubsection*{Comparison with cinematic curves}
This degeneracy is also visible in physical space. The ellipses
$\{x+(t\cos\theta,s\sin\theta):\theta\in[0,2\pi)\}$,
$(x,t,s)\in\R^2\times\I^2$, form a four-parameter family that is not
cinematic in the sense of Zahl~\cite[Definitions 1.1 and 1.7]{Zahl} at
the axis vertices. For $1<t<\sqrt2$ and $s=t^2$, the upper graph
translated to the origin is
$h_t(u)=-u^2/2-u^4/(8t^2)+O(u^6)$, so
$(h_t,h'_t,h''_t,h_t^{(3)})(0)=(0,0,-1,0)$ along a nonconstant family.
The same holds for $\M^c$, whose averaging curves form the same
family, since the translation $(tc_1,sc_2)$ is absorbed into $x$. Thus~\cite[Theorem 1.8]{Zahl} does not apply
directly, and Proposition~\ref{prop:spacetime} is proved by a
Fourier-analytic argument. A normalization combining parabolic input
rescaling with a projective change of the output variables yields
uniformly cinematic families on caps approaching the axis vertices;
this is developed in forthcoming work~\cite{LLLOcin}. See
Section~\ref{sec:related}.

\subsubsection*{Differences in the translated case}
For $c\ne0$, the translation $(tc_1,sc_2)$ preserves the bilinear and
kernel estimates, but the supremum over $(t,s)$ incurs a larger Sobolev
loss (Section~\ref{sec:translated-sufficiency}); this moves the vertex
to $(1/4,1/8)$. Correspondingly, the test function $f_h$ of
Section~\ref{sec:translated-necessity}, supported in an
$O(h^4)$-neighborhood of a parabolic arc of length $O(h)$, satisfies
$\M^cf_h\gtrsim h$ on a set of area comparable to $h^2$, with constants
depending on $c$; this construction degenerates when $c=0$, since the
Jacobian $c_1(c_2+\varepsilon_0)t^2$ of its output map, with
$\varepsilon_0=\pm1$, then vanishes. The diagonal
estimate for each fixed $c$, including $c=0$, is proved in
Section~\ref{sec:translated-diagonal-estimates} by verifying the
hypotheses of the incidence estimate of Pramanik--Yang--Zahl~\cite{PYZ}
for the translated ellipses, with an additional rescaling near finitely
many degenerate points when $|c|\ge1$; local smoothing then gives
Lemma~\ref{lem:translated-diagonal}.

\subsection{Related results and extensions}\label{sec:related}
For the elliptic maximal function, in which rotations of the ellipses
are also allowed, Erdo\u gan~\cite{Erdogan} proved Sobolev bounds and
showed that no $L^p$ bound holds for $p\le4$, and Lee--Lee--Oh~\cite[Theorem 1.1]{LLOelliptic} proved $L^p$ boundedness for $p>12$.
Maximal functions with several independent parameters have been studied
for averages over surfaces and curves; see, for example,
\cite{Cho,Heo,LLtori,MR,MRZ}. For the analogous strong spherical
maximal functions in higher dimensions, see~\cite{Chen,HZ,LLOspherical},
and~\cite{LLOellr} for $\ell^r$ spheres, where $L^p$--$L^q$ bounds for
the local operator are also obtained.

Our argument is two-dimensional and does not determine the sharp
$L^p$--$L^q$ range in higher dimensions. Lee--Lee--Oh~\cite{LLOellr} obtain
partial bounds in higher dimensions using $L^2$--$L^q$ decay estimates.
For Euclidean spheres, their results are sharp up to endpoints for
$p\ge2$ and $d\ge4$.
These results do not determine the sharp range for translated spheres.

For axis dilations of general curves, the degeneracy set of the bilinear
mixed Hessian presents an additional difficulty. Even a small smooth
perturbation of the circle can move this set away from the coordinate
axes. Extending Theorem~\ref{thm:main} by the present method would therefore require
a decomposition adapted to the perturbed degeneracy set and uniform
estimates replacing those in Section~\ref{sec:mixed-hessian}. We do not pursue such
extensions here.

In forthcoming work~\cite{LLLOcin}, we build on Zahl's
incidence-geometric framework~\cite{Zahl} to obtain $L^p$-improving
bounds for multiparameter cinematic maximal functions with
two-dimensional output, under suitable jet nondegeneracy and output
transversality hypotheses. For full parameter patches with
nonvanishing averaging density, the resulting range is sharp up to
the boundary.

\subsubsection*{Organization}
Sections~\ref{sec:frequency-reduction}--\ref{sec:maximal-estimates}
prove the centered estimates, and Section~\ref{sec:necessary-conditions}
establishes the necessary conditions.
Section~\ref{sec:translated-diagonal-estimates} proves the diagonal
estimates for fixed $c$, and Section~\ref{sec:translated-case}
completes the proof of Theorem~\ref{thm:translated}.
Appendix~\ref{app:selected-curves} contains the proof of the
selected-curves estimate and the regularized bound of
Corollary~\ref{cor:translated-regularized}.

\subsubsection*{Notation}
We use $\widehat f(\xi)=\int e^{-ix\cdot\xi}f(x)\,dx$ and write
$\|f\|_r=\|f\|_{L^r(\R^2)}$ when the domain is unambiguous.
Throughout, $\lambda=2^j\ge2$ denotes the radial frequency, $b$ the
angular distance from a coordinate axis, and $\delta$ the separation
scale of two angular intervals (in Sections~\ref{sec:angular-decomposition}--\ref{sec:maximal-estimates} and \ref{sec:translated-sufficiency}); in Section~\ref{sec:translated-diagonal-estimates} and Appendix~\ref{app:selected-curves}, $\delta$ denotes a tube width. We write
$e(\theta)=(\cos\theta,\sin\theta)$. The semiaxes are always denoted
by $t,s$. We write $A\lesssim B$ when $A\le CB$, with $C$ independent
of the frequency and angular scales, and $A\sim B$ when both
inequalities hold. Dependence on the subdivision parameter $K$ of
Sections~\ref{sec:angular-decomposition} and \ref{sec:translated-sufficiency} is
indicated by $\lesssim_K$. All functions may initially be taken in the
Schwartz class.

\section{Frequency reduction in the centered case}\label{sec:frequency-reduction}
This section reduces the sufficiency part of Theorem~\ref{thm:main} to
Theorem~\ref{thm:frequency}: combined with the diagonal decay of
Lemma~\ref{lem:diagonal}, it yields every point of the open triangle.

Choose a smooth inhomogeneous radial partition of unity
\[
 1=\psi_0(\xi)+\sum_{j\ge1}\psi(2^{-j}|\xi|),
\]
where $\supp\psi\subset(1/2,2)$, and let $P_{\le1}$ and $P_\lambda$, $\lambda=2^j$, be the Fourier
multipliers with symbols $\psi_0$ and $\psi(|\xi|/\lambda)$. With these
projections, the frequency-localized operators of
Section~\ref{sec:main-result} are
\begin{equation}\label{eq:frequency-localization}
 \M^c_\lambda f(x)=\sup_{(t,s)\in\I^2}|\A^c_{t,s}P_\lambda f(x)|.
\end{equation}
We use \eqref{eq:frequency-localization} for every $c\in\R^2$; the
estimates in this section concern the centered case $c=0$.

\subsection{Diagonal frequency decay}\label{sec:diagonal-decay}
Besides Theorem~\ref{thm:frequency}, the proof uses decay on the
diagonal. This decay is known; see~\cite[(3.14)]{CGY}, where the
smoothing input is~\cite[Proposition~5.2]{CGY}. We include a short
derivation that combines an incidence bound at $L^3$ with the local
smoothing estimate of~\cite{LLOelliptic}.

\begin{lemma}[Diagonal frequency decay]\label{lem:diagonal}
For every $3<r<\infty$ there is $\kappa_r>0$ such that
\begin{equation}\label{eq:diagonal-decay}
 \|\M_\lambda f\|_r\lesssim_r\lambda^{-\kappa_r}\|f\|_r.
\end{equation}
\end{lemma}
\begin{proof}
The incidence estimate of Pramanik--Yang--Zahl~\cite[Theorem 1.7 and
(1.9)]{PYZ}, through the elliptic reduction in~\cite[Section 3.1 and (3.15)]{CGY},
gives
\[
 \|\M_\lambda f\|_3\lesssim_\varepsilon\lambda^\varepsilon\|f\|_3
 \qquad(\varepsilon>0).
\]
The estimate in~\cite{CGY,PYZ} concerns semiaxes near $(1,1)$. We cover
$\I^2$ by finitely many patches $(t,s)=(t_0t',s_0s')$ with $(t',s')$
near $(1,1)$. The dilation $\xi\mapsto(t_0\xi_1,s_0\xi_2)$ preserves comparable
frequency annuli and has bounded determinant, so a finite
Littlewood--Paley decomposition transfers the localized estimate to
$\M_\lambda$. On each patch, reflections and interchange of coordinates
reduce the ellipses to finitely many upper graph pieces. With
$a=t\cos\theta$, the angular measure on each piece has density
$da/(2\pi\sqrt{t^2-a^2})$, which is smooth and bounded away from the
vertical tangents; this puts each piece in the graph form~\cite[(3.1)]{CGY}.

For $12\le p_0<\infty$, the local smoothing estimate of Lee--Lee--Oh~\cite[Proposition 2.3]{LLOelliptic} (which rests on decoupling inequalities~\cite{BD,BDG}
and the variable-coefficient argument of~\cite{BHS}), combined with the parameter
Sobolev inequality of Lemma~\ref{lem:parameter-sobolev} with $L_1=L_2=\lambda$
and $r=p_0$ (cf.\ \cite[Lemma 2.1]{LLOelliptic}), gives
\[
 \|\M_\lambda f\|_{p_0}
 \lesssim_{p_0,\varepsilon}
 \lambda^{-1/2}\lambda^{2/p_0}
 \lambda^{1/2-3/p_0+\varepsilon}\|f\|_{p_0}
 =\lambda^{-1/p_0+\varepsilon}\|f\|_{p_0}.
\]
The three factors are the order of the averaging symbol, the cost of
the suprema over the two parameters, and the local smoothing bound, respectively.
The parameter Sobolev inequality also requires the corresponding spacetime
bound for $\lambda^{-\alpha-\beta}\partial_t^\alpha\partial_s^\beta
\A_{t,s}P_\lambda f$, with $\alpha,\beta\in\{0,1\}$. All four normalized
derivatives have branch amplitudes satisfying the same symbol bounds, so
the local smoothing estimate applies to each of them as in~\cite[Section 2]{LLOelliptic}. The estimate in~\cite{LLOelliptic} is local in
space. Since $\M_\lambda$ is translation invariant and the kernels of
$\A_{t,s}P_\lambda$, $(t,s)\in\I^2$, decay rapidly away from $B(0,2)$
uniformly in $(t,s)$, a standard localization argument extends it to
the global estimate above. For the prescribed $r>3$, choose
$p_0>\max\{12,r\}$ and $0<\vartheta<1$ such that
\[
 \frac1r=\frac{1-\vartheta}{3}+\frac\vartheta{p_0}.
\]
Interpolation, after linearizing the supremum as in
Section~\ref{sec:sufficiency}, yields
$\|\M_\lambda f\|_r\lesssim\lambda^{\varepsilon-\vartheta/p_0}\|f\|_r$.
Taking $\varepsilon=\vartheta/(2p_0)$ proves \eqref{eq:diagonal-decay}.
\end{proof}

\subsection{Proof of the sufficiency part of Theorem~\ref{thm:main}}\label{sec:sufficiency}
Assume Theorem~\ref{thm:frequency}. By a standard limiting argument (see, e.g., \cite{Stein}),
it suffices to prove the estimates for Schwartz functions $f$. Besides
\eqref{eq:frequency-bound} and \eqref{eq:diagonal-decay}, we use the trivial bound
$\|\M_\lambda f\|_\infty\lesssim\|f\|_\infty$, which follows from
the unit mass of $\sigma_{t,s}$ and the uniform $L^1$ bound for the
kernel of $P_\lambda$. These bounds may be interpolated. Indeed, for
Schwartz $f$ the supremum in \eqref{eq:frequency-localization} is the increasing limit of maxima
over finite subsets of a countable dense subset of $\I^2$. For each
such maximum, a measurable choice of $(t,s)=(t(x),s(x))$ linearizes
the operator, with the same three bounds uniformly in the choice.

Fix $3<r<\infty$, and let $\kappa_r>0$ be as in \eqref{eq:diagonal-decay}. For
$0<\vartheta<1$, define $p_\vartheta,q_\vartheta$ by
\[
 \left(\frac1{p_\vartheta},\frac1{q_\vartheta}\right)
 =(1-\vartheta)\left(\frac27,\frac17\right)
   +\vartheta\left(\frac1r,\frac1r\right).
\]
Interpolation between \eqref{eq:frequency-bound} and \eqref{eq:diagonal-decay} gives
\begin{equation}\label{eq:interpolated-decay}
 \|\M_\lambda f\|_{q_\vartheta}
 \lesssim\lambda^{(1-\vartheta)\varepsilon-\vartheta \kappa_r}
 \|f\|_{p_\vartheta}.
\end{equation}
Taking $\varepsilon<\vartheta \kappa_r/(1-\vartheta)$ makes the exponent
of $\lambda$ in \eqref{eq:interpolated-decay} negative on the open segment joining
$(2/7,1/7)$ and $(1/r,1/r)$. Let
\[
 \Delta_r=\conv\left\{(0,0),\left(\frac27,\frac17\right),
                              \left(\frac1r,\frac1r\right)\right\}.
\]
A further interpolation with the $L^\infty$ bound shows that for every
$(1/p,1/q)\in\operatorname{int}\Delta_r$ there is $\eta_{p,q}>0$ with
\begin{equation}\label{eq:interior-decay}
 \|\M_\lambda f\|_q\lesssim_{p,q}\lambda^{-\eta_{p,q}}\|f\|_p.
\end{equation}
The triangles $\Delta_r$ increase to $\Delta$ as $r\downarrow3$, so
\eqref{eq:interior-decay} holds for every $(1/p,1/q)\in\operatorname{int}\Delta$.

The multipliers $\widehat\sigma_{t,s}(\xi)\psi_0(\xi)$ have uniformly
bounded smooth derivatives and compact support. Their kernels have a
common Schwartz majorant $H$, so $\M P_{\le1}f\le H*|f|$ and Young's
inequality applies since $p\le q$. Sublinearity and \eqref{eq:interior-decay} now yield
\[
 \|\M f\|_q\le\|\M P_{\le1}f\|_q
   +\sum_{j\ge1}\|\M P_{2^j}f\|_q
 \lesssim\Big(1+\sum_{j\ge1}2^{-j\eta_{p,q}}\Big)\|f\|_p.
\]
For finite $p=q>3$, sum the decay in Lemma~\ref{lem:diagonal} directly, with the
same low-frequency bound. The case $(p,q)=(\infty,\infty)$ follows
from the unit mass of $\sigma_{t,s}$. This proves the sufficiency in
Theorem~\ref{thm:main}, subject to Theorem~\ref{thm:frequency}.

\section{Angular decomposition and proof of Theorem~\ref{thm:frequency}}\label{sec:angular-decomposition}

We prove \eqref{eq:frequency-bound} by repeatedly subdividing the frequency support into angular
intervals. At each step, the pointwise inequality of Lemma~\ref{lem:brd-nrw} bounds the
contribution of an interval by those of its subintervals and of pairs of
separated subintervals. Separated pairs are handled by the bilinear estimate
of Proposition~\ref{prop:separated}. The iteration stops at width $\lambda^{-1/2}$, where the
linear estimate of Proposition~\ref{prop:short} applies, and the square-function bound of
Lemma~\ref{lem:cordoba} sums the pieces. After writing the averages as oscillatory
integrals in Section~\ref{sec:fourier-representation}, we state these estimates.
Section~\ref{sec:translated-sufficiency} uses the same iteration with
exponents $4,8$ in place of $7/2,7$.

\subsection{Fourier representation of the averages}\label{sec:fourier-representation}
For $f_\lambda=P_\lambda f$, we write each average $\mathcal A_{t,s}f_\lambda$ as the sum of two
exact oscillatory branches. Keeping the full branch symbols will allow us
to apply the estimates in Section~\ref{sec:maximal-estimates} directly to the averages. Write
$\xi_{t,s}=(t\xi_1,s\xi_2)$ and
\[
  \Phi_{t,s}(\xi)=|\xi_{t,s}|.
\]
Since $\widehat{\sigma_{t,s}}(\xi)=\widehat\sigma(\xi_{t,s})$, Fourier
inversion gives
\[
  \mathcal A_{t,s}f(x)
  =(2\pi)^{-2}\int e^{ix\cdot\xi}\widehat\sigma(\xi_{t,s})
     \widehat f(\xi)\,d\xi.
\]
Here $\widehat\sigma(\xi)=J_0(|\xi|)$. We use the exact symbol form
\begin{equation}
  \widehat\sigma(\xi)
  =|\xi|^{-1/2}\bigl\{e^{i|\xi|}c_+(|\xi|)
                      +e^{-i|\xi|}c_-(|\xi|)\bigr\},
  \label{eq:bessel-expansion}
\end{equation}
where
\[
  |\partial_r^k c_\pm(r)|\lesssim_k r^{-k}
  \qquad (r\geq 1).
\]
Indeed, let $\mathcal H^{(1)}$ and $\mathcal H^{(2)}$ be the Hankel
functions of order zero, and take
$c_+(r)=\frac12 r^{1/2}e^{-ir}\mathcal H^{(1)}(r)$ and
$c_-(r)=\frac12 r^{1/2}e^{ir}\mathcal H^{(2)}(r)$. The equality in \eqref{eq:bessel-expansion}
is the identity expressing the Bessel function of order zero as half the sum
of the two Hankel functions~\cite[(10.4.4)]{DLMF}. By the large-argument expansions~\cite[Section 10.17]{DLMF}, $c_\pm(z)$ is holomorphic and uniformly bounded on
the discs $|z-r|\leq r/2$ for large positive $r$, which lie in the sector
$|\arg z|\le\pi/6$.
Cauchy's estimate gives the displayed derivative bounds, and smoothness
covers the remaining compact range of $r$.

Choose once and for all
$\widetilde\psi\in C_c^\infty((1/4,4))$, real-valued and equal to one
on $\supp\psi$. Thus $\widetilde\psi(|\xi|/\lambda)$ equals one on the
Fourier support of $f_\lambda$, with fixed compact annular support in the
normalized radial variable. For $\lambda\geq4$, so that
$\Phi_{t,s}(\xi)\geq1$ on the support of
$\widetilde\psi(|\xi|/\lambda)$, define
\begin{equation}
  \mathcal A^\pm_\lambda g(x,t,s)
  =\frac{\lambda^{-1/2}}{(2\pi)^2}
    \int e^{i(x\cdot\xi\pm\Phi_{t,s}(\xi))}
      m^\pm_\lambda(t,s,\xi)\widehat g(\xi)\,d\xi,
  \label{eq:branch-operator}
\end{equation}
where
\begin{equation}
  m^\pm_\lambda(t,s,\xi)
  =\widetilde\psi(|\xi|/\lambda)
    \left(\frac{\lambda}{\Phi_{t,s}(\xi)}\right)^{1/2}
    c_\pm(\Phi_{t,s}(\xi)).
  \label{eq:branch-symbol}
\end{equation}
Then
\begin{equation}
  \mathcal A_{t,s}f_\lambda
  =\mathcal A^+_\lambda f_\lambda(\cdot,t,s)
   +\mathcal A^-_\lambda f_\lambda(\cdot,t,s).
  \label{eq:branch-decomposition}
\end{equation}
Set
\[
  \mathcal M^+_\lambda g
  =\sup_{(t,s)\in\mathbb I^2}|\mathcal A^+_\lambda g|.
\]
The signs in the superscripts label the phase branches; the centered
superscript $0$ is again omitted.
Since $c_-=\overline{c_+}$ on $(0,\infty)$ and $\Phi_{t,s}$ is even,
$\overline{\mathcal A^-_\lambda g}
 =\mathcal A^+_\lambda\overline g$; hence it suffices to prove the bounds
for the positive sign, as in~\cite[Section 2]{LLOelliptic}. For $\lambda=2$, the kernels of
$\mathcal A_{t,s}P_\lambda$ have a common integrable majorant, so this
term is bounded trivially.

\subsection{Angular localization and the main estimates}\label{sec:angular-localization}
We now introduce the angular decomposition and state the estimates for
separated intervals and for intervals of width $\lambda^{-1/2}$. The
square-function estimate at the end of this subsection will control the
input norms when these estimates are combined. For an interval $I$ of
directions, define
\begin{equation}
  f_I=Q_I f_\lambda,
  \label{eq:angular-projection}
\end{equation}
where
\[
  \widehat{Q_I g}(\xi)=\mathbf1_I(\arg\xi)\widehat g(\xi).
\]
For Schwartz $f$, the Fourier transform of $f_I$ is compactly supported
and integrable, although $f_I$ need not be Schwartz. Hence its branch
integrals are well defined and continuous in $(t,s)$. The sharp projections
make the decomposition exact: if $I$ is partitioned into consecutive
intervals $I^{(1)},\ldots,I^{(K)}$, then $f_I=\sum_{\ell=1}^K f_{I^{(\ell)}}$ and
\begin{equation}
  \mathcal A^+_\lambda f_I
  =\sum_{\ell=1}^K\mathcal A^+_\lambda f_{I^{(\ell)}}.
  \label{eq:angular-additivity}
\end{equation}
When estimating one term, we insert a smooth angular cutoff into the
operator. For an interval $I$, let $\chi_I$ be a smooth $2\pi$-periodic
function of the angle with $\chi_I=1$ on $I$,
\[
  \supp\chi_I\subset\{\theta:\operatorname{dist}(\theta,I)\leq|I|/10\},
\]
and $\|\partial_\theta^k\chi_I\|_\infty\lesssim_k |I|^{-k}$ for $k\geq0$, uniformly in $I$. Let $\mathcal A^+_{\lambda,I}$ be defined by \eqref{eq:branch-operator}
with $m^+_\lambda(t,s,\xi)$ replaced by
$m^+_\lambda(t,s,\xi)\chi_I(\arg\xi)$. Since $\chi_I=1$ on $I$, we have
$\mathcal A^+_{\lambda,I}f_I=\mathcal A^+_\lambda f_I$, so the
identity \eqref{eq:angular-additivity} is unchanged.

By coordinate reflections and interchange of $t,s$, it suffices to consider
$f$ whose Fourier transform is supported in the first octant
$\{\xi:0\leq\arg\xi\leq\pi/4\}$. Decompose $(0,\pi/4]$ into
\[
  I_b=[b,2b],
\]
where $b=2^{-m}\pi/8$, $m\geq0$.
For intervals $I$ contained in $I_b$, the supports of $\chi_I$ remain at
distance comparable to $b$ from the axes. For intervals $I,I'\subset I_b$, define
the bilinear maximal operator
\begin{equation}
  \mathcal B_{\lambda,I,I'}(g_1,g_2)(x)
  =\sup_{(t,s)\in\mathbb I^2}
    |\mathcal A^+_{\lambda,I}g_1(x,t,s)
      \mathcal A^+_{\lambda,I'}g_2(x,t,s)|.
  \label{eq:bilinear-maximal}
\end{equation}
The same pair $(t,s)$ is used in both factors.

\begin{proposition}[Separated angular supports]\label{prop:separated}
Let $\lambda^{-1/2}\lesssim\delta\lesssim b\lesssim1$.
Suppose $I,I'\subset I_b$ satisfy
\begin{equation}
  |I|\sim|I'|\sim\operatorname{dist}(I,I')\sim\delta,
  \label{eq:angular-separation}
\end{equation}
and that the supports of $\chi_I$ and $\chi_{I'}$ are at distance
comparable to $\delta$. Then
\begin{equation}
  \|\mathcal B_{\lambda,I,I'}(g_1,g_2)\|_{7/2}
  \lesssim\|g_1\|_{7/2}\|g_2\|_{7/2},
  \label{eq:separated-maximal}
\end{equation}
with a constant independent of $\lambda,b,\delta$.
\end{proposition}

\begin{proposition}[Short angular supports]\label{prop:short}
Fix $C_0\geq1$. If $I\subset I_b$ and
$\lambda^{-1/2}\leq|I|\leq C_0\lambda^{-1/2}$, then
\begin{equation}
  \Big\|\sup_{(t,s)\in\mathbb I^2}|\mathcal A^+_{\lambda,I}g|
  \Big\|_7
  \lesssim_{C_0}\lambda^{-1/14}b^{1/7}\|g\|_{7/2}.
  \label{eq:short-cap-maximal}
\end{equation}
If instead $I\subset[0,C_0\lambda^{-1/2}]$, then
\begin{equation}
  \Big\|\sup_{(t,s)\in\mathbb I^2}|\mathcal A^+_{\lambda,I}g|
  \Big\|_7
  \lesssim_{C_0}\lambda^{-1/7}\|g\|_{7/2}.
  \label{eq:axis-cap-maximal}
\end{equation}
\end{proposition}

Proposition~\ref{prop:separated} follows from the spacetime estimate of Proposition~\ref{prop:spacetime},
the parameter Sobolev inequality (Lemma~\ref{lem:parameter-sobolev}), and the kernel
bounds of Lemma~\ref{lem:kernel}. Proposition~\ref{prop:short} uses only the last two. Both are
proved in Section~\ref{sec:maximal-estimates}.

Finally, we record the square-function estimate used to sum the inputs.

\begin{lemma}[C\'ordoba]\label{lem:cordoba}
If $\mathcal I$ is a subcollection of an equiangular partition of the circle
into $N$ sectors, then, for some fixed $\gamma>0$ and every $2\le r\le4$,
\begin{equation}
  \Big\|\Bigl(\sum_{I\in\mathcal I}|Q_I g|^2\Bigr)^{1/2}\Big\|_r
  \lesssim_r (\log(2+N))^\gamma\|g\|_r.
  \label{eq:cordoba-square-function}
\end{equation}
The estimate is uniform under rotations of the partition.
\end{lemma}

\begin{proof}
For the full partition this is~\cite[Theorem 1(ii)]{Cordoba} at exponent $4$; as
noted in~\cite[p.~217]{Cordoba}, interpolation with the trivial $L^2$ bound gives
the same estimate for $2\leq r\leq4$.
Passing to a subcollection decreases the square function pointwise, and
rotations preserve the norms. We use this estimate only for input Fourier
projections.
\end{proof}

\subsection{A pointwise inequality}\label{sec:pointwise-inequality}
The following inequality gives the recursive step: it bounds an angular
sum by its largest piece and by products of nonadjacent pieces. It is a
two-term variant of~\cite[Lemma 2.8]{HL}.

\begin{lemma}\label{lem:brd-nrw}
Let $K\geq1$ and let $z_1,\ldots,z_K\in\mathbb C$. Then, with the
convention that a maximum over the empty set is zero,
\begin{equation}
  \Big|\sum_{\ell=1}^K z_\ell\Big|
  \leq3\max_\ell|z_\ell|
     +K\max_{|\ell-\ell'|\geq2}|z_\ell z_{\ell'}|^{1/2}.
  \label{eq:broad-narrow}
\end{equation}
\end{lemma}

\begin{proof}
Choose $\ell_0$ so that $|z_{\ell_0}|$ is maximal. The terms with
$|\ell-\ell_0|\leq1$ have sum bounded in modulus by $3|z_{\ell_0}|$.
For every other index, $|z_\ell|\leq|z_\ell z_{\ell_0}|^{1/2}$.
Adding these estimates proves \eqref{eq:broad-narrow}.
\end{proof}

When the lemma is applied to the $K$ consecutive subintervals of an
interval, the condition $|\ell-\ell'|\geq2$ means that the two
subintervals are separated by at least one subinterval length. This is the
separation condition in the definition of $\mathcal P_j$ in Section~\ref{sec:frequency-proof}.

\subsection{Proof of Theorem~\ref{thm:frequency}}\label{sec:frequency-proof}
We combine Propositions~\ref{prop:separated} and~\ref{prop:short} by iterating Lemma~\ref{lem:brd-nrw} on each
angular shell. Lemma~\ref{lem:cordoba} will sum the input norms at every stage. Fix
$\varepsilon>0$ and choose an integer $K\geq4$ so large that
\[
  \frac{\log3}{2\log K}<\frac{\varepsilon}{4}.
\]
This choice remains fixed as $\lambda$ and the angular scales vary.

\subsubsection*{The angular shells}
Recall that $\widehat f$ is supported in the first octant. Since the
finitely many bounded values of $\lambda$ are covered by the trivial
estimate, we may assume $\lambda^{-1/2}\leq\pi/8$. Put
\[
  b_m=2^{-m}\frac\pi8\quad(m\geq0),
\]
and choose the unique integer $M=M(\lambda)\geq0$ such that
\begin{equation}
  b_M\geq\lambda^{-1/2}>b_{M+1}.
  \label{eq:terminal-shell-scale}
\end{equation}
Define
\begin{equation}
  I_{b_m}=[b_m,2b_m]\quad(0\leq m\leq M),
  \label{eq:angular-shells}
\end{equation}
and $I_{\mathrm{ax}}=[0,b_M]$.
Thus $|I_{\mathrm{ax}}|\sim\lambda^{-1/2}$ and, up to endpoints,
\[
  [0,\pi/4]=I_{\mathrm{ax}}\cup\bigcup_{m=0}^M I_{b_m}.
\]
With the notation $f_I=Q_I f_\lambda$ of \eqref{eq:angular-projection}, the angular decomposition
is exact:
\begin{equation}
  f_\lambda=f_{I_{\mathrm{ax}}}+\sum_{m=0}^M f_{I_{b_m}}.
  \label{eq:shell-decomposition}
\end{equation}
The axis term will be estimated directly by \eqref{eq:axis-cap-maximal}.

\subsubsection*{Multiscale decomposition of a shell}
Fix $0\leq m\leq M$ and write $b=b_m$. We first prove the shell estimate
\begin{equation}
  \|\mathcal M^+_\lambda f_{I_b}\|_7
  \lesssim_K 3^L(1+\log\lambda)^\gamma\|f_{I_b}\|_{7/2},
  \label{eq:shell-bound}
\end{equation}
where $L$ is the number of subdivision steps defined below. We obtain it
by subdividing $I_b$ repeatedly and applying Lemma~\ref{lem:brd-nrw} at each scale.

For $I\subset I_b$, the projections in \eqref{eq:angular-projection} satisfy
$f_I=Q_I f_{I_b}=Q_I f_\lambda$. Subdivide $I_b$ repeatedly into $K$
equal pieces, writing
\begin{equation}
  \mathcal I_j=\{[b+\ell h_j,b+(\ell+1)h_j):0\leq\ell<K^j\},
  \label{eq:shell-partition}
\end{equation}
where $h_j=bK^{-j}$.
Let $L=L(\lambda,b)$ be the largest integer such that
$h_L\geq\lambda^{-1/2}$. Then
\begin{equation}
  \lambda^{-1/2}\leq h_L<K\lambda^{-1/2},
  \label{eq:terminal-scale}
\end{equation}
and hence $K^L=b/h_L\leq b\lambda^{1/2}\leq\lambda^{1/2}$.
At each scale $1\leq j\leq L$, let $\mathcal P_j$ be the set of pairs
$(I,I')\in\mathcal I_j^2$ lying in the same interval of
$\mathcal I_{j-1}$ and satisfying $\operatorname{dist}(I,I')\geq h_j$.
Thus each eligible pair consists of two nonadjacent subintervals of one
interval of length $h_{j-1}$.

Set
\begin{equation}\label{eq:iteration-maxima}
\begin{aligned}
  F_j(x)&=\max_{I\in\mathcal I_j}\mathcal M^+_\lambda f_I(x),\\
  G_j(x)&=\max_{(I,I')\in\mathcal P_j}\sup_{(t,s)\in\mathbb I^2}
    |\mathcal A^+_\lambda f_I(x,t,s)
      \mathcal A^+_\lambda f_{I'}(x,t,s)|^{1/2}.
\end{aligned}
\end{equation}
Since $\mathcal I_0=\{I_b\}$, we have
$F_0=\mathcal M^+_\lambda f_{I_b}$. Apply Lemma~\ref{lem:brd-nrw} at a fixed
$(x,t,s)$ to the exact sum \eqref{eq:angular-additivity} over the $K$ subintervals of each
interval of $\mathcal I_{j-1}$; take the common parameter supremum
afterwards. This gives
\[
  F_{j-1}\leq3F_j+KG_j,\qquad1\leq j\leq L.
\]
Iterating in $j$, we obtain
\begin{equation}
  F_0\leq3^L F_L+K\sum_{j=1}^L3^{j-1}G_j.
  \label{eq:iteration}
\end{equation}
Thus the recursion leaves the terminal intervals at scale $h_L$ and,
from each preceding scale, one contribution of separated pairs.

\subsubsection*{The separated pairs}
Fix $1\leq j\leq L$ and $(I,I')\in\mathcal P_j$, and let
$\delta=\operatorname{dist}(I,I')$. Since $I$ and $I'$ are nonadjacent
subintervals of one interval of $\mathcal I_{j-1}$, we have
$h_j\leq\delta\leq Kh_j$ and
$\lambda^{-1/2}\leq h_j\leq\delta\leq h_{j-1}\leq b$. The supports
of $\chi_I$ and $\chi_{I'}$ contain $I$ and $I'$ and lie within distance
$h_j/10$ of them. Hence their lengths lie between $h_j$ and $6h_j/5$,
their mutual distance is at least $\delta-h_j/5\geq4\delta/5$, and they
lie at distance comparable to $b$ from the axes. Thus the hypotheses of
Proposition~\ref{prop:separated} hold, with constants depending only on $K$.

Since $\mathcal A^+_{\lambda,I}f_I=\mathcal A^+_\lambda f_I$ (Section~\ref{sec:angular-localization}), bounding the maximum by the sum and applying Proposition~\ref{prop:separated}, we obtain
\begin{equation}
  \|G_j\|_7^7
  \leq\sum_{(I,I')\in\mathcal P_j}
    \|\mathcal B_{\lambda,I,I'}(f_I,f_{I'})\|_{7/2}^{7/2}
  \lesssim_K\Bigl(\sum_{I\in\mathcal I_j}\|f_I\|_{7/2}^{7/2}\Bigr)^2.
  \label{eq:separated-pair-sum}
\end{equation}
For every $p\geq2$,
\begin{equation}
  \Bigl(\sum_I\|f_I\|_p^p\Bigr)^{1/p}
  \leq\Big\|\Bigl(\sum_I|f_I|^2\Bigr)^{1/2}\Big\|_p.
  \label{eq:input-sum}
\end{equation}
Taking the seventh root in \eqref{eq:separated-pair-sum} gives the exponent $2/7$, so \eqref{eq:input-sum}
gives
\[
  \|G_j\|_7
  \lesssim_K\Bigl(\sum_{I\in\mathcal I_j}\|f_I\|_{7/2}^{7/2}\Bigr)^{2/7}
  \leq\Big\|\Bigl(\sum_{I\in\mathcal I_j}|f_I|^2\Bigr)^{1/2}
       \Big\|_{7/2}.
\]
Each partition \eqref{eq:shell-partition} embeds in a complete equiangular partition of the
circle: $2\pi/h_j=16\cdot2^m K^j$ is an integer. Moreover
$h_j\geq\lambda^{-1/2}$, so this partition has $O(\lambda^{1/2})$
intervals, and Lemma~\ref{lem:cordoba} applies with $N\lesssim\lambda^{1/2}$. Since
every $f_I=Q_I f_{I_b}$, it yields
\begin{equation}
  \|G_j\|_7\lesssim_K(1+\log\lambda)^\gamma\|f_{I_b}\|_{7/2}.
  \label{eq:separated-pair-bound}
\end{equation}
At every scale the square-function estimate is applied to the same input
$f_{I_b}$, so its constants are not multiplied along the recursion.

\subsubsection*{The terminal intervals}
For $F_L$, the same identity, Proposition~\ref{prop:short}, and $h_L<K\lambda^{-1/2}$ give
\[
  \|F_L\|_7
  \lesssim_K\Bigl(\sum_{I\in\mathcal I_L}\|f_I\|_{7/2}^{7}\Bigr)^{1/7}
  \leq\Bigl(\sum_{I\in\mathcal I_L}\|f_I\|_{7/2}^{7/2}\Bigr)^{2/7}.
\]
The first inequality uses $\lambda^{-1/14}b^{1/7}\leq1$. Using \eqref{eq:input-sum}
with $p=7/2$ and Lemma~\ref{lem:cordoba} on $f_{I_b}$, as above, we obtain
\begin{equation}
  \|F_L\|_7\lesssim_K(1+\log\lambda)^\gamma\|f_{I_b}\|_{7/2}.
  \label{eq:terminal-bound}
\end{equation}
If $L=0$, the whole shell is already a terminal interval and the same
estimate applies directly. Combining \eqref{eq:iteration}, \eqref{eq:separated-pair-bound}, and \eqref{eq:terminal-bound}, and using
$\sum_{j=1}^L3^{j-1}=(3^L-1)/2$, we obtain \eqref{eq:shell-bound}.

\subsubsection*{Summing the angular pieces}
By the bound $K^L\le\lambda^{1/2}$ following \eqref{eq:terminal-scale} and our choice of $K$, we have
$3^L=(K^L)^{\log3/\log K}\leq\lambda^{\log3/(2\log K)}
<\lambda^{\varepsilon/4}$. Absorb $(1+\log\lambda)^\gamma$ into a
further $\lambda^{\varepsilon/4}$. Since $K$ depends only on
$\varepsilon$, \eqref{eq:shell-bound} shows that each shell in \eqref{eq:angular-shells} satisfies
\begin{equation}
  \|\mathcal M^+_\lambda f_{I_b}\|_7
  \lesssim_\varepsilon\lambda^{\varepsilon/2}\|f_{I_b}\|_{7/2},
  \qquad b=b_m,\ 0\leq m\leq M.
  \label{eq:shell-epsilon-bound}
\end{equation}
For the remaining interval $I_{\mathrm{ax}}$, \eqref{eq:terminal-shell-scale} gives
$|I_{\mathrm{ax}}|=b_M<2\lambda^{-1/2}$, so Proposition~\ref{prop:short} with
$C_0=2$ yields
\begin{equation}
  \|\mathcal M^+_\lambda f_{I_{\mathrm{ax}}}\|_7
  \lesssim\lambda^{-1/7}\|f_{I_{\mathrm{ax}}}\|_{7/2}.
  \label{eq:axis-bound}
\end{equation}
Moreover, a sector of aperture less than $\pi$ is the intersection
of two half-planes. Its sharp Fourier projection is therefore the
composition of two half-plane Fourier projections, each uniformly
bounded on $L^p$, $1<p<\infty$. Hence
\begin{equation}
  \|f_{I_{\mathrm{ax}}}\|_{7/2}
  +\sup_{0\leq m\leq M}\|f_{I_{b_m}}\|_{7/2}
  \lesssim\|f_\lambda\|_{7/2}.
  \label{eq:shell-projection-bound}
\end{equation}
Using the exact decomposition \eqref{eq:shell-decomposition}, sublinearity, \eqref{eq:shell-epsilon-bound}, and \eqref{eq:axis-bound},
we get
\begin{align*}
  \|\mathcal M^+_\lambda f_\lambda\|_7
  &\leq\|\mathcal M^+_\lambda f_{I_{\mathrm{ax}}}\|_7
       +\sum_{m=0}^M\|\mathcal M^+_\lambda f_{I_{b_m}}\|_7\\
  &\lesssim_\varepsilon
    \bigl(\lambda^{-1/7}+(M+1)\lambda^{\varepsilon/2}\bigr)
    \|f_\lambda\|_{7/2}.
\end{align*}
Since $M+1\lesssim1+\log\lambda$, we have
\begin{equation}
  \|\mathcal M^+_\lambda f_\lambda\|_7
  \lesssim_\varepsilon\lambda^\varepsilon\|f_\lambda\|_{7/2}.
  \label{eq:positive-branch-bound}
\end{equation}
By \eqref{eq:branch-decomposition} and the identity
$\overline{\mathcal A^-_\lambda g}=\mathcal A^+_\lambda\overline g$,
we have $\mathcal M_\lambda f\leq\mathcal M^+_\lambda f_\lambda
+\mathcal M^+_\lambda\overline{f_\lambda}$. Apply \eqref{eq:positive-branch-bound}, after
reflections, to the octant pieces of $f_\lambda$ and
$\overline{f_\lambda}$. The corresponding projections are uniformly
bounded on $L^{7/2}$ by the argument for \eqref{eq:shell-projection-bound}; together with
$\|f_\lambda\|_{7/2}\lesssim\|f\|_{7/2}$, this gives \eqref{eq:frequency-bound}. This proves
Theorem~\ref{thm:frequency}.\qed

Section~\ref{sec:bilinear-spacetime} proves the bilinear estimate of
Proposition~\ref{prop:spacetime}, and Section~\ref{sec:maximal-estimates} deduces
Propositions~\ref{prop:separated} and~\ref{prop:short}.

\section{A bilinear \texorpdfstring{$L^2$}{L2} spacetime estimate}\label{sec:bilinear-spacetime}

The spacetime bound proved in this section supplies the $L^2$ input for the separated-support estimates in both the centered and translated cases. It is applied in Proposition~\ref{prop:separated} and, after parameter differentiation and a shear in the spatial variable, in Section~\ref{sec:translated-sufficiency}.

The local coordinates used below are related to the reductions in~\cite[Section 4 and Appendix A]{LLOelliptic}. The analytic input here is the uniform H\"ormander estimate proved below, and the bilinear estimate is proved independently of the local smoothing theorem of~\cite{LLOelliptic}.

Recall that $I_b=[b,2b]\subset(0,\pi/4]$, so that $b\le\pi/8$. Throughout this section we fix two angular intervals contained in $[9b/10,21b/10]\subset(0,21\pi/80]$, each of length $O(\delta)$ and at mutual distance comparable to $\delta$, where $0<\delta\lesssim b$. In the application to Proposition~\ref{prop:separated}, these intervals contain the supports of the cutoffs $\chi_I$ and $\chi_{I'}$ of Section~\ref{sec:angular-localization}; since $\operatorname{supp}\chi_I$ lies within $|I|/10$ of $I\subset I_b$ and $|I|\le b$, these supports lie in $[9b/10,21b/10]$. Let $\theta_0\sim b$ be an angle between them. With $e(\theta)=(\cos\theta,\sin\theta)$ as in Section~\ref{sec:introduction}, put $e'(\theta)=(-\sin\theta,\cos\theta)$. The letters $\xi,\eta\in\R^2$ denote the two independent input frequencies. Since each belongs to an annulus of radius $\lambda$, set
\begin{equation}
(\xi,\eta)=\bigl(\lambda r_1e(\theta_0+\delta u),\lambda r_2e(\theta_0+\delta v)\bigr),
\label{eq:normalized-frequencies}
\end{equation}
where $\zeta=(r_1,u,r_2,v)$. Here $r_1,r_2$ are normalized radial frequencies and $u,v$ are normalized angles. On the normalized input support, $\zeta$ lies in the set
\begin{equation}
\{\zeta:\ c\le r_1,r_2\le C,\ |u|+|v|\le C,\ c\le |u-v|\le C\}.
\label{eq:normalized-support}
\end{equation}
The constants in this display are fixed independently of $\lambda,b,\delta$.

\subsubsection*{The bilinear operator.}
Let $m(t,s,\xi,\eta)$ be smooth and supported in a fixed compact subset of $(0,\infty)^2$ in $(t,s)$. On its support, $\arg\xi$ and $\arg\eta$ are required to lie in the first and second angular intervals fixed above, respectively. We also assume that its expression in the normalized input coordinates \eqref{eq:normalized-frequencies},
\begin{equation}
\widetilde m(t,s,\zeta)=m(t,s,\lambda r_1e(\theta_0+\delta u),\lambda r_2e(\theta_0+\delta v)),
\label{eq:normalized-symbol}
\end{equation}
is supported in the set \eqref{eq:normalized-support} and satisfies
\begin{equation}
\sup_{t,s,\zeta}\bigl|\partial_\zeta^\alpha\widetilde m(t,s,\zeta)\bigr|\le C_\alpha,\qquad |\alpha|\le6,
\label{eq:symbol-bounds}
\end{equation}
with constants independent of $\lambda,b,\delta$. Recall $\Phi_{t,s}$ from Section~\ref{sec:fourier-representation}. For such $m$, define
\begin{align}
\mathcal TF(x,t,s)&=\iint e^{i\Psi(x,t,s,\xi,\eta)}m(t,s,\xi,\eta)F(\xi,\eta)\,d\xi\,d\eta,
\notag\\
\Psi(x,t,s,\xi,\eta)&=x\cdot(\xi+\eta)+\Phi_{t,s}(\xi)+\Phi_{t,s}(\eta).
\label{eq:bilinear-phase}
\end{align}
The scales $\lambda,b,\delta$ enter only through the support of $m$ and through \eqref{eq:normalized-symbol}.

\begin{proposition}\label{prop:spacetime}
Let $m$ satisfy the assumptions above. Then
\begin{equation}
\|\mathcal TF\|_{L^2(\R^2\times(0,\infty)^2)}\lesssim b^{-1/2}\delta^{-2}\|F\|_{L^2(\R^4)}.
\label{eq:spacetime-bound}
\end{equation}
The constant depends only on the fixed supports, the constants in \eqref{eq:normalized-support}, the implicit constants in $\delta\lesssim b$ and $\theta_0\sim b$, and the constants $C_\alpha$, $|\alpha|\le6$, in \eqref{eq:symbol-bounds}.
\end{proposition}

Taking $F(\xi,\eta)=\widehat f(\xi)\widehat g(\eta)$ and using Plancherel gives the bilinear estimate used in Sections~\ref{sec:maximal-estimates} and \ref{sec:translated-sufficiency}; no smoothness of $F$ is needed.

\subsubsection*{Plan of the proof.}
The proof consists of three changes of variables followed by an application of a quantitative form of H\"ormander's $L^2$ estimate, Lemma~\ref{lem:uniform-hormander}. In Section~\ref{sec:input-jet-coordinates} the input frequencies are normalized to the fixed set \eqref{eq:normalized-support}, with Jacobian comparable to $\lambda^4\delta^2$. For the output we use the coefficients of an exact third-order expansion of the angular phase at $\theta_0$: the coordinates $z=(z_0,z_1)$ and the parameter coordinates $\varrho=(\varrho_1,\varrho_2)$. In Section~\ref{sec:phase-normalization} these coordinates are rescaled at the respective scales $\lambda,\lambda\delta,\lambda\delta^2,\lambda b\delta^3$, the sizes of the four terms at angular scale $\delta$. The output Jacobian is then comparable to $\lambda^4b\delta^6$, so the $L^2$ change of measure contributes
\[
\left(\frac{\lambda^4\delta^2}{\lambda^4b\delta^6}\right)^{1/2}=b^{-1/2}\delta^{-2},
\]
and it remains to bound the normalized operator uniformly. In Section~\ref{sec:mixed-hessian} its mixed Hessian is shown to be uniformly bounded and invertible: the original determinant vanishes like $b\delta^4$, and the rescaling cancels exactly this factor. In Section~\ref{sec:spacetime-proof} a fixed finite partition puts the normalized operator in the form of Lemma~\ref{lem:uniform-hormander}.

\subsection{Normalized input and jet coordinates}\label{sec:input-jet-coordinates}
We introduce the normalized input variables and the output coordinates
determined by the value and first three angular derivatives of the phase
at $\theta_0$ (its 3-jet). The corresponding norm relations are collected
in \eqref{eq:norm-relations}.

Fix a compact rectangle in $(0,\infty)^2$ containing the $(t,s)$ support in its interior. All constants in the proof may depend on this rectangle, the fixed normalized frequency ranges and separation constants, and the seminorms in \eqref{eq:symbol-bounds}, but not on $\lambda,b,\delta$ or the choice of $\theta_0$. In particular, the constants implicit in $\delta\lesssim b$ and $\theta_0\sim b$ are fixed. We will check that the coordinate neighborhoods, amplitude bounds, and finite partitions used below share this independence of the scales.

Restrict $F$ to the frequency support of the amplitude.

\subsubsection*{Normalization of the input.}
Polar coordinates in each frequency plane give
\begin{equation}
d\xi\,d\eta=\lambda^4\delta^2r_1r_2\,d\zeta.
\label{eq:input-jacobian}
\end{equation}
Indeed $d\xi=|\xi|\,d|\xi|\,d\arg\xi=\lambda^2\delta r_1\,dr_1\,du$, and similarly for $\eta$. It is convenient to absorb the square root of this Jacobian into the input. With $(\xi,\eta)$ as in \eqref{eq:normalized-frequencies}, define
\begin{equation}
F^\sharp(\zeta)=\lambda^2\delta\sqrt{r_1r_2}F(\xi,\eta),
\label{eq:normalized-input}
\end{equation}
so that $\|F^\sharp\|_{L^2_\zeta}=\|F\|_{L^2_{\xi,\eta}}$ by \eqref{eq:input-jacobian}.
Write the angular part of the phase as
\begin{equation}
H_{t,s}(\theta)=\Phi_{t,s}(e(\theta))=(t^2\cos^2\theta+s^2\sin^2\theta)^{1/2},
\label{eq:angular-phase}
\end{equation}
and define
\[
\mathcal T_1G(x,t,s)=\int e^{i\Psi_1(x,t,s,\zeta)}\sqrt{r_1r_2}\widetilde m(t,s,\zeta)G(\zeta)\,d\zeta,
\]
where
\begin{equation}
\begin{split}
\Psi_1(x,t,s,\zeta)&=\lambda r_1\{x\cdot e(\theta_0+\delta u)+H_{t,s}(\theta_0+\delta u)\}\\
&\qquad+\lambda r_2\{x\cdot e(\theta_0+\delta v)+H_{t,s}(\theta_0+\delta v)\}.
\end{split}
\label{eq:normalized-input-phase}
\end{equation}
Since $\Phi_{t,s}$ is homogeneous of degree one, $\Phi_{t,s}(\lambda re(\theta))=\lambda rH_{t,s}(\theta)$; hence the change of variables \eqref{eq:normalized-frequencies} and the definition \eqref{eq:normalized-input} give the exact identity
\begin{equation}
\mathcal TF=\lambda^2\delta\mathcal T_1F^\sharp.
\label{eq:input-norm-relation}
\end{equation}
The normalized input now lies in a fixed compact set. The remaining problem is to choose output coordinates in which the phase is uniformly nondegenerate.

\subsubsection*{The parameter coordinates.}
The combination of angular derivatives that controls the output coordinates is the radius of curvature
\begin{equation}
\rho_{t,s}(\theta):=H_{t,s}(\theta)+H''_{t,s}(\theta)=\frac{t^2s^2}{H_{t,s}(\theta)^3}.
\label{eq:curvature-radius}
\end{equation}
Here primes are angular derivatives.
On a slightly enlarged fixed parameter rectangle,
\[
c\le t,s,H_{t,s}(\theta)\le C,
\]
and all fixed-order derivatives of $H_{t,s}$ and $\rho_{t,s}$ in $(t,s,\theta)$ are uniformly bounded. Indeed, $H_{t,s}(\theta)\ge\min(t,s)$, so the denominators in the differentiated formulas stay uniformly away from zero.

Consider the function
\[
P_{x,t,s}(\theta)=x\cdot e(\theta_0+\theta)+H_{t,s}(\theta_0+\theta).
\]
Its 3-jet at $\theta=0$ determines the output coordinates. Let $H_0=H_{t,s}(\theta_0)$ be the value of \eqref{eq:angular-phase} at $\theta_0$, put
\begin{equation}
B_0=\sin2\theta_0\sim b,
\label{eq:reference-angle}
\end{equation}
and define
\begin{equation}
\varrho=(\varrho_1,\varrho_2)=\Bigl(\rho_{t,s}(\theta_0),\frac{\rho'_{t,s}(\theta_0)}{B_0}\Bigr).
\label{eq:curvature-coordinates}
\end{equation}
Thus $\varrho_1$ is the radius of curvature at $\theta_0$ and $\varrho_2$ is its normalized angular derivative; both depend only on $(t,s)$ and $\theta_0$. Set
\begin{equation}
(z_0,z_1)=\bigl(x\cdot e(\theta_0)+H_0,\ x\cdot e'(\theta_0)+H'_{t,s}(\theta_0)\bigr).
\label{eq:spatial-coordinates}
\end{equation}
Then
\[
(P_{x,t,s},P'_{x,t,s},P''_{x,t,s},P'''_{x,t,s})(0)=(z_0,\,z_1,\,\varrho_1-z_0,\,B_0\varrho_2-z_1).
\]
Thus $z_0,z_1$ record the value and first derivative of the angular phase, whereas $\varrho_1,B_0\varrho_2$ isolate its next two derivatives after those lower-order terms have been removed. The factor $B_0$ reflects the vanishing of $\rho'_{t,s}$ on the axes, which follows from reflection symmetry.

Direct differentiation gives
\begin{equation}
(\varrho_1,\varrho_2)=\Bigl(\frac{t^2s^2}{H_0^3},\frac{3t^2s^2(t^2-s^2)}{2H_0^5}\Bigr),
\label{eq:curvature-jacobian}
\end{equation}
whose Jacobian is computed in \eqref{eq:parameter-jacobian} below. The first component is \eqref{eq:curvature-radius} at $\theta_0$. For the second, differentiating \eqref{eq:angular-phase} gives $H'_{t,s}(\theta)=(s^2-t^2)\sin2\theta/(2H_{t,s}(\theta))$, so
\[
\rho'_{t,s}(\theta_0)=-\frac{3t^2s^2H'_{t,s}(\theta_0)}{H_0^4}
=\frac{3t^2s^2(t^2-s^2)}{2H_0^5}B_0;
\]
the factor $B_0$ cancels in the definition \eqref{eq:curvature-coordinates} of $\varrho_2$.

For the determinant, write $\varrho_2=\frac32\varrho_1(t^2-s^2)/H_0^2$, and recall $H_0^2=t^2\cos^2\theta_0+s^2\sin^2\theta_0$. Then
\[
\begin{aligned}
(\partial_t,\partial_s)\varrho_1&=\varrho_1\left(\frac2t-\frac{3t\cos^2\theta_0}{H_0^2},\ \frac2s-\frac{3s\sin^2\theta_0}{H_0^2}\right),\\
(\partial_t,\partial_s)\frac{t^2-s^2}{H_0^2}&=\frac{2ts}{H_0^4}\,(s,-t).
\end{aligned}
\]
In the determinant of $(\varrho_1,\varrho_2)$ the terms in which $(t^2-s^2)/H_0^2$ is not differentiated cancel, so
\[
\frac{\partial(\varrho_1,\varrho_2)}{\partial(t,s)}
=\frac32\varrho_1\Bigl(\partial_t\varrho_1\,\partial_s\frac{t^2-s^2}{H_0^2}-\partial_s\varrho_1\,\partial_t\frac{t^2-s^2}{H_0^2}\Bigr)
=-\frac{3ts\varrho_1}{H_0^4}(t\partial_t\varrho_1+s\partial_s\varrho_1).
\]
Since $\varrho_1$ is homogeneous of degree one in $(t,s)$, Euler's identity gives $t\partial_t\varrho_1+s\partial_s\varrho_1=\varrho_1$, and hence
\begin{equation}
\frac{\partial(\varrho_1,\varrho_2)}{\partial(t,s)}=-\frac{3ts\varrho_1^2}{H_0^4}=-\frac{3t^5s^5}{H_0^{10}}.
\label{eq:parameter-jacobian}
\end{equation}
In particular, the explicit formula for $\varrho_2$ contains no factor $B_0^{-1}$ and extends smoothly to $\theta_0=0$. On the fixed parameter rectangle, \eqref{eq:curvature-jacobian} and \eqref{eq:parameter-jacobian} give
\[
c\le\left|\frac{\partial(\varrho_1,\varrho_2)}{\partial(t,s)}\right|\le C,
\]
together with $c\le H_0,\varrho_1\le C$ and $|\varrho_2|\le C$.
All fixed-order derivatives of this parameter map are bounded by constants independent of $\theta_0$. Thus the division by $B_0$ in \eqref{eq:curvature-coordinates} causes no loss as $b\to0$.

The inverse map is explicit:
\begin{equation}
\begin{aligned}
H_0&=\varrho_1\left(1+\frac{2\varrho_2}{3\varrho_1}\sin^2\theta_0\right)^{-1}
\left(1-\frac{2\varrho_2}{3\varrho_1}\cos^2\theta_0\right)^{-1},\\
t&=H_0\left(1+\frac{2\varrho_2}{3\varrho_1}\sin^2\theta_0\right)^{1/2},\\
s&=H_0\left(1-\frac{2\varrho_2}{3\varrho_1}\cos^2\theta_0\right)^{1/2}.
\end{aligned}
\label{eq:inverse-curvature-coordinates}
\end{equation}
Indeed, the two expressions under the square roots equal $t^2/H_0^2$ and $s^2/H_0^2$, respectively. Since $c\le t,s,H_{t,s}\le C$, they are bounded below by a positive constant; they are also the two factors inverted in the formula for $H_0$. On a fixed bounded $\varrho$ range with $\varrho_1$ bounded away from zero, the square-root arguments have uniformly bounded first derivatives. Their positive lower bounds therefore persist on a neighborhood of the parameter image of uniform positive width. The inverse values $t,s$ stay in a fixed compact subset of $(0,\infty)^2$ there. Differentiating \eqref{eq:inverse-curvature-coordinates} gives, for each fixed order,
\[
|\partial_{\varrho_1}^j\partial_{\varrho_2}^kt|+|\partial_{\varrho_1}^j\partial_{\varrho_2}^ks|\le C_{j,k}
\]
on these neighborhoods, with constants independent of $\lambda,b,\delta$. The images lie in a fixed bounded set. They can therefore be covered by a uniformly bounded number of convex rectangles of fixed small size, with partition cutoffs having uniform derivatives. Their locations may depend on $\theta_0$; their sizes and number do not. Multiply $m$ by a partition of unity $\chi_k(\varrho(t,s))$ subordinate to these rectangles; the products satisfy the same bounds \eqref{eq:symbol-bounds}. From now on we fix one $k$ and work on the corresponding rectangle, which is convex.

\subsubsection*{The output coordinates.}
Let
\[
X(z,t,s)=(z_0-H_0)e(\theta_0)+(z_1-H'_{t,s}(\theta_0))e'(\theta_0),
\]
where $z=(z_0,z_1)$ and $H_0=H_{t,s}(\theta_0)$. This is the inverse of \eqref{eq:spatial-coordinates} for fixed $(t,s)$. The output change $(x,t,s)\mapsto(z,\varrho)$ is block triangular: $z$ depends on $x$ through a rotation of determinant one and a translation depending on $(t,s)$, while $\varrho$ depends only on $(t,s)$. Hence its absolute Jacobian equals that of $(t,s)\mapsto\varrho$ with $\theta_0$ fixed, and by \eqref{eq:parameter-jacobian},
\begin{equation}
\begin{gathered}
J(t,s)=\left|\frac{\partial(z_0,z_1,\varrho_1,\varrho_2)}{\partial(x,t,s)}\right|
=\frac{3t^5s^5}{H_0^{10}}\sim1,\\
dx\,dt\,ds=J(t,s)^{-1}\,dz\,d\varrho.
\end{gathered}
\label{eq:output-measure}
\end{equation}
With $t=t(\varrho)$ and $s=s(\varrho)$ given by \eqref{eq:inverse-curvature-coordinates}, we define
\[
\mathcal T_2G(z,\varrho)=J(t,s)^{-1/2}\int e^{i\Psi_2(z,\varrho,\zeta)}\sqrt{r_1r_2}\widetilde m(t,s,\zeta)G(\zeta)\,d\zeta,
\]
where $\Psi_2(z,\varrho,\zeta)=\Psi_1(X(z,t,s),t,s,\zeta)$. Thus
\[
\mathcal T_2G(z,\varrho)=J(t,s)^{-1/2}\mathcal T_1G(X(z,t,s),t,s).
\]
By \eqref{eq:output-measure}, the factor $J(t,s)^{-1/2}$ makes the output change unitary:
\begin{equation}
\|\mathcal T_1G\|_{L^2_{x,t,s}}=\|\mathcal T_2G\|_{L^2_{z,\varrho}}.
\label{eq:output-norm-relation}
\end{equation}
This change is used only for a spacetime norm. In Section~\ref{sec:maximal-estimates}, where a parameter supremum is taken at each fixed $x$, no such parameter-dependent spatial translation will be made.

\subsubsection*{The phase in jet coordinates.}
For a fixed $\theta$, differentiation in $w$ and the identity $H_{t,s}+H''_{t,s}=\rho_{t,s}$ in \eqref{eq:curvature-radius} give
\[
\frac{d}{dw}\{H_{t,s}(\theta_0+w)\cos(\theta-w)+H'_{t,s}(\theta_0+w)\sin(\theta-w)\}
=\rho_{t,s}(\theta_0+w)\sin(\theta-w).
\]
Integration from $0$ to $\theta$ therefore yields, with $H_0=H_{t,s}(\theta_0)$ and $B_0$ as in \eqref{eq:reference-angle},
\[
\begin{split}
H_{t,s}(\theta_0+\theta)&=H_0\cos\theta+H'_{t,s}(\theta_0)\sin\theta\\
&\qquad+\int_0^\theta\sin(\theta-w)\rho_{t,s}(\theta_0+w)\,dw.
\end{split}
\]
This is an exact identity, also for negative $\theta$ with the usual oriented integral. The addition formula $e(\theta_0+\theta)=e(\theta_0)\cos\theta+e'(\theta_0)\sin\theta$ handles the spatial term. Since
\[
\int_0^\theta(1,w)\sin(\theta-w)\,dw=(1-\cos\theta,\ \theta-\sin\theta),
\]
writing $\rho_{t,s}(\theta_0+w)=\varrho_1+B_0\varrho_2w+\{\rho_{t,s}(\theta_0+w)-\varrho_1-B_0\varrho_2w\}$ and using \eqref{eq:curvature-coordinates} and \eqref{eq:spatial-coordinates}, we obtain
\begin{equation}
P_{x,t,s}(\theta)=z_0\cos\theta+z_1\sin\theta+\varrho_1(1-\cos\theta)+B_0\varrho_2(\theta-\sin\theta)+E_\varrho(\theta),
\label{eq:exact-phase-expansion}
\end{equation}
where
\begin{equation}
E_\varrho(\theta)=\int_0^\theta\sin(\theta-w)\{\rho_{t,s}(\theta_0+w)-\varrho_1-B_0\varrho_2w\}\,dw.
\label{eq:phase-remainder}
\end{equation}
By \eqref{eq:normalized-input-phase} and the definition of $\Psi_2$,
\begin{equation}
\Psi_2(z,\varrho,\zeta)=\lambda r_1P_{x,t,s}(\delta u)+\lambda r_2P_{x,t,s}(\delta v),
\label{eq:transformed-phase}
\end{equation}
where $x=X(z,t,s)$ and $P_{x,t,s}$ is given by \eqref{eq:exact-phase-expansion}.

\subsection{Normalization of the phase}\label{sec:phase-normalization}
We next bound the remainder in the expansion \eqref{eq:exact-phase-expansion} and rescale the output variables. This leads to the normalized operator $S$ in \eqref{eq:normalized-operator}, and the norm relations \eqref{eq:norm-relations} reduce Proposition~\ref{prop:spacetime} to a uniform $L^2$ bound for $S$. The remainder is not expanded further; it stays in the phase throughout the proof.

\subsubsection*{The Taylor remainder.}
Taylor's formula is used only to bound the remainder in \eqref{eq:phase-remainder}. With $t=t(\varrho)$ and $s=s(\varrho)$ as in \eqref{eq:inverse-curvature-coordinates}, and since $\rho_{t,s}(\theta_0)=\varrho_1$ and $\rho'_{t,s}(\theta_0)=B_0\varrho_2$ by \eqref{eq:curvature-coordinates},
\[
\rho_{t,s}(\theta_0+w)-\varrho_1-B_0\varrho_2w
=w^2\int_0^1(1-\tau)\rho''_{t,s}(\theta_0+\tau w)\,d\tau.
\]
Consequently,
\begin{equation}
E_\varrho(\theta)=\theta^4\mathcal E_\varrho(\theta),
\label{eq:remainder-factorization}
\end{equation}
where an exact formula is
\[
\mathcal E_\varrho(\theta)=\int_0^1\int_0^1(1-r)r^2(1-\tau)
\frac{\sin(\theta(1-r))}{\theta(1-r)}\rho''_{t,s}(\theta_0+\tau r\theta)\,d\tau\,dr.
\]
The sine quotient extends smoothly across zero. To bound the derivatives in $\varrho$ uniformly in $b$, we express $\rho_{t,s}$ in terms of $\varrho$:
\begin{equation}
\rho_{t,s}(\theta_0+w)=\varrho_1\left(1-\frac{2\varrho_2}{3\varrho_1}\{\sin^2(\theta_0+w)-\sin^2\theta_0\}\right)^{-3/2}.
\label{eq:curvature-formula}
\end{equation}
Indeed,
\[
H_{t,s}(\theta_0+w)^2=H_0^2+(s^2-t^2)\{\sin^2(\theta_0+w)-\sin^2\theta_0\},
\]
and $(t^2-s^2)/H_0^2=2\varrho_2/(3\varrho_1)$ by \eqref{eq:curvature-jacobian}. The expression in parentheses in \eqref{eq:curvature-formula} equals $H_{t,s}(\theta_0+w)^2/H_0^2$, which is bounded above and below by fixed positive constants on the coordinate neighborhoods. This identity and the inverse-coordinate bounds give uniform bounds for all required $(\varrho,\theta)$ derivatives of $\rho''_{t,s}(\theta_0+\tau r\theta)$, for $0\le r,\tau\le1$. Differentiation under the integral for $\mathcal E$ therefore gives
\[
\bigl|\partial_\theta^j\partial_{\varrho_1}^k\partial_{\varrho_2}^{\ell}\mathcal E_\varrho(\theta)\bigr|\le C_{j,k,\ell},
\]
uniformly in $b,\delta$ on the coordinate neighborhoods and bounded $\theta$ ranges considered above. By \eqref{eq:remainder-factorization},
\[
E_\varrho(\delta u)=\delta^4u^4\mathcal E_\varrho(\delta u).
\]
Each $u$ derivative falling on $\mathcal E_\varrho(\delta u)$ contributes a factor $\delta$. Since $u$ and $\delta$ lie in fixed bounded ranges, the chain rule and the preceding derivative bounds give, for each fixed order required in the proof,
\begin{equation}
\|E_\varrho(\delta u)\|_{C^m_{u,\varrho}}\lesssim_m\delta^4.
\label{eq:remainder-bounds}
\end{equation}

\subsubsection*{Rescaled output coordinates.}
In \eqref{eq:transformed-phase}, the four main terms of \eqref{eq:exact-phase-expansion} at $\theta=\delta u$, as coefficients of $z_0,z_1,\varrho_1,\varrho_2$, have sizes $1,\delta,\delta^2,B_0\delta^3$. This determines the rescaling
\begin{equation}
(y_0,y_1,y_2,y_3)=(\lambda z_0,\ \lambda\delta z_1,\ \lambda\delta^2\varrho_1,\ \lambda B_0\delta^3\varrho_2).
\label{eq:rescaled-output}
\end{equation}
This diagonal rescaling has Jacobian $\lambda^4B_0\delta^6$, so by \eqref{eq:output-measure}
\begin{equation}
dy=\lambda^4B_0\delta^6\,dz\,d\varrho=\lambda^4B_0\delta^6J(t,s)\,dx\,dt\,ds.
\label{eq:measure-relations}
\end{equation}
Define
\begin{equation}
\gamma_\delta(u)=\left(\cos\delta u,\frac{\sin\delta u}{\delta},\frac{1-\cos\delta u}{\delta^2},\frac{\delta u-\sin\delta u}{\delta^3}\right).
\label{eq:model-curve}
\end{equation}
The limit $\gamma_\delta(u)\to(1,u,u^2/2,u^3/6)$ as $\delta\to0$ explains the four scales; the exact trigonometric functions will be retained. Let $(z,\varrho)$ be related to $y$ by \eqref{eq:rescaled-output}, and $t=t(\varrho)$, $s=s(\varrho)$. Substituting \eqref{eq:rescaled-output} into \eqref{eq:transformed-phase} gives $\Psi_2(z,\varrho,\zeta)=\widetilde\Psi(y,\zeta)$; for instance, $\lambda z_1\sin\delta u=y_1(\sin\delta u)/\delta$ and $\lambda B_0\varrho_2(\delta u-\sin\delta u)=y_3(\delta u-\sin\delta u)/\delta^3$. The fully normalized operator is therefore
\begin{equation}
SG(y)=\mathcal T_2G(z,\varrho)=\int e^{i\widetilde\Psi(y,\zeta)}m_S(y,\zeta)G(\zeta)\,d\zeta,
\label{eq:normalized-operator}
\end{equation}
where
\[
m_S(y,\zeta)=J(t,s)^{-1/2}\sqrt{r_1r_2}\widetilde m(t,s,\zeta),
\]
\begin{equation}
\begin{split}
\widetilde\Psi(y,\zeta)&=r_1y\cdot\gamma_\delta(u)+r_2y\cdot\gamma_\delta(v)\\
&\qquad+\lambda\{r_1E_\varrho(\delta u)+r_2E_\varrho(\delta v)\}.
\end{split}
\label{eq:normalized-phase}
\end{equation}
The cutoff $\chi_k(\varrho)$ is included in $\widetilde m$. At fixed $y$, the quantities $t,s$ depend only on $y_2,y_3$, not on $\zeta$. Hence \eqref{eq:symbol-bounds}, the bounds on $r_1,r_2$ in \eqref{eq:normalized-support}, and $J(t,s)\sim1$ in \eqref{eq:output-measure} give
\begin{equation}
\sup_{y,\zeta}|\partial_\zeta^\alpha m_S(y,\zeta)|\le C_\alpha\qquad (|\alpha|\le6),
\label{eq:amplitude-bounds}
\end{equation}
uniformly in $\lambda,b,\delta$. These are precisely the amplitude bounds needed for Lemma~\ref{lem:uniform-hormander}: only input derivatives enter, so rescaling the output cutoff preserves them.

By \eqref{eq:input-norm-relation} and \eqref{eq:output-norm-relation}, and since $SG(y)=\mathcal T_2G(z,\varrho)$ with $dy=\lambda^4B_0\delta^6\,dz\,d\varrho$ by \eqref{eq:measure-relations}, we have
\begin{equation}
\begin{split}
\|\mathcal TF\|_2&=\lambda^2\delta\|\mathcal T_1F^\sharp\|_2=\lambda^2\delta\|\mathcal T_2F^\sharp\|_2\\
&=\lambda^2\delta(\lambda^4B_0\delta^6)^{-1/2}\|SF^\sharp\|_2.
\end{split}
\label{eq:norm-relations}
\end{equation}
Since $\lambda^2\delta(\lambda^4B_0\delta^6)^{-1/2}=B_0^{-1/2}\delta^{-2}$, it remains to prove that $S$ is bounded on $L^2$ with a uniform constant.

\subsection{The mixed Hessian}\label{sec:mixed-hessian}
The main hypothesis of Lemma~\ref{lem:uniform-hormander} concerns the mixed Hessian of the normalized phase,
\[
M(y,\zeta):=D_y(\nabla_\zeta\widetilde\Psi(y,\zeta)),
\]
with rows indexed by $\zeta$ and columns by $y$. We first bound the entries of $M$ by differentiating the normalized phase. Its determinant is then computed in the original coordinates, where the factors $b$ and $\delta$ appear explicitly, and transferred to the normalized coordinates. Together these give \eqref{eq:hessian-bounds}.

\subsubsection*{The normalized mixed Hessian.}
We now differentiate \eqref{eq:normalized-phase}. For a normalized angular variable $w$, with $\gamma_\delta$ as in \eqref{eq:model-curve}, set
\[
\widetilde\gamma(w)=\gamma_\delta(w)+\left(0,0,
\delta^{-2}\partial_{\varrho_1}E_\varrho(\delta w),
(B_0\delta^3)^{-1}\partial_{\varrho_2}E_\varrho(\delta w)\right).
\]
Here $B_0=\sin2\theta_0$ is as in \eqref{eq:reference-angle}, and the dependence of $\widetilde\gamma$ on $\varrho$ is suppressed. In the remainder term, by \eqref{eq:rescaled-output}, the chain rule contributes the factors $(\lambda\delta^2)^{-1}$ and $(\lambda B_0\delta^3)^{-1}$ when differentiating in $y_2$ and $y_3$, respectively. These cancel the factor $\lambda$ in \eqref{eq:normalized-phase}. Thus
\begin{equation}
\nabla_y\widetilde\Psi(y,\zeta)=r_1\widetilde\gamma(u)+r_2\widetilde\gamma(v).
\label{eq:phase-gradient}
\end{equation}
In particular, $M$ depends on $y$ only through $\varrho$ and is independent of $y_0,y_1$. Since $M$ is the transpose of $D_\zeta(\nabla_y\widetilde\Psi)$, its $C^m$ norm in $(\varrho,\zeta)$ is bounded by the $C^{m+1}$ norm of \eqref{eq:phase-gradient}.

The remainder need not be small: $\lambda E_\varrho(\delta u)$ can have size $\lambda\delta^4$. Its normalized parameter derivatives are nevertheless bounded. By \eqref{eq:remainder-bounds},
\begin{equation}
\begin{aligned}
\bigl\|\delta^{-2}\partial_{\varrho_1}E_\varrho(\delta w)\bigr\|_{C^m_{w,\varrho_1,\varrho_2}}&\lesssim_m\delta^2,\\
\bigl\|(B_0\delta^3)^{-1}\partial_{\varrho_2}E_\varrho(\delta w)\bigr\|_{C^m_{w,\varrho_1,\varrho_2}}&\lesssim_m\frac{\delta}{B_0}\lesssim\frac{\delta}{b}\lesssim1.
\end{aligned}
\label{eq:normalized-remainder-bounds}
\end{equation}
The last bound is where $\delta\lesssim b$ enters the upper bounds. The components of $\gamma_\delta$ in \eqref{eq:model-curve} extend smoothly to $\delta=0$ and have uniformly bounded $w$ derivatives on the fixed angular range. Using \eqref{eq:normalized-remainder-bounds} with $m+1$ in place of $m$ and the fixed bounds on $r_1,r_2$, we obtain $\|M\|_{C^m_{\varrho_1,\varrho_2,\zeta}}\le C_m$. These are derivatives in the normalized parameter variables $\varrho$, as required for the finite partition below; no higher $y$ derivatives are used. The same upper bounds hold on fixed enlargements of the normalized input ranges: the remainder formula is smooth for bounded $u,v$, and these upper bounds do not use the separation condition $|u-v|\gtrsim1$.

\subsubsection*{The determinant in the original coordinates.}
Recall from Section~\ref{sec:fourier-representation} that $\xi_{t,s}=(t\xi_1,s\xi_2)$ and $\Phi_{t,s}(\xi)=|\xi_{t,s}|$, and set $\vartheta_1=\arg\xi_{t,s}$, $\vartheta_2=\arg\eta_{t,s}$. With
\[
\Gamma(\theta)=(\cos\theta,\sin\theta,\cos^2\theta,\sin^2\theta),
\]
a direct calculation of the four output derivatives gives
\[
\nabla_{x,t,s}\Psi=\diag(t^{-1},s^{-1},t^{-1},s^{-1})\{\Phi_{t,s}(\xi)\Gamma(\vartheta_1)+\Phi_{t,s}(\eta)\Gamma(\vartheta_2)\}.
\]
At fixed $(t,s)$, the diagonal matrix has determinant $(t^2s^2)^{-1}$, while the changes $\xi\mapsto\xi_{t,s}$ and $\eta\mapsto\eta_{t,s}$ have determinant $t^2s^2$. Write $\Phi_1=\Phi_{t,s}(\xi)$ and $\Phi_2=\Phi_{t,s}(\eta)$ for the two radial variables. In the polar coordinates $(\Phi_1,\vartheta_1)$ and $(\Phi_2,\vartheta_2)$ of $\xi_{t,s}$ and $\eta_{t,s}$, the derivative of $(\Phi_1,\vartheta_1,\Phi_2,\vartheta_2)\mapsto\Phi_1\Gamma(\vartheta_1)+\Phi_2\Gamma(\vartheta_2)$ has columns
\[
\Gamma(\vartheta_1),\quad\Phi_1\Gamma'(\vartheta_1),\quad\Gamma(\vartheta_2),\quad\Phi_2\Gamma'(\vartheta_2),
\]
and the polar Jacobian $\Phi_1\Phi_2$ cancels the same factor. Hence
\[
|\det\partial_{(\xi,\eta)}\nabla_{x,t,s}\Psi|
=|\det(\Gamma(\vartheta_1),\Gamma'(\vartheta_1),\Gamma(\vartheta_2),\Gamma'(\vartheta_2))|,
\]
and
\begin{equation}
\det(\Gamma(\vartheta_1),\Gamma'(\vartheta_1),\Gamma(\vartheta_2),\Gamma'(\vartheta_2))
=-4\sin^4\left(\frac{\vartheta_1-\vartheta_2}{2}\right)\sin(\vartheta_1+\vartheta_2).
\label{eq:model-determinant}
\end{equation}
To verify the constant and sign, put $\theta_+=(\vartheta_1+\vartheta_2)/2$ and $\theta_-=(\vartheta_1-\vartheta_2)/2$. Rotate the first two rows by $-\theta_+$, and replace the last two rows by their sum and difference; the combined row operation has determinant $-2$. Next subtract column one from column three and add column two to column four. Expansion along the third row, now $(1,0,0,0)$, reduces the transformed determinant to
\[
\begin{split}
&\det\begin{pmatrix}
-\sin\theta_-&0&0\\
\cos\theta_-&-2\sin\theta_-&2\cos\theta_-\\
-2\sin(2\theta_++2\theta_-)&2\sin(2\theta_+)\sin(2\theta_-)&-4\sin(2\theta_+)\cos(2\theta_-)
\end{pmatrix}\\
&\qquad=8\sin^4\theta_-\sin(2\theta_+).
\end{split}
\]
Division by $-2$ gives \eqref{eq:model-determinant}, displaying both the fourth-order vanishing at $\vartheta_1=\vartheta_2$ and the additional axial factor. Consequently
\begin{equation}
\bigl|\det\partial_{(\xi,\eta)}\nabla_{x,t,s}\Psi\bigr|
=4\sin^4\left(\frac{\vartheta_1-\vartheta_2}{2}\right)|\sin(\vartheta_1+\vartheta_2)|\sim b\delta^4.
\label{eq:original-hessian-determinant}
\end{equation}
To make the constants in the last comparison explicit, the angular map $\theta\mapsto\arg(t\cos\theta,s\sin\theta)$ has derivative $ts/H_{t,s}(\theta)^2$, bounded above and below by fixed positive constants. The mean value theorem and \eqref{eq:normalized-support} therefore give $|\vartheta_1-\vartheta_2|\sim\delta|u-v|\sim\delta$. Since $\theta_0+\delta u,\theta_0+\delta v\in[9b/10,21b/10]\subset(0,\pi/2)$ and this angular map preserves $(0,\pi/2)$, the transformed angles lie in $(0,\pi/2)$, so $|\sin((\vartheta_1-\vartheta_2)/2)|\sim\delta$. For the axial factor,
\[
\sin(\vartheta_1+\vartheta_2)=\frac{ts\sin(2\theta_0+\delta(u+v))}{H_{t,s}(\theta_0+\delta u)H_{t,s}(\theta_0+\delta v)}.
\]
Here $2\theta_0+\delta(u+v)\in[9b/5,21b/5]\subset(0,21\pi/40]$, so its sine is comparable to $b$. Since the other factors in this identity are bounded above and below, $|\sin(\vartheta_1+\vartheta_2)|\sim b$ with uniform constants. This proves \eqref{eq:original-hessian-determinant} without any constant deteriorating as $b$ or $\delta$ tends to zero.

\subsubsection*{The determinant after normalization.}
Since the input change depends only on $\zeta$ and the output change only on $y$, the mixed chain rule involves only their first derivatives:
\[
M=[D_\zeta(\xi,\eta)]^{\mathsf T}[D_{(\xi,\eta)}(\nabla_{x,t,s}\Psi)]^{\mathsf T}D_y(x,t,s);
\]
the transposes appear because the rows of $M$ are indexed by the input. By \eqref{eq:input-jacobian} and \eqref{eq:measure-relations}, $|\partial(\xi,\eta)/\partial\zeta|=\lambda^4\delta^2r_1r_2$ and $|\partial(x,t,s)/\partial y|=(\lambda^4B_0\delta^6J(t,s))^{-1}$. Hence, by \eqref{eq:original-hessian-determinant}, \eqref{eq:normalized-support}, \eqref{eq:output-measure}, and $B_0\sim b$ in \eqref{eq:reference-angle},
\begin{equation}
|\det M(y,\zeta)|=\frac{r_1r_2}{B_0\delta^4J(t,s)}\bigl|\det\partial_{(\xi,\eta)}\nabla_{x,t,s}\Psi\bigr|
\sim\frac{r_1r_2}{J(t,s)}\frac{b}{B_0}\sim1.
\label{eq:normalized-hessian-determinant}
\end{equation}
Thus the factor $b\delta^4$ in the original determinant is canceled by the output normalization; all powers of $\lambda$ have already canceled in the two Jacobians. Combining \eqref{eq:normalized-hessian-determinant} with the upper bounds above gives
\begin{equation}
|M\omega|\ge c|\omega|\qquad(\omega\in\R^4),
\label{eq:hessian-bounds}
\end{equation}
together with $\|M\|\le C$ and $\|M\|_{C^m_{\varrho_1,\varrho_2,\zeta}}\le C_m$ for $0\le m\le6$, with $c,C,C_m$ independent of $\lambda,b,\delta$. To see the lower bound explicitly, let $\sigma_1\ge\sigma_2\ge\sigma_3\ge\sigma_4$ be the singular values of $M$. The entry bound gives $\sigma_j\le\|M\|\le C$, while $\prod_{j=1}^4\sigma_j=|\det M|\gtrsim1$. Consequently
\[
\sigma_{\min}(M)\ge\frac{|\det M|}{\|M\|^3}\gtrsim1.
\]
These are the matrix bounds needed to apply Lemma~\ref{lem:uniform-hormander}.

\subsection{Completion of the proof of Proposition~\ref{prop:spacetime}}\label{sec:spacetime-proof}
It remains to prove the uniform bound
\begin{equation}
\|SG\|_{L^2_y}\lesssim\|G\|_{L^2_\zeta}
\label{eq:normalized-L2-bound}
\end{equation}
for the operator in \eqref{eq:normalized-operator}, uniformly in $\lambda,b,\delta$. We use the following quantitative form of H\"ormander's $L^2$ estimate~\cite[Theorem 1.1]{Hormander}. The proof below shows that the constant depends on the amplitude only through its input derivatives and is independent of the size of the output domain.

\subsubsection*{A quantitative H\"ormander estimate.}
Let $\Omega\subset\R^4$ be open and convex, let $U\subset\R^4$ be open, and fix a compact set $Z\Subset U$. Let $\varphi$ be real and smooth on $\Omega\times U$, and let $q$ be a smooth amplitude with $\supp q(y,\cdot)\subset Z$ for every $y\in\Omega$. Define
\[
T_\varphi g(y)=\int_U e^{i\varphi(y,\zeta)}q(y,\zeta)g(\zeta)\,d\zeta,
\]
and let $M=D_y(\nabla_\zeta\varphi)$, where the rows of the $4\times4$ matrix $M$ are indexed by $\zeta$.

\begin{lemma}[Uniform H\"ormander $L^2$ estimate]\label{lem:uniform-hormander}
Suppose a fixed matrix $M_0$ and $c>0$ satisfy $|M_0\omega|\ge c|\omega|$ for $\omega\in\R^4$ and, on $\Omega\times U$,
\[
\|M-M_0\|\le c/2,
\]
and
\[
\sup_{\Omega\times U}\bigl(|\partial_\zeta^\alpha q|+\|\partial_\zeta^\alpha M\|\bigr)\le A_\alpha\qquad(|\alpha|\le6).
\]
Then $\|T_\varphi g\|_{L^2(\Omega)}\le C\|g\|_{L^2(U)}$, where $C$ depends only on $Z$, $c$ and the displayed derivative bounds, not on the size of $\Omega$ or on output derivatives of $q$.
\end{lemma}

\begin{proof}
First restrict the output to $\Omega\cap B(0,N)$, obtaining a Hilbert--Schmidt operator. For $w=y-y'$ and $\Theta(\zeta)=\varphi(y,\zeta)-\varphi(y',\zeta)$, convexity gives
\[
\nabla_\zeta\Theta=\int_0^1M(y'+\tau w,\zeta)w\,d\tau.
\]
Consequently $|\nabla_\zeta\Theta|\ge(c/2)|w|$ and $|\partial_\zeta^\alpha\Theta|\lesssim_\alpha|w|$ for $1\le|\alpha|\le7$. When $|w|\ge1$, the operator
\[
\mathcal L=\frac{\nabla_\zeta\Theta}{i|\nabla_\zeta\Theta|^2}\cdot\nabla_\zeta
\]
satisfies $\mathcal L(e^{i\Theta})=e^{i\Theta}$, and its coefficients and their $\zeta$ derivatives through order six are $O(|w|^{-1})$. Six integrations by parts in the kernel of $T_\varphi T_\varphi^*$, together with the direct bound for $|w|\le1$, yield
\[
\Big|\int_U e^{i\Theta(\zeta)}q(y,\zeta)\overline{q(y',\zeta)}\,d\zeta\Big|\lesssim(1+|y-y'|)^{-6}.
\]
There are no boundary terms because the input support lies in $Z\Subset U$. Schur's test in dimension four gives a bound independent of $N$; letting $N\to\infty$ proves the assertion.
\end{proof}

\subsubsection*{A uniform finite partition.}
Let $c$ be the constant in \eqref{eq:hessian-bounds}, and let $C_1$ bound the Lipschitz constant of $M$ in operator norm as a function of the six variables $(\varrho,\zeta)$, also on the enlarged input ranges (see the end of the paragraph following \eqref{eq:normalized-remainder-bounds}); by the bound $\|M\|_{C^1_{\varrho_1,\varrho_2,\zeta}}\le C_1$ following \eqref{eq:hessian-bounds}, $C_1$ is independent of $\lambda,b,\delta$. Around points $(\varrho,\zeta)$ of the amplitude support, take products $R\times U$ of an open convex parameter rectangle $R$ and an input box $U$ of a common fixed size, with diameter at most $c/(2C_1)$ and small enough to lie in one inverse-coordinate neighborhood. If $M_0$ is the matrix at the chosen point, then $|M_0\omega|\ge c|\omega|$ by \eqref{eq:hessian-bounds}, and the mean value theorem gives
\begin{equation}
\|M(y,\zeta)-M_0\|\le c/2.
\label{eq:hessian-freezing}
\end{equation}
The reference matrix $M_0$ may depend on the scales, but its lower bound $c$ does not.

All the $(\varrho,\zeta)$ ranges lie in a fixed bounded subset of $\R^6$. We can cover the support by boxes of half this size centered at points of the amplitude support, using a uniformly bounded number of them. Take a smooth partition of unity subordinate to these smaller boxes, and use the larger boxes as the neighborhoods $R\times U$. The cutoff derivatives in $(\varrho,\zeta)$ are uniformly bounded, and the input support of each piece lies a fixed positive distance inside $U$.

\subsubsection*{Verification of the hypotheses of Lemma~\ref{lem:uniform-hormander}.}
Under \eqref{eq:rescaled-output}, the output variables with $\varrho\in R$ form the convex set
\[
\Omega=\R^2\times\diag(\lambda\delta^2,\lambda B_0\delta^3)R,
\]
and \eqref{eq:hessian-freezing} holds throughout $\Omega\times U$. The first two output variables need no subdivision because $M$ is independent of $y_0,y_1$. Thus a small value of $\lambda\delta^2$ or $\lambda B_0\delta^3$ changes the shape of $\Omega$, but not the number of pieces or the matrix bounds. Each input support is compactly contained in its $U$ with a uniform margin. The amplitude bounds \eqref{eq:amplitude-bounds}, the fixed cutoff bounds, and the bounds $\|M\|_{C^m_{\varrho_1,\varrho_2,\zeta}}\le C_m$ following \eqref{eq:hessian-bounds} give uniform $\zeta$ derivatives of $m_S$ and $M$ through order six. These are exactly the hypotheses of Lemma~\ref{lem:uniform-hormander}. Its constant is independent of the size of $\Omega$ and of output derivatives of the rescaled cutoff. Hence it gives the same $L^2$ bound on every piece. Summing over the uniformly bounded number of pieces proves \eqref{eq:normalized-L2-bound}. The argument applies for every $0<\delta\lesssim b$, including when $\lambda B_0\delta^3$ is small; no lower bound on $\delta$ in terms of $\lambda$ is needed.

\subsubsection*{Conclusion.}
By \eqref{eq:normalized-L2-bound} and \eqref{eq:norm-relations},
\[
\|\mathcal TF\|_2\lesssim B_0^{-1/2}\delta^{-2}\|F\|_2\lesssim b^{-1/2}\delta^{-2}\|F\|_2
\]
on each parameter coordinate neighborhood, with constants depending only on the fixed data listed at the start of the proof. Summing over the uniformly bounded number of neighborhoods proves Proposition~\ref{prop:spacetime}.\qed

\begin{remark}[Sharpness of \eqref{eq:spacetime-bound}]\label{rem:spacetime-sharpness}
The factor $b^{-1/2}\delta^{-2}$ is sharp when $\lambda b\delta^3\gtrsim1$. Let $E_0$ correspond to a fixed small box in the normalized input support, and choose an admissible amplitude $m$ equal to one on $E_0$ for $(t,s)$ near $(1,1)$. Set
\[
F(\xi,\eta)=e^{-i\Psi(-e(\theta_0),1,1,\xi,\eta)}\mathbf1_{E_0}(\xi,\eta),
\]
so that $\|F\|_2^2=|E_0|\sim\lambda^4\delta^2$.
The reference output $(-e(\theta_0),1,1)$ corresponds to $y_*=(0,0,\lambda\delta^2,0)$ in \eqref{eq:rescaled-output}. By \eqref{eq:phase-gradient} and \eqref{eq:normalized-remainder-bounds}, $\nabla_y\widetilde\Psi$ is uniformly bounded. Thus, on a sufficiently small fixed cube about $y_*$, the phase remaining after modulation has absolute value at most $\pi/3$ on $E_0$. The assumption $\lambda b\delta^3\gtrsim1$, which also gives $\lambda\delta^2\gtrsim\lambda b\delta^3\gtrsim1$ since $b\delta\lesssim1$, and the uniform inverse bounds in \eqref{eq:inverse-curvature-coordinates} ensure that this cube lies in the parameter region where $m=1$.

Consequently, $|\mathcal TF|\ge|E_0|/2$ on its preimage, whose measure is comparable to $(\lambda^4b\delta^6)^{-1}$ by \eqref{eq:measure-relations}. Hence
\[
\|\mathcal TF\|_2\gtrsim\lambda^4\delta^2(\lambda^4b\delta^6)^{-1/2}\sim b^{-1/2}\delta^{-2}\|F\|_2.
\]
\end{remark}

\section{Maximal estimates for the centered operator}\label{sec:maximal-estimates}

We prove the two maximal estimates used in Section~\ref{sec:angular-decomposition}. For Proposition~\ref{prop:separated}, the spacetime estimate and the parameter Sobolev inequality give the $L^2$ endpoint, while the kernel $L^1$ bound gives the $L^\infty$ endpoint. Proposition~\ref{prop:short} uses the differentiated kernel $L^1$ bounds to obtain an $L^6$ estimate, and the kernel $L^\infty$ bound for the other endpoint. Interpolation completes both arguments. We follow the parameter Sobolev approach of~\cite[Section~2]{LLOelliptic}, keeping track of the amplitudes produced by differentiation. The adapted coordinates below exploit cancellation of the centered phase; the parameter Sobolev inequality and kernel estimates are also used in Section~\ref{sec:translated-sufficiency}.

\subsection{The parameter Sobolev inequality and adapted coordinates}\label{sec:parameter-sobolev}
We choose coordinates in which the derivative scales in the two parameters are $\lambda$ and $\max\{1,\lambda b\delta\}$. The following two-parameter form of~\cite[Lemma~2.1]{LLOelliptic} then converts derivative bounds into a bound for the parameter supremum.

\begin{lemma}\label{lem:parameter-sobolev}
Let $E_1,E_2$ be fixed compact intervals of length comparable to one, $L_1,L_2\geq 1$, and $1\leq r\leq\infty$. For smooth $G=G(t,s)$,
\begin{equation}
\sup_{(t,s)\in E_1\times E_2}|G(t,s)|
\lesssim (L_1L_2)^{1/r}\sum_{\alpha,\beta\in\{0,1\}}
L_1^{-\alpha}L_2^{-\beta}\|\partial_t^\alpha\partial_s^\beta G\|_{L^r(E_1\times E_2)}.
\label{eq:parameter-sobolev}
\end{equation}
A fixed enlargement of the rectangle may be used on the right.
\end{lemma}

\begin{proof}
Let $L\ge1$ and let $g$ be smooth on $E_1$. For each $t\in E_1$, choose an interval $I\subset E_1$ containing $t$ and of length comparable to $L^{-1}$. The fundamental theorem of calculus and averaging give
\[
|g(t)|\lesssim L\int_I|g|+\int_I|g'|
\lesssim L^{1/r}\bigl(\|g\|_{L^r(E_1)}+L^{-1}\|g'\|_{L^r(E_1)}\bigr).
\]
Apply this successively in the two variables, with $L=L_1$ and $L=L_2$, and then bound the local norms by the norms on the fixed rectangle. Near the endpoints use a one-sided interval, or extend to a fixed slightly larger rectangle. This gives \eqref{eq:parameter-sobolev}.
\end{proof}

We will apply the lemma in the coordinates $(\mu,\nu)$ below, first at each fixed $x$ and then in the spatial norm. The scales $L_1,L_2$ measure parameter derivatives; no parameter Fourier support condition is required.

Retain the reference angle $\theta_0\sim b$ chosen between the two enlarged angular supports in Section~\ref{sec:bilinear-spacetime}; both supports then lie within $O(\delta)$ of $\theta_0$. With $H_{t,s}$ as in \eqref{eq:angular-phase}, set
\begin{equation}
(\mu,\nu)=\Bigl(H_{t,s}(\theta_0),\frac{s^2-t^2}{H_{t,s}(\theta_0)}\Bigr).
\label{eq:sobolev-coordinates}
\end{equation}
These coordinates leave $x$ unchanged. At the center angle, the phase depends only on $\mu$, so its $\nu$ derivative vanishes there. This is the cancellation responsible for the smaller derivative scale $\lambda b\delta$.

The inverse functions $t=t(\mu,\nu)$ and $s=s(\mu,\nu)$ satisfy
\begin{equation}
(t^2,s^2)=(\mu^2-\mu\nu\sin^2\theta_0,\ \mu^2+\mu\nu\cos^2\theta_0),
\label{eq:inverse-sobolev-coordinates}
\end{equation}
and $\partial(\mu,\nu)/\partial(t,s)=2ts/H_{t,s}(\theta_0)^2\sim1$, so that $dt\,ds=\mu^2/(2ts)\,d\mu\,d\nu$.
The inverse map has uniformly bounded derivatives of all required orders on a fixed finite cover of the parameter set. On each piece, choose a cutoff equal to one on that piece and supported in a slightly larger neighborhood, and absorb it into the amplitude. The image of the cutoff support in the $(\mu,\nu)$ coordinates lies in a fixed rectangle on which the expressions for $t^2$ and $s^2$ in \eqref{eq:inverse-sobolev-coordinates} are positive and smooth. Adding finitely many such estimates recovers the full supremum.

Write $\chi(\mu,\nu)$ for this parameter cutoff. Recall from Section~\ref{sec:angular-localization} that the amplitude of $\mathcal A^+_{\lambda,I}$ contains the angular cutoff $\chi_I$. The transformed branch operator is
\begin{equation}
\begin{aligned}
\widetilde{\mathcal A}_{\lambda,I}g(x,\mu,\nu)
&:=\chi(\mu,\nu)\mathcal A^+_{\lambda,I}g(x,t(\mu,\nu),s(\mu,\nu))\\
&=\frac{\lambda^{-1/2}}{(2\pi)^2}
\int e^{i\{x\cdot\xi+\phi(\mu,\nu,\xi)\}}
\widetilde m_I(\mu,\nu,\xi)\widehat g(\xi)\,d\xi,
\end{aligned}
\label{eq:transformed-branch}
\end{equation}
where 
\begin{align*}
\phi(\mu,\nu,\xi)&=\Phi_{t(\mu,\nu),s(\mu,\nu)}(\xi),\\
\widetilde m_I(\mu,\nu,\xi)&=\chi(\mu,\nu)\chi_I(\arg\xi)
m^+_\lambda(t(\mu,\nu),s(\mu,\nu),\xi).
\end{align*}
The frequency variable is unchanged, so the integral acquires no Jacobian. Spacetime norms change by the bounded Jacobian $\partial(\mu,\nu)/\partial(t,s)\sim1$ following \eqref{eq:inverse-sobolev-coordinates}, and the supremum on each parameter piece is preserved by the one-to-one coordinate change.

For $\xi=\lambda r e(\theta_0+\delta u)$ write
\[
D(u)=\sin^2(\theta_0+\delta u)-\sin^2\theta_0,
\]
and put $H=(\mu^2+\mu\nu D(u))^{1/2}$.
Substitution of \eqref{eq:inverse-sobolev-coordinates} gives $\phi(\mu,\nu,\xi)=\lambda rH$, and
\begin{equation}
\frac{D(u)}{b\delta}
=\frac{\sin(2\theta_0+\delta u)}{b}\frac{\sin(\delta u)}{\delta}.
\label{eq:angular-difference}
\end{equation}
All fixed-order $u$ derivatives of $D/(b\delta)$ are bounded, since $\theta_0\sim b$ and $\delta\lesssim b$. In particular $D=O(b\delta)$ rather than just $O(\delta)$.

The derivatives needed below are
\begin{equation}
(H_\mu,H_\nu,H_{\mu\nu})=\Bigl(\frac{2\mu+\nu D}{2H},\frac{\mu D}{2H},\frac{\mu\nu D^2}{4H^3}\Bigr).
\label{eq:phase-parameter-derivatives}
\end{equation}
Thus
\begin{equation}
\Bigl(\frac{\phi_\mu}{\lambda},\frac{\phi_\nu}{\lambda b\delta}\Bigr)=r\Bigl(\frac{2\mu+\nu D}{2H},\frac{\mu}{2H}\frac{D}{b\delta}\Bigr).
\label{eq:normalized-phase-derivatives}
\end{equation}
Since $H=H_{t,s}(\theta_0+\delta u)\ge\min(t,s)$, \eqref{eq:angular-difference} and the remark following it show that both quotients have uniformly bounded normalized frequency and parameter derivatives. On a pair of supports of separation scale $\delta$, the two scales in Lemma~\ref{lem:parameter-sobolev} can consequently be taken to be
\begin{equation}
(L_1,L_2)=(\lambda,\max\{1,\lambda b\delta\}).
\label{eq:sobolev-scales}
\end{equation}
For Proposition~\ref{prop:separated}, the assumptions $b\gtrsim\delta\gtrsim\lambda^{-1/2}$ give
\[
\lambda b\delta\gtrsim\lambda\delta^2\gtrsim 1.
\]
Thus $L_2\sim\lambda b\delta$, with fixed comparison constants. Below we use $(\lambda b\delta)^{-1}\partial_\nu$ to display the formulas.

\subsection{Normalized derivatives of the branch operators}\label{sec:normalized-derivatives}
We check that the normalized parameter derivatives used in
Lemma~\ref{lem:parameter-sobolev} preserve the symbol bounds.
This allows Proposition~\ref{prop:spacetime} to be applied to each
differentiated product of branches. Set
\begin{equation}
(\mathsf D_\mu,\mathsf D_\nu)=(\lambda^{-1}\partial_\mu,(\lambda b\delta)^{-1}\partial_\nu).
\label{eq:normalized-parameter-derivatives}
\end{equation}
For $\alpha,\beta\in\{0,1\}$ define
\begin{equation}
\widetilde{\mathcal A}^{[\alpha,\beta]}_{\lambda,I}g
:=\mathsf D_\mu^\alpha\mathsf D_\nu^\beta\widetilde{\mathcal A}_{\lambda,I}g
=\frac{\lambda^{-1/2}}{(2\pi)^2}\int e^{i(x\cdot\xi+\phi)}
m_I^{[\alpha,\beta]}(\mu,\nu,\xi)\widehat g(\xi)\,d\xi,
\label{eq:differentiated-branch}
\end{equation}
where
\[
m_I^{[\alpha,\beta]}
=e^{-i\phi}\mathsf D_\mu^\alpha\mathsf D_\nu^\beta(e^{i\phi}\widetilde m_I).
\]
For instance,
\[
m_I^{[1,0]}=\mathsf D_\mu\widetilde m_I+i(\mathsf D_\mu\phi)\widetilde m_I,
\]
$m_I^{[0,1]}$ is given by the same formula with $\mathsf D_\nu$ in place of $\mathsf D_\mu$, and the mixed derivative is
\begin{equation}
\begin{aligned}
m_I^{[1,1]}={}&\mathsf D_\mu\mathsf D_\nu\widetilde m_I
+i(\mathsf D_\mu\phi)\mathsf D_\nu\widetilde m_I
+i(\mathsf D_\nu\phi)\mathsf D_\mu\widetilde m_I\\
&+i(\mathsf D_\mu\mathsf D_\nu\phi)\widetilde m_I
-(\mathsf D_\mu\phi)(\mathsf D_\nu\phi)\widetilde m_I.
\end{aligned}
\label{eq:differentiated-symbols}
\end{equation}
By \eqref{eq:phase-parameter-derivatives}, the mixed phase factor satisfies
\begin{equation}
\mathsf D_\mu\mathsf D_\nu\phi
=\frac{r\mu\nu D^2}{4\lambda H^3b\delta}=O(b\delta/\lambda).
\label{eq:mixed-phase-derivative}
\end{equation}
Equations \eqref{eq:angular-difference}, \eqref{eq:normalized-phase-derivatives}, and \eqref{eq:mixed-phase-derivative} control every phase factor in \eqref{eq:differentiated-symbols}.

On $\xi=\lambda r e(\theta_0+\delta u)$, the amplitude of $\mathcal A^+_{\lambda,I}$, namely \eqref{eq:branch-symbol} multiplied by $\chi_I(\arg\xi)$, becomes
\begin{equation}
\widetilde\psi(r)(rH)^{-1/2}c_+(\lambda rH)\chi_I(\theta_0+\delta u).
\label{eq:normalized-branch-symbol}
\end{equation}
The last factor has uniformly bounded derivatives in $u$.
The quantity $rH$ lies in a fixed compact subinterval of $(0,\infty)$, and the bounds on $c_\pm$ following \eqref{eq:bessel-expansion} give
\[
\left|\frac{d^j}{dq^j}c_+(\lambda q)\right|\lesssim_j 1
\qquad(q\sim 1).
\]
The chain rule therefore gives uniformly bounded derivatives of \eqref{eq:normalized-branch-symbol} in $r,u,\mu,\nu$. This remains true after any fixed number of unnormalized parameter derivatives. Multiplying by a parameter cutoff produces no loss because $\lambda,\lambda b\delta\gtrsim 1$. Hence all the amplitudes in \eqref{eq:differentiated-branch} satisfy the normalized symbol bounds.

\subsection{Kernel estimates}\label{sec:kernel-estimates}
For a direction $\theta_c$ and $0<h\lesssim1$, let $n=e(\theta_c)$, $n^\perp=e'(\theta_c)$, and
\begin{equation}
\xi=B_{\lambda,h}(r,u):=\lambda rn+\lambda hu n^\perp,
\label{eq:cap-frequency-map}
\end{equation}
so that $|\det B_{\lambda,h}|=\lambda^2h$.
We call the image under $B_{\lambda,h}$ of a fixed bounded set of $(r,u)$, with $r$ bounded away from zero, a cap of width $h$ about $\theta_c$. For width $h\lesssim\lambda^{-1/2}$, the change of variables \eqref{eq:cap-frequency-map} reduces the kernel to a Schwartz function with seminorms uniform in the scales; larger supports are covered by such caps. The $L^1$ bound gives the bilinear $L^\infty$ endpoint; the remaining bounds will be used for Proposition~\ref{prop:short}.

\begin{lemma}\label{lem:kernel}
Let $(t,s)\in\mathbb I^2$, and let $K_{t,s,\lambda,I}$ be the convolution kernel of $\mathcal A^+_{\lambda,I}$ with $\chi_I$ replaced by a smooth angular cutoff supported in an interval of length comparable to $h$, whose $k$-th derivatives are $O(h^{-k})$. The following hold uniformly in $(t,s)$.
\begin{enumerate}\addtolength{\itemindent}{-8pt}\addtolength{\leftskip}{-8pt}
\renewcommand{\theenumi}{\textbf{\alph{enumi}}}\renewcommand{\labelenumi}{(\theenumi)}
\item\label{lem:kernel-large} If $h\gtrsim\lambda^{-1/2}$, then $\|K_{t,s,\lambda,I}\|_1\lesssim h$ and $\|K_{t,s,\lambda,I}\|_\infty\lesssim\lambda$.
\item\label{lem:kernel-small} If $0<h\le C_0\lambda^{-1/2}$, then $\|K_{t,s,\lambda,I}\|_1\lesssim_{C_0}\lambda^{-1/2}$ and $\|K_{t,s,\lambda,I}\|_\infty\lesssim_{C_0}\lambda$.
\item\label{lem:kernel-amplitude} The $L^1$ bounds in \eqref{lem:kernel-large} and \eqref{lem:kernel-small} remain valid if $m^+_\lambda\chi_I$ is replaced by an amplitude with the same support whose seminorms in the coordinates $(r,u)$ of \eqref{eq:cap-frequency-map} are uniformly bounded on each of $O(1+h\lambda^{1/2})$ caps of width at most $C_0\lambda^{-1/2}$ covering the support. The prefactor $\lambda^{-1/2}$ in \eqref{eq:branch-operator} is retained.
\item\label{lem:kernel-derivative} In particular, \eqref{lem:kernel-amplitude} applies to the amplitudes $m_I^{[\alpha,\beta]}$ of \eqref{eq:differentiated-branch} when the scales $\lambda$ and $\lambda b\delta$ in \eqref{eq:normalized-parameter-derivatives} are replaced by
\begin{itemize}
\item $\lambda$ and $\max\{1,\lambda bh\}$, for a support of width $h\lesssim b$ in $I_b$;
\item $\lambda$ and $1$, for a support in $[0,C_0\lambda^{-1/2}]$.
\end{itemize}
\end{enumerate}
\end{lemma}

\begin{proof}
First take $0<h\leq C_0\lambda^{-1/2}$ and use \eqref{eq:cap-frequency-map} with $\theta_c$ the center of the angular support.
Work in a fixed short angular chart, using a bounded number of such charts if necessary. On the annular support, $r=(|\xi|/\lambda)\cos(\arg\xi-\theta_c)$ is then bounded away from zero, while $u$ stays bounded. By homogeneity and Taylor's formula in the transverse direction,
\[
\Phi_{t,s}(B_{\lambda,h}(r,u))
=\lambda rH_{t,s}(\theta_c)+\lambda huH'_{t,s}(\theta_c)+R_{t,s,\lambda,h}(r,u),
\]
\[
R_{t,s,\lambda,h}(r,u)
=\lambda h^2u^2\int_0^1(1-\tau)
\bigl\langle D_\xi^2\Phi_{t,s}(rn+\tau hu n^\perp)n^\perp,n^\perp\bigr\rangle\,d\tau.
\]
This integral remainder is exact. Since $\lambda h^2\leq C_0^2$,
\[
\|R_{t,s,\lambda,h}\|_{C^N_{r,u}}\leq C_{N,C_0}.
\]
Define
\[
q_{t,s}=\nabla\Phi_{t,s}(n)=H_{t,s}(\theta_c)n+H'_{t,s}(\theta_c)n^\perp,
\]
and let
\[
\widetilde m_{t,s,\lambda,h}(r,u)
=e^{iR_{t,s,\lambda,h}(r,u)}m^+_\lambda(t,s,B_{\lambda,h}(r,u))
\chi_I(\arg B_{\lambda,h}(r,u)).
\]
This is a smooth amplitude on a fixed compact set, with uniform seminorms. For the angular cutoff, the relevant normalized argument is
\[
\frac{\arg B_{\lambda,h}(r,u)-\theta_c}{h}
=\frac{\arctan(hu/r)}{h}
=\frac ur\int_0^1\frac{d\tau}{1+\tau^2h^2u^2/r^2}.
\]
Thus its fixed-order $(r,u)$ derivatives remain bounded even as $h\to 0$. Let $\mathcal K_{t,s,\lambda,h}$ be its inverse Fourier transform in $(r,u)$, with the $(2\pi)^{-2}$ normalization. It is Schwartz, with uniform seminorms. The transformed convolution kernel is exactly
\begin{equation}
K_{t,s,\lambda,I}(y)
=\lambda^{3/2}h\mathcal K_{t,s,\lambda,h}\bigl(B_{\lambda,h}^{\mathsf T}(y+q_{t,s})\bigr).
\label{eq:kernel-rescaling}
\end{equation}
The prefactor is $\lambda^{-1/2}|\det B_{\lambda,h}|$. The sign of the translation is consistent with the positive phase: the kernel is concentrated near $-q_{t,s}$.

From \eqref{eq:kernel-rescaling},
\[
\begin{aligned}
|K_{t,s,\lambda,I}(y)|\leq C_{N,C_0}\lambda^{3/2}h\bigl(1
&+\lambda|(y+q_{t,s})\cdot n|\\
&+\lambda h|(y+q_{t,s})\cdot n^\perp|\bigr)^{-N}.
\end{aligned}
\]
Integration gives $\|K\|_1\lesssim_{C_0}\lambda^{-1/2}$, while $\|K\|_\infty\lesssim_{C_0}\lambda^{3/2}h\lesssim_{C_0}\lambda$. In particular, at $h=\lambda^{-1/2}$ this is a kernel of height $O(\lambda)$ and dimensions $\lambda^{-1}$ by $\lambda^{-1/2}$.

For $h\geq\lambda^{-1/2}$, decompose the smooth angular cutoff into $O(1+h\lambda^{1/2})$ pieces of width $\lambda^{-1/2}$ whose normalized derivatives are uniformly bounded. Adding their $L^1$ norms gives $O(h+\lambda^{-1/2})=O(h)$. To estimate the supremum norm, note that their centers $-q_{t,s}$ are separated in arclength by comparable multiples of $\lambda^{-1/2}$ along an ellipse of bounded positive curvature. Indeed, as the center angle varies, $q_{t,s}(\theta)=(t^2\cos\theta,s^2\sin\theta)/H_{t,s}(\theta)$ moves with velocity
\[
q'_{t,s}(\theta)=\frac{t^2s^2}{H_{t,s}(\theta)^3}e'(\theta),
\]
whose coefficient is bounded above and below. On each fixed short chart, projection onto its central tangent shows that chord distance is comparable to angular distance, uniformly in $t,s\in\mathbb I$. Each envelope is bounded by
\[
C_N\lambda(1+\lambda^{1/2}|y+q_{t,s}|)^{-N}.
\]
The number of such centers in a ball of radius $r\lambda^{-1/2}$ is $O(1+r)$. Summing the envelopes over dyadic annuli centered at $y$ therefore gives an $O(\lambda)$ bound when $N>2$.

The $L^1$ argument uses only the support and uniform amplitude seminorms in the cap coordinates, so it also proves the stated extension. For the normalized differentiated amplitudes in \eqref{eq:differentiated-branch}, it remains to check these seminorms after subdivision. With the scales in the statement, these are $\lambda$ and $\max\{1,\lambda bh\}$ for an off-axis support of width $h\lesssim b$, and $\lambda$ and $1$ for the axis cap used below. The maximum prevents a loss when a parameter derivative falls on a cutoff and $\lambda bh<1$. These scales and the original coordinates $(\mu,\nu)$ are retained when the frequency support is subdivided. In the off-axis case, for a subcap of width $h'\leq h$ centered at $\theta_c$, its normalized angular coordinate in the original support is
\[
\frac{\theta_c-\theta_0}{h}+\frac1h\arctan\frac{h'u}{r}.
\]
Since $|\theta_c-\theta_0|\lesssim h$ and $h'/h\leq 1$, this coordinate and all required $(r,u)$ derivatives are uniformly bounded.

For an axis cap, $\theta_0=0$ and $|\theta_c|+h'\lesssim_{C_0}\lambda^{-1/2}$. Under the linear frequency change \eqref{eq:cap-frequency-map}, the angle is $\theta=\theta_c+\arctan(h'u/r)$, so all fixed-order $(r,u)$ derivatives of $\lambda\sin^2\theta$ are uniformly bounded. Since
\[
H_{t,s}(\theta)^2=\mu^2+\mu\nu\sin^2\theta,
\]
the normalized phase derivatives with parameter scales $\lambda$ and $1$ have the same uniform bounds.

The normalized differentiated amplitudes therefore retain their uniform seminorms after the linear frequency change, and the same proof gives their $L^1$ bounds. This proves the lemma.
\end{proof}

\subsection{Proof of Proposition~\ref{prop:separated}}\label{sec:separated-proof}
We first combine the spacetime estimate with
Lemma~\ref{lem:parameter-sobolev} to obtain the $L^2$ endpoint,
and then interpolate with the $L^\infty$ endpoint obtained
from the kernel $L^1$ bound.

\subsubsection*{The $L^2$ endpoint}
Work on a fixed parameter coordinate neighborhood.
To apply Proposition~\ref{prop:spacetime}, write the differentiated product as a single oscillatory integral. Let
\[
G(x,\mu,\nu)=\widetilde{\mathcal A}_{\lambda,I}g_1(x,\mu,\nu)
\widetilde{\mathcal A}_{\lambda,I'}g_2(x,\mu,\nu),
\]
and $F(\xi,\eta)=\widehat g_1(\xi)\widehat g_2(\eta)$.
For $\alpha,\beta\in\{0,1\}$, using the normalized derivatives
in \eqref{eq:normalized-parameter-derivatives}, we have
\begin{equation}
\mathsf D_\mu^\alpha\mathsf D_\nu^\beta G(x,\mu,\nu)
=\frac{\lambda^{-1}}{(2\pi)^4}
\mathcal T_{n_{\alpha,\beta}}F(x,t(\mu,\nu),s(\mu,\nu)),
\label{eq:differentiated-product}
\end{equation}
where $\mathcal T_n$ has exactly the phase in \eqref{eq:bilinear-phase} and its amplitude, written in the $(\mu,\nu)$ coordinates, is
\[
n_{\alpha,\beta}
=e^{-i(\phi(\xi)+\phi(\eta))}\mathsf D_\mu^\alpha\mathsf D_\nu^\beta
\bigl[e^{i(\phi(\xi)+\phi(\eta))}\widetilde m_I(\xi)\widetilde m_{I'}(\eta)\bigr].
\]
The same formulas as \eqref{eq:differentiated-symbols}, now with the sum of the two phases, show that $n_{\alpha,\beta}$ satisfies \eqref{eq:symbol-bounds}. Its pullback is supported in \eqref{eq:normalized-support}, since the supports of $\chi_I$ and $\chi_{I'}$ lie within $O(\delta)$ of $\theta_0$ and are at distance comparable to $\delta$ by the hypotheses of Proposition~\ref{prop:separated}. Parameter differentiation does not enlarge the angular supports, which lie in $[9b/10,21b/10]$ by the definition of $\chi_I$ in Section~\ref{sec:angular-localization}. The coordinate change between $(\mu,\nu)$ and $(t,s)$ has uniformly
bounded derivatives on the fixed coordinate neighborhoods. Hence
$n_{\alpha,\beta}$ satisfies the symbol assumptions of
Proposition~\ref{prop:spacetime}.

In \eqref{eq:differentiated-product}, the prefactor is $\lambda^{-1}(2\pi)^{-4}$, while the stated Fourier convention gives
\[
\|F\|_{L^2_{\xi,\eta}}=\|\widehat g_1\|_2\|\widehat g_2\|_2
=(2\pi)^2\|g_1\|_2\|g_2\|_2.
\]
Substituting these into \eqref{eq:spacetime-bound} gives
\begin{equation}\label{eq:bilinear-L2}
\|\mathsf D_\mu^\alpha\mathsf D_\nu^\beta G\|_{L^2_{x,\mu,\nu}}
\lesssim\lambda^{-1}b^{-1/2}\delta^{-2}\|g_1\|_2\|g_2\|_2.
\end{equation}
The constant includes the bounded parameter Jacobian $\partial(\mu,\nu)/\partial(t,s)\sim1$ following \eqref{eq:inverse-sobolev-coordinates}. This is the spacetime bound used in the maximal $L^2$ estimate.

Since $\mathcal B_{\lambda,I,I'}$ in \eqref{eq:bilinear-maximal} is a supremum of products evaluated at a common parameter, the parameter Sobolev inequality applies to the product $G$. Lemma~\ref{lem:parameter-sobolev}, with $r=2$ and the scales in \eqref{eq:sobolev-scales}, gives at each $x$, recalling the normalized derivatives \eqref{eq:normalized-parameter-derivatives},
\[
\sup_{\mu,\nu}|G(x,\mu,\nu)|
\lesssim\lambda(b\delta)^{1/2}\sum_{\alpha,\beta\in\{0,1\}}
\|\mathsf D_\mu^\alpha\mathsf D_\nu^\beta G(x,\cdot,\cdot)\|_{L^2_{\mu,\nu}}.
\]
Take $L_x^2$, use Fubini, and apply \eqref{eq:bilinear-L2} to each of the four terms. The resulting factor is
\[
\lambda(b\delta)^{1/2}\lambda^{-1}b^{-1/2}\delta^{-2}=\delta^{-3/2}.
\]
After summing over the fixed finite parameter cover, this proves
\begin{equation}\label{eq:bilinear-L2-result}
\|\mathcal B_{\lambda,I,I'}(g_1,g_2)\|_2
\lesssim\delta^{-3/2}\|g_1\|_2\|g_2\|_2.
\end{equation}

\subsubsection*{The $L^\infty$ endpoint}
The other endpoint comes directly from Lemma~\ref{lem:kernel}. By \eqref{eq:angular-separation}, $|I|,|I'|\sim\delta\gtrsim\lambda^{-1/2}$, so each kernel has $L^1$ norm $O(\delta)$ by Lemma~\ref{lem:kernel}\eqref{lem:kernel-large}, uniformly in $(t,s)$. Young's inequality bounds each factor by $C\delta\|g_i\|_\infty$, and hence
\begin{equation}
\|\mathcal B_{\lambda,I,I'}(g_1,g_2)\|_\infty
\lesssim\delta^2\|g_1\|_\infty\|g_2\|_\infty.
\label{eq:bilinear-Linfty}
\end{equation}

\subsubsection*{Interpolation}
The operator $\mathcal B_{\lambda,I,I'}$ is a supremum and hence not bilinear, so we linearize it in order to interpolate between \eqref{eq:bilinear-L2-result} and \eqref{eq:bilinear-Linfty}. For fixed $\lambda,I,I'$ the kernels are Schwartz functions depending continuously on $(t,s)$. Hence the products in \eqref{eq:bilinear-maximal} are continuous in $(t,s)$, and the supremum may be taken over a countable dense subset of $\mathbb I^2$. For measurable $t,s:\mathbb R^2\to\mathbb I$, set
\[
B^{t,s}(g_1,g_2)(x)
=\mathcal A^+_{\lambda,I}g_1(x,t(x),s(x))
\mathcal A^+_{\lambda,I'}g_2(x,t(x),s(x)).
\]
This operator is bilinear, and by pointwise domination it satisfies \eqref{eq:bilinear-L2-result} and \eqref{eq:bilinear-Linfty} uniformly in the selection. Bilinear interpolation, with weight $2/r$ on the $L^2$ endpoint, gives for $2\leq r\leq\infty$
\begin{equation}
\|B^{t,s}(g_1,g_2)\|_r
\lesssim\delta^{-\frac32\frac2r+2(1-\frac2r)}\|g_1\|_r\|g_2\|_r
=\delta^{2-7/r}\|g_1\|_r\|g_2\|_r.
\label{eq:bilinear-interpolation}
\end{equation}
For given $g_1,g_2$, choosing at each $x$ the first of finitely many parameters that maximizes the absolute value of the product gives a measurable selection. These finite maxima increase pointwise to $\mathcal B_{\lambda,I,I'}(g_1,g_2)$. Since \eqref{eq:bilinear-interpolation} is uniform in the selection, monotone convergence gives the same bound for $\mathcal B_{\lambda,I,I'}$.

At $r=7/2$ the exponent $2-7/r$ vanishes, which is \eqref{eq:separated-maximal}. Thus the loss $\delta^{-3/2}$ at the $L^2$ endpoint, coming from the bound $\lambda^{-1}b^{-1/2}\delta^{-2}$ in \eqref{eq:bilinear-L2} and the Sobolev factor $\lambda(b\delta)^{1/2}$, is exactly balanced at $r=7/2$ by the gain $\delta^2$ at the $L^\infty$ endpoint. This completes the proof of Proposition~\ref{prop:separated}.\qed

\subsection{Proof of Proposition~\ref{prop:short}}\label{sec:short-proof}
For a short cap, the differentiated kernel $L^1$ bounds and the parameter Sobolev inequality give a linear $L^6$ estimate. We interpolate this with the $L^1\to L^\infty$ estimate supplied by the kernel supremum bound. The axis case has a smaller Sobolev cost because its second parameter derivative scale is $1$.

\subsubsection*{Short caps off the axis}
Let $h=|I|$. Use \eqref{eq:transformed-branch} with reference angle at the center of $I$, and use \eqref{eq:normalized-parameter-derivatives}--\eqref{eq:differentiated-branch} with $\delta=h$. Since $I\subset I_b$ and $\lambda^{-1/2}\leq|I|\leq C_0\lambda^{-1/2}$, we have $\lambda^{-1/2}\leq h\leq b$, so $\lambda bh\geq\lambda h^2\geq 1$; the two parameter scales are therefore $\lambda$ and $\lambda bh$.

The branch and each normalized parameter derivative have convolution kernel $L^1$ norm $O_{C_0}(\lambda^{-1/2})$ by Lemma~\ref{lem:kernel}\eqref{lem:kernel-small}--\eqref{lem:kernel-derivative}. Consequently, for $\alpha,\beta\in\{0,1\}$, Young's inequality at each parameter and integration over a fixed parameter rectangle give
\[
\|\widetilde{\mathcal A}^{[\alpha,\beta]}_{\lambda,I}g\|_{L^6_{x,\mu,\nu}}
\lesssim_{C_0}\lambda^{-1/2}\|g\|_6.
\]
Applying Lemma~\ref{lem:parameter-sobolev} first at each $x$ gives
\[
\Big\|\sup_{(t,s)\in\mathbb I^2}|\mathcal A^+_{\lambda,I}g|\Big\|_6
\lesssim_{C_0}(\lambda^2bh)^{1/6}\lambda^{-1/2}\|g\|_6
\lesssim_{C_0}\lambda^{-1/4}b^{1/6}\|g\|_6.
\]
By Lemma~\ref{lem:kernel}\eqref{lem:kernel-small}, the kernel supremum bound gives the other endpoint
\[
\Big\|\sup_{(t,s)\in\mathbb I^2}|\mathcal A^+_{\lambda,I}g|\Big\|_\infty
\lesssim_{C_0}\lambda\|g\|_1.
\]
Fix a measurable parameter selection before interpolating the resulting linear operator. Both endpoint bounds are uniform in this selection. With weights $1/7$ and $6/7$, the input and output exponents satisfy
\[
\left(\frac27,\frac17\right)
=\frac17(1,0)+\frac67\left(\frac16,\frac16\right).
\]
Parameter continuity and the finite-maxima argument in
Section~\ref{sec:separated-proof} restore the supremum. We obtain
\[
\begin{aligned}
\Big\|\sup_{(t,s)\in\mathbb I^2}|\mathcal A^+_{\lambda,I}g|\Big\|_7
&\lesssim_{C_0}\lambda^{1/7}(\lambda^{-1/4}b^{1/6})^{6/7}\|g\|_{7/2}\\
&=\lambda^{-1/14}b^{1/7}\|g\|_{7/2}.
\end{aligned}
\]
This is \eqref{eq:short-cap-maximal}, the first assertion of Proposition~\ref{prop:short}.

\subsubsection*{The axis cap}
For the axis assertion let $I\subset[0,C_0\lambda^{-1/2}]$ and use \eqref{eq:sobolev-coordinates} with $\theta_0=0$. In this case
\[
H_{t,s}(\theta)^2=\mu^2+\mu\nu\sin^2\theta,
\]
where $\mu=t$ and $\nu=(s^2-t^2)/t$.
On the support, $\sin^2\theta=O_{C_0}(\lambda^{-1})$. Write $\xi=\lambda r e(\theta)$ and $H=(\mu^2+\mu\nu\sin^2\theta)^{1/2}$. The phase derivatives are
\[
\Bigl(\frac{\phi_\mu}{\lambda},\phi_\nu\Bigr)
=r\Bigl(\frac{2\mu+\nu\sin^2\theta}{2H},\ \lambda\frac{\mu\sin^2\theta}{2H}\Bigr)=O_{C_0}(1).
\]
These bounds also hold after fixed numbers of normalized frequency and parameter derivatives. In particular, on a cap of width $h\lesssim_{C_0}\lambda^{-1/2}$ centered at $\theta_c$, writing $\theta=\theta_c+hu$ leaves all derivatives of $\lambda\sin^2(\theta_c+hu)$ uniformly bounded. The two parameter scales are therefore $\lambda$ and $1$, up to $C_0$-dependent constants. Thus the normalized derivatives here are $\lambda^{-1}\partial_\mu$ and $\partial_\nu$. If $I$ is smaller than $\lambda^{-1/2}$, use its own width in the linear map \eqref{eq:cap-frequency-map}; the kernel proof is uniform in that width. The operator in \eqref{eq:transformed-branch} and these derivatives then have the same kernel $L^1$ bound as before, by Lemma~\ref{lem:kernel}\eqref{lem:kernel-small}--\eqref{lem:kernel-derivative}, and hence
\[
\Big\|\sup_{(t,s)\in\mathbb I^2}|\mathcal A^+_{\lambda,I}g|\Big\|_6
\lesssim_{C_0}\lambda^{1/6}\lambda^{-1/2}\|g\|_6
=\lambda^{-1/3}\|g\|_6.
\]
The same interpolation and supremum argument gives
\[
\Big\|\sup_{(t,s)\in\mathbb I^2}|\mathcal A^+_{\lambda,I}g|\Big\|_7
\lesssim_{C_0}\lambda^{1/7}(\lambda^{-1/3})^{6/7}\|g\|_{7/2}
=\lambda^{-1/7}\|g\|_{7/2}.
\]
This proves \eqref{eq:axis-cap-maximal} and completes the proof of Proposition~\ref{prop:short}, and with it the sufficiency part of Theorem~\ref{thm:main}.\qed

\section{Necessary conditions for the centered operator}\label{sec:necessary-conditions}
We prove the necessity assertion of Theorem~\ref{thm:main}. We first test the estimate on a large disc, a thin neighborhood of an ellipse,
and a thin strip around a circular arc. These are two-parameter versions
of the standard examples for the circular maximal function~\cite{Lee,Schlag,SS},
and they give the three sides of the triangle $\Delta$ in \eqref{eq:triangle}. A smooth
cusp construction then rules out weak type for $p\le3$. For the global
operator, the failure of strong type for $p\le3$ was noted in~\cite{LLOelliptic};
the cusp construction gives the weak type statement for the local
operator $\M$. For $t,s>0$ write
$E_{t,s}=\{(t\cos\theta,s\sin\theta):\theta\in[0,2\pi)\}$ for the
centered ellipse with semiaxes $t,s$. The letters $\delta$ and $h$ and
the Cartesian coordinates used in these constructions are local to
this section.

\subsubsection*{The condition $1/q\le1/p$}
For $f_R=\mathbf1_{B(0,R)}$, $R\ge4$, we have
$\supp\sigma_{t,s}\subset B(0,2)$ for $(t,s)\in\I^2$, so $\M f_R=1$
on $B(0,R-2)$. Letting $R\to\infty$ in $R^{2/q}\lesssim R^{2/p}$
gives $1/q\le1/p$.

\subsubsection*{The condition $1/p\le2/q$}
Fix $(t_0,s_0)\in\I^2$, let $E=E_{t_0,s_0}$, and let $f_\delta$
be the indicator of the $\delta$-neighborhood of $E$. Its support has
area comparable to $\delta$. If $|x|\le\delta$, then $x+E$ lies in
this neighborhood, so $\M f_\delta(x)\ge\A_{t_0,s_0}f_\delta(x)=1$.
Hence $\delta^{2/q}\lesssim\delta^{1/p}$ as $\delta\downarrow0$,
which gives $1/p\le2/q$.

\subsubsection*{The condition $4/p\le1+1/q$: a thin strip around a circular arc}
Let $0<\delta\ll1$ and $h=\delta^{1/3}$. The test function is the
indicator of the $\delta$-neighborhood of an arc of length comparable
to $h$ on a fixed circle, so its support has area comparable to
$\delta h=h^4$. We find a set of centers of area comparable to $h$
such that, for each center $x$, some ellipse $x+E_{t,s}$ contains an
arc of length comparable to $h$ in this neighborhood. Then $\M f\gtrsim h$
on that set, and comparing norms gives $h^{1+1/q}\lesssim h^{4/p}$.
To arrange this, we translate the ellipses so that they share a point
and tangent, and restrict their curvatures at that point to an interval
of length comparable to $h$. On arcs of length comparable to $h$, the
second-order deviation of the curves is then $O(h\cdot h^2)$, and the
third-order deviation is $O(h^3)$.

\subsubsection*{Step 1: a common point and tangent}
Fix $r_0\in(1,2)$ and restrict $(t,s)$ to a small fixed neighborhood
of $(r_0,r_0)$. The upper half of $E_{t,s}$ is the graph
\[
 y_2=\psi_{t,s}(y_1),
\]
where $\psi_{t,s}(a)=s\sqrt{1-a^2/t^2}$.
We align the points where the slope of this graph is $-1$; they stay
away from the axes and from the vertical tangents of the ellipses. Put
\[
 \bigl(u(t,s),v(t,s)\bigr)=\frac{(t^2,s^2)}{\sqrt{t^2+s^2}}.
\]
A direct substitution and differentiation give
\[
 (\psi_{t,s},\psi'_{t,s},-\psi''_{t,s})(u(t,s))=\bigl(v(t,s),-1,\mathfrak c(t,s)\bigr),
\]
where $\mathfrak c(t,s)=(t^2+s^2)^{3/2}/(t^2s^2)$.
In particular $\mathfrak c_0:=\mathfrak c(r_0,r_0)=2\sqrt2/r_0$.

Let $u_0=v_0=r_0/\sqrt2$, so that
$(u_0,v_0)=(u(r_0,r_0),v(r_0,r_0))$ is the corresponding point on
the circle of radius $r_0$. Choose the center
\begin{equation}\label{eq:test-centers}
 x(t,s)=(u_0-u(t,s),\ v_0-v(t,s)).
\end{equation}
After translating $E_{t,s}$ by $x(t,s)$, the point $(u(t,s),v(t,s))$
moves to $(u_0,v_0)$. Writing $w$ for the horizontal displacement from
this common point, the spatially shifted ellipse $x(t,s)+E_{t,s}$
has the local graph $y_1=u_0+w$, $y_2=g_{t,s}(w)$, where
\begin{equation}\label{eq:translated-ellipse}
 g_{t,s}(w)=v_0-v(t,s)+\psi_{t,s}(u(t,s)+w).
\end{equation}
Thus every translated graph has the same point and first derivative:
\[
 (g_{t,s},g'_{t,s},g''_{t,s})(0)=\bigl(v_0,-1,-\mathfrak c(t,s)\bigr).
\]
Since $x(r_0,r_0)=0$, the reference graph $g_0=g_{r_0,r_0}$ is the
upper half of the circle of radius $r_0$:
$g_0(w)=\sqrt{r_0^2-(u_0+w)^2}$. Since
$t^2-u(t,s)^2=t^2s^2/(t^2+s^2)>1/2$ and $u(t,s)<t<2$, we have
\begin{equation}\label{eq:third-derivative-bound}
 t^2-(u(t,s)+w)^2\ge\frac12-4|w|-w^2>\frac14
 \qquad(|w|\le\bar w),
\end{equation}
where $\bar w=1/20$.
Hence the third derivatives of $g_{t,s}$ are uniformly bounded on
$[-\bar w,\bar w]$, and Taylor's formula gives, with a fixed constant
$C_3$,
\begin{equation}\label{eq:contact-expansion}
 \begin{aligned}
 g_{t,s}(w)-g_0(w)&=-\frac{\mathfrak c(t,s)-\mathfrak c_0}{2}w^2+R_{t,s}(w),\\
 |R_{t,s}(w)|&\le C_3|w|^3\qquad(|w|\le\bar w).
 \end{aligned}
\end{equation}

\subsubsection*{Step 2: the parameter set and its centers}
Since $\partial_s \mathfrak c(r_0,r_0)=-\sqrt2/r_0^2\ne0$, we may choose a
small open rectangle $\mathcal Q\Subset(1,2)^2$ about $(r_0,r_0)$ on
which $T(t,s)=(t,\mathfrak c(t,s))$ is a diffeomorphism onto its image, with
Jacobian $|\partial_s\mathfrak c|\sim1$, and whose image contains a fixed
rectangle about $(r_0,\mathfrak c_0)$. Define
\[
 \Pi_h=\{(t,s)\in\mathcal Q:|\mathfrak c(t,s)-\mathfrak c_0|\le h\}.
\]
For small $h$, $T(\Pi_h)$ is the part of $T(\mathcal Q)$ in the band
$|\mathfrak c-\mathfrak c_0|\le h$. It contains a rectangle of fixed length and width
$2h$, and it lies in a strip of the same width over a bounded
$t$-interval. Hence $|\Pi_h|\sim h$.

Next, the centers $x(t,s)$, $(t,s)\in\Pi_h$, fill a set of comparable
area. The output Jacobian is
\begin{equation}\label{eq:center-jacobian}
 \frac{\partial(x_1,x_2)}{\partial(t,s)}
 =\frac{\partial(u,v)}{\partial(t,s)}
 =\frac{2ts}{t^2+s^2}\sim1.
\end{equation}
Since $u+v=\sqrt{t^2+s^2}$, we have $t^2=u(u+v)$ and
$s^2=v(u+v)$. Thus $(t,s)$ is determined by $(u,v)$, and the maps
$(t,s)\mapsto(u,v)$ and $(t,s)\mapsto x(t,s)$ are injective.
Consequently
\begin{equation}\label{eq:center-measure}
 |x(\Pi_h)|\sim h.
\end{equation}

\subsubsection*{Step 3: the test function and the lower bound}
Choose $C_2\ge\frac12+C_3$ and take $\delta$ so small that
$h\le\bar w$. Define
\begin{equation}\label{eq:test-neighborhood}
 F_\delta=\{(u_0+w,y_2):|w|\le h,\ |y_2-g_0(w)|\le C_2\delta\},
\end{equation}
and $f_\delta=\mathbf1_{F_\delta}$, so that $|F_\delta|=4C_2\delta h\sim\delta h$.
Thus $F_\delta$ is the $C_2\delta$-neighborhood, in the vertical
direction, of an arc of length comparable to $h$ on the circle of
radius $r_0$. For $(t,s)\in\Pi_h$ and $|w|\le h$, \eqref{eq:contact-expansion} gives
\[
 |g_{t,s}(w)-g_0(w)|\le\tfrac12h^3+C_3h^3
 \le C_2h^3=C_2\delta.
\]
Hence the arc $\{(u_0+w,g_{t,s}(w)):|w|\le h\}$ of
$x(t,s)+E_{t,s}$ lies in $F_\delta$.

By the definition of $\sigma_{t,s}$ and its symmetry,
\[
 \A_{t,s}f(x)=\int f(x+y)\,d\sigma_{t,s}(y)
 =\frac1{2\pi}\int_0^{2\pi}f(x+(t\cos\theta,s\sin\theta))\,d\theta.
\]
On the upper half $0<\theta<\pi$ we substitute $a=t\cos\theta$, so
that $s\sin\theta=\psi_{t,s}(a)$. With $a=u(t,s)+w$, the pushforward
of normalized angular measure has density
\begin{equation}\label{eq:angular-density}
 \frac{da}{2\pi\sqrt{t^2-a^2}}
 =\frac{dw}{2\pi\sqrt{t^2-(u(t,s)+w)^2}}.
\end{equation}
Since $t\le2$, this density is at least $1/(4\pi)$; by the
choice of $\bar w$ it is also bounded above for $|w|\le\bar w$.
By \eqref{eq:test-centers} and \eqref{eq:translated-ellipse},
\[
 x(t,s)+(u(t,s)+w,\ \psi_{t,s}(u(t,s)+w))
 =(u_0+w,\ g_{t,s}(w)).
\]
Let $(t,s)\in\Pi_h$. Since $f_\delta\ge0$, we may discard the lower
half of the ellipse and the part of the upper half with $|w|>h$;
on the remaining arc $f_\delta=1$, because this arc lies in
$F_\delta$. Therefore
\begin{equation}\label{eq:test-lower-bound}
 \begin{aligned}
 \M f_\delta(x(t,s))
 &\ge\frac1{2\pi}\int_0^\pi
       f_\delta(x(t,s)+(t\cos\theta,s\sin\theta))\,d\theta\\
 &=\int_{-t}^t f_\delta(x(t,s)+(a,\psi_{t,s}(a)))
          \frac{da}{2\pi\sqrt{t^2-a^2}}\\
 &\ge\int_{-h}^{h}f_\delta(u_0+w,g_{t,s}(w))
          \frac{dw}{2\pi\sqrt{t^2-(u(t,s)+w)^2}}\\
 &=\int_{-h}^{h}\frac{dw}{2\pi\sqrt{t^2-(u(t,s)+w)^2}}
       \ge\frac{h}{2\pi}.
 \end{aligned}
\end{equation}
In the last two lines, $|w|\le h\le\bar w$, so by \eqref{eq:third-derivative-bound} the point
$a=u(t,s)+w$ lies in $(-t,t)$ and the integrand is well defined; the
final inequality uses the lower bound $1/(4\pi)$ for the density.

\subsubsection*{Step 4: conclusion}
If \eqref{eq:main-bound} holds, then \eqref{eq:center-measure}, the bound $|F_\delta|\sim\delta h$ following \eqref{eq:test-neighborhood}, and \eqref{eq:test-lower-bound} give
\[
 h\cdot h^{1/q}\lesssim\|\M f_\delta\|_{L^q(x(\Pi_h))}
 \le\|\M f_\delta\|_q\lesssim\|f_\delta\|_p
 =|F_\delta|^{1/p}\sim(\delta h)^{1/p}=h^{4/p},
\]
with $h^{1/q}=1$ when $q=\infty$. Letting $h\to0$ gives
\begin{equation}\label{eq:necessary-edge}
 \frac4p\le1+\frac1q.
\end{equation}

\subsubsection*{The triangle $\Delta$}
If \eqref{eq:main-bound} is assumed only for Schwartz functions, the
test functions above may be replaced by smooth functions with supports of comparable measure and
the same lower bounds up to fixed constants. The resulting half-planes
$1/q\le1/p$, $1/p\le2/q$, and \eqref{eq:necessary-edge} intersect in the triangle
$\Delta$ of \eqref{eq:triangle}: the first two lines meet at $(0,0)$, the first and
third at $(1/3,1/3)$, and the second and third at $(2/7,1/7)$. It
remains to show that $p>3$ is necessary, even for weak type.

\subsubsection*{Failure of weak type for $p\le3$}
We now use the common tangent without restricting the second
derivative. All the spatially shifted arcs lie in a quadratic cusp near their
common point. A logarithmic weight gives bounded $L^3$ input norms
while its integral along every arc diverges.

Choose a fixed closed rectangle $\mathcal Q_0\Subset\mathcal Q$ of
positive area. Use the same centers $x(t,s)$ and translated graphs
$g_{t,s}$ as in \eqref{eq:test-centers}--\eqref{eq:translated-ellipse}. The injectivity and Jacobian calculation
in \eqref{eq:center-jacobian} show that $\Omega=x(\mathcal Q_0)$ has fixed positive area.
Since all these graphs share their value and first derivative at zero,
there are fixed constants $C_1$ and $0<w_0<e^{-2}$ such that
\[
 |g_{t,s}(w)-(v_0-w)|\le C_1w^2,\qquad
 (t,s)\in\mathcal Q_0,\quad0<w<w_0,
\]
and, taking $w_0\le\bar w$, these arcs lie on the graph piece where
the density \eqref{eq:angular-density} is bounded above and below. All choices are
independent of the truncation below.

Take smooth functions $0\le\chi_1,\chi_2,\chi_3\le1$ such that
$\chi_1\in C_c^\infty((-w_0,w_0))$ equals one on $[0,w_0/2]$,
$\chi_2$ vanishes on $(-\infty,1]$ and equals one on $[2,\infty)$, and
$\chi_3\in C_c^\infty(\R)$ equals one on $[-C_1,C_1]$.
For $0<\nu<w_0/4$ define, when $0<w<w_0$,
\begin{equation}\label{eq:cusp-test-function}
 f_\nu(u_0+w,y_2)
 =\frac{\chi_1(w)\chi_2(w/\nu)\chi_3((y_2-v_0+w)/w^2)}
        {w\sqrt{\log(1/w)}},
\end{equation}
and extend by zero elsewhere. Each $f_\nu$ is nonnegative and in
$C_c^\infty(\R^2)$, since it vanishes near $w=0$ and $w=w_0$.
Their supports lie in one fixed bounded set. Since the $y_2$-section
of the support has length $\lesssim w^2$, integration first in $y_2$
gives
\[
 \sup_\nu\|f_\nu\|_3^3
 \lesssim\int_0^{w_0}\frac{dw}{w(\log(1/w))^{3/2}}
 =\frac2{\sqrt{\log(1/w_0)}}<\infty.
\]
Let $(t,s)\in\mathcal Q_0$ and $2\nu\le w\le w_0/2$. Then
$\chi_1(w)=1$ and $\chi_2(w/\nu)=1$, and the bound
$|g_{t,s}(w)-(v_0-w)|\le C_1w^2$ gives
$\chi_3((g_{t,s}(w)-v_0+w)/w^2)=1$. Hence
\[
 f_\nu(u_0+w,g_{t,s}(w))=\frac1{w\sqrt{\log(1/w)}}
 \qquad(2\nu\le w\le w_0/2).
\]
We now argue as for \eqref{eq:test-lower-bound}. Since $f_\nu\ge0$, we may discard the
lower half of the ellipse and the part of the upper half outside
$2\nu\le w\le w_0/2$. Since $w_0\le\bar w$, \eqref{eq:third-derivative-bound} shows that
$a=u(t,s)+w$ lies in $(-t,t)$ there, and the density \eqref{eq:angular-density} is at
least $1/(4\pi)$. Therefore
\[
 \begin{aligned}
 \M f_\nu(x(t,s))
 &\ge\frac1{4\pi}\int_{2\nu}^{w_0/2}\frac{dw}{w\sqrt{\log(1/w)}}\\
 &=\frac1{2\pi}\bigl(\sqrt{\log(1/(2\nu))}-\sqrt{\log(2/w_0)}\bigr).
 \end{aligned}
\]
The right-hand side tends to infinity as $\nu\to0$, uniformly for
$(t,s)\in\mathcal Q_0$, that is, uniformly on $\Omega$. Since $\Omega$
has fixed positive area, the $L^{q,\infty}$ norms diverge for every
finite $q$, and the essential supremum diverges when $q=\infty$.
The inputs have one fixed bounded support, so H\"older's inequality
also gives uniform $L^p$ norms for $1\le p\le3$. This proves the
stated failure of weak type and completes the necessity assertion
of Theorem~\ref{thm:main}.

\section{Diagonal estimates for translated circles}\label{sec:translated-diagonal-estimates}
We prove the diagonal estimate needed for Theorem~\ref{thm:translated}
for every fixed $c\in\R^2$, including $c=0$.
We first state the incidence estimate of Pramanik--Yang--Zahl~\cite{PYZ},
which we call the PYZ estimate,
and a maximal estimate for selected curves, together with a
determinant criterion for its hypotheses. We then verify the
hypotheses for the two slices from~\cite[Section~3.1]{CGY}. These slices
suffice when $|c|<1$. When $|c|\ge1$, they may degenerate
simultaneously at finitely many points, the common zeros of
Section~\ref{sec:app-slices}, and neighborhoods of these points require
rescaling (Section~\ref{sec:app-common-zeros}). Frequency localization and summation
give the following estimate. All constants may depend on the fixed
vector $c$. Throughout, $P_\lambda$ and $\M^c_\lambda$ are as in
Section~\ref{sec:frequency-reduction}.

\begin{proposition}\label{prop:translated-pyz-cinematic}
Fix $c\in\R^2$.  For every $3\le r<\infty$, every $\varepsilon>0$, and
every dyadic $\lambda\ge2$,
\begin{equation}\label{eq:app-diagonal-loss}
 \|\M^c_\lambda f\|_r
 \lesssim_{c,r,\varepsilon}\lambda^\varepsilon\|f\|_r.
\end{equation}
\end{proposition}

Only the case $r=3$ is needed for Theorem~\ref{thm:translated}.
Combining it with local smoothing gives the diagonal frequency decay
in Section~\ref{sec:translated-diagonal}.

\subsection{The PYZ estimate and selected curves}

Let $J$ be a fixed compact interval, and consider real-valued functions
$\phi\in C^2(J)$, identified with their graphs
$\{(v,\phi(v)):v\in J\}\subset\R^2$.  Following~\cite[Definition~1.6]{PYZ}, a family $\mathcal F\subset C^2(J)$ is
cinematic with constants $K,D\ge1$ if its $C^2$ diameter is at most $K$,
it is a doubling metric space with doubling constant at most $D$ (every
ball in $\mathcal F$ of radius $r$ is covered by at most $D$ balls of
radius $r/2$), and
\begin{equation}\label{eq:app-cinematic-metric}
 \inf_{v\in J}\bigl(|\phi(v)-\psi(v)|
       +|\phi'(v)-\psi'(v)|+|\phi''(v)-\psi''(v)|\bigr)
 \ge K^{-1}\|\phi-\psi\|_{C^2(J)}
\end{equation}
for all $\phi,\psi\in\mathcal F$.

We use~\cite[Theorem~1.7, $\alpha=\zeta=1$, and (1.9)]{PYZ}.
For every $\varepsilon>0$ and all sufficiently small $\delta>0$,
depending on $K,D,\varepsilon$, a finite set $F\subset\mathcal F$
satisfying
\[
 \#(F\cap B)\le\delta^{-\varepsilon}\rho/\delta
\]
for every ball $B\subset C^2(J)$ of radius $\rho\ge\delta$ obeys
\begin{equation}\label{eq:app-pyz-incidence}
 \Big\|\sum_{\phi\in F}\mathbf1_{\phi^\delta}\Big\|_{3/2}
 \lesssim\delta^{-C\varepsilon}(\delta\#F)^{2/3},
\end{equation}
where $C$ depends only on $D$.
For bounded graph patches this is the special case~\cite[(1.9)]{PYZ}, after a fixed
normalization.  The theorem uses Euclidean neighborhoods of the graphs;
we use the vertical neighborhoods
\[
 \phi^\delta=\{(v,z):v\in J,\ |z-\phi(v)|<\delta\},
\]
which are contained in them, so \eqref{eq:app-pyz-incidence} also holds
with this convention.  The normalizations below have uniformly bounded
$C^2$ norms and fixed graph intervals.

The following estimate is deduced from \eqref{eq:app-pyz-incidence}
by weighted duality in Appendix~\ref{app:selected-curves}.

\begin{lemma}\label{lem:app-selected-curves}
Let $\mathcal F\subset C^2(J)$ be a compact cinematic family with
uniformly bounded $C^2$ norms.  Assign to each curve a real output
coordinate $\eta(\phi)$, ranging over a bounded interval, and suppose
\begin{equation}\label{eq:app-output-projection}
 |\eta(\phi)-\eta(\psi)|\le C\|\phi-\psi\|_{C^2(J)}.
\end{equation}
For $0<\delta<1$, define
\[
 T_\delta g(\eta)=
 \sup_{\substack{\phi\in\mathcal F\\\eta(\phi)=\eta}}
 \frac1{2\delta}\int_J\int_{-\delta}^{\delta}
 |g(v,\phi(v)-z)|\,dz\,dv,
\]
with value zero when the set of curves is empty.  Then, for every
$3\le r<\infty$ and every $\varepsilon>0$,
\begin{equation}\label{eq:app-selected-strong}
 \|T_\delta g\|_{L^r(d\eta)}
 \lesssim_{r,\varepsilon}\delta^{-\varepsilon}\|g\|_r.
\end{equation}
The constants depend only on the cinematic constants, the stated
bounds, $J$, and an upper bound for the length of the output interval.
\end{lemma}

The output measure in \eqref{eq:app-selected-strong} is unnormalized;
the estimate remains uniform when the output interval is shorter
than $\delta$. In the proof, given in Appendix~\ref{app:selected-curves},
the Lipschitz condition \eqref{eq:app-output-projection} gives the
ball-count hypothesis preceding \eqref{eq:app-pyz-incidence} for curves
selected at $\delta$-separated output values.

In our applications the hypotheses of Lemma~\ref{lem:app-selected-curves}
are verified by the following lemma, which reduces them to a
determinant condition.

\begin{lemma}[Nondegenerate parameter families]\label{lem:app-nondegenerate}
Let $J\subset\R$ be a compact interval, let $B\subset\R^3$ be a compact
convex set, and let $(u,p)\mapsto\phi_p(u)$ be a $C^4$ function on a
neighborhood of $J\times B$, with all derivatives of order at most $4$
bounded by $A$. Suppose that the parameter $2$-jet map
\[
 \mathcal J(u,p)=\bigl(\phi_p(u),\partial_u\phi_p(u),\partial_u^2\phi_p(u)\bigr)
\]
satisfies $|\det D_p\mathcal J(u,p)|\ge c_0>0$ on $J\times B$. Then
$J\times B$ is the union of finitely many products $J'\times B'$, where
$J'$ is a closed interval and $B'=B\cap Q$ for a closed box
$Q\subset\R^3$, such that, with
$\mathcal F=\{\phi_p|_{J'}:p\in B'\}$, the following hold on each piece.
\begin{enumerate}\addtolength{\itemindent}{-8pt}\addtolength{\leftskip}{-8pt}
\renewcommand{\theenumi}{\textbf{\alph{enumi}}}\renewcommand{\labelenumi}{(\theenumi)}
\item\label{lem:nd-a} $c|p-q|\le\|\phi_p-\phi_q\|_{C^2(J')}\le C|p-q|$ for $p,q\in B'$.
 In particular, $p\mapsto\phi_p|_{J'}$ is injective, and $\mathcal F$ is
 a compact doubling subset of $C^2(J')$ of diameter at most $C$.
\item\label{lem:nd-b} \eqref{eq:app-cinematic-metric} holds for all $\phi,\psi\in\mathcal F$,
 with $J'$ in place of $J$ and $K=C$. Thus $\mathcal F$ is a cinematic
 family.
\item\label{lem:nd-c} If $\eta$ is a function on $B'$ with $|\eta(p)-\eta(q)|\le L|p-q|$,
 then $\eta(\phi_p):=\eta(p)$ satisfies \eqref{eq:app-output-projection}
 with the constant $CL$.
\end{enumerate}
The number of pieces and the constants $c,C$ depend only on $A$, $c_0$,
and the diameters of $J$ and $B$.
\end{lemma}

\begin{proof}
Subdivide $J$ into closed intervals $J'$ and intersect $B$ with a grid
of closed boxes $Q$, with the interval lengths and box diameters to be
chosen. For each nonempty piece $J'\times B'$, where $B'=B\cap Q$,
choose $(u_0,p_0)\in J'\times B'$. The entries of
$D_p\mathcal J$ are bounded by $A$ and its determinant is at least
$c_0$ in absolute value, so $M_0=D_p\mathcal J(u_0,p_0)$ satisfies
$\|M_0^{-1}\|\le C(A,c_0)$ by the adjugate formula. The derivatives of
$D_p\mathcal J$ are bounded by $A$, so if the size of the piece is at
most $c(A,c_0)$, then
$\|D_p\mathcal J(u,p)-M_0\|\le(2\|M_0^{-1}\|)^{-1}$ on $J'\times B'$.
Since $B'$ is convex, the mean value theorem gives, for $u\in J'$ and
$p,q\in B'$,
\begin{equation}\label{eq:app-jet-bilipschitz}
 (2\|M_0^{-1}\|)^{-1}|p-q|\le|\mathcal J(u,p)-\mathcal J(u,q)|\le C|p-q|.
\end{equation}

For \eqref{lem:nd-a}, the mean value theorem in $p$ gives
$\|\phi_p-\phi_q\|_{C^2(J')}\le CA|p-q|$, while
$|\mathcal J(u,p)-\mathcal J(u,q)|\le\sqrt3\,\|\phi_p-\phi_q\|_{C^2(J')}$
and the lower bound in \eqref{eq:app-jet-bilipschitz} give
$c|p-q|\le\|\phi_p-\phi_q\|_{C^2(J')}$. Hence $p\mapsto\phi_p|_{J'}$ is
injective and bi-Lipschitz from $B'$ onto $\mathcal F$, so $\mathcal F$
is compact, doubling with a constant depending only on $A$ and $c_0$,
and of diameter at most $CA\operatorname{diam}B$.

Assertion \eqref{lem:nd-b} follows from the lower bound in
\eqref{eq:app-jet-bilipschitz} and \eqref{lem:nd-a} since,  for every
$u\in J'$,
\[
 |\phi_p-\phi_q|(u)+|\partial_u(\phi_p-\phi_q)|(u)
 +|\partial_u^2(\phi_p-\phi_q)|(u)
 \ge|\mathcal J(u,p)-\mathcal J(u,q)|.
\]

For \eqref{lem:nd-c}, assertion \eqref{lem:nd-a} gives
$|\eta(\phi_p)-\eta(\phi_q)|\le L|p-q|\le CL\|\phi_p-\phi_q\|_{C^2(J')}$,
which is \eqref{eq:app-output-projection}.
\end{proof}

\subsection{Two slices and the case \texorpdfstring{$|c|<1$}{|c|<1}}
\label{sec:app-slices}
For $c=0$, Chen--Guo--Yang~\cite[Section~3.1]{CGY} used two
slices to verify the nondegeneracy required for applying the PYZ
estimate. We use the same slices for general $c$ and verify the
hypotheses of Lemma~\ref{lem:app-selected-curves} on regular arcs,
that is, arcs separated from the common zeros defined below, where
both parameter $2$-jet determinants vanish. There are no common zeros
when $|c|<1$.

\subsubsection*{The family of curves}
Consider an upper graph patch of the curve
\[
 \{(t(c_1+\cos\theta),s(c_2+\sin\theta)):\theta\in\R\},
 \qquad (t,s)\in\I^2.
\]
Write
\begin{equation}\label{eq:app-graph}
 F_t(u)=c_2+\sqrt{1-(u/t-c_1)^2},
\end{equation}
and put $z=u/t-c_1$ and $w=\sqrt{1-z^2}$.
We work on a compact subinterval of $|z|<1$.  Reflections and interchange
of the coordinate axes give finitely many such patches covering the
circle.  The averaging density on this patch is
\[
 \frac{du}{2\pi t\sqrt{1-(u/t-c_1)^2}},
\]
which is smooth and uniformly bounded on the patch.

\subsubsection*{The first slice}
First fix the horizontal output coordinate $x$.  With the displacement
$u$ as the horizontal integration variable, the input is evaluated at
$(x-u,y-sF_t(u))$.  The graph family for the remaining output coordinate is
\[
 \phi_1(u;y,t,s)=y-sF_t(u).
\]
The parameter $2$-jet map sends $(y,t,s)$ to the value and first two
$u$-derivatives of the graph.  Its determinant is
\begin{equation}\label{eq:app-first-slicing}
 \det D_{(y,t,s)}(\phi_1,\partial_u\phi_1,\partial_u^2\phi_1)
 =\frac{s}{t^4w^6}P_{c_1}(z),
 \end{equation}
 where
 \[  P_{c_1}(z)=c_1(3z^2-1)+2z^3.\]
Indeed, let $G=-sF_t$ and $\mathbf G=(G_0,G_1,G_2)$ with $G_k=\partial_u^kG$.
With primes denoting $z$-derivatives,
\[
 \mathbf G=-s\bigl(c_2+w,\ t^{-1}w',\ t^{-2}w''\bigr).
\]
Then $(\phi_1,\partial_u\phi_1,\partial_u^2\phi_1)=(y,0,0)+\mathbf G$, so expanding
along the $y$-column $(1,0,0)^{\mathsf T}$ gives
\[
 \det D_{(y,t,s)}(\phi_1,\partial_u\phi_1,\partial_u^2\phi_1)
 =\partial_tG_1\,\partial_sG_2-\partial_sG_1\,\partial_tG_2.
\]
Since $\partial_tz=-(z+c_1)/t$,
\[
 \partial_tG_k=st^{-k-1}\bigl(kw^{(k)}+(z+c_1)w^{(k+1)}\bigr)\quad(k=0,1,2),
\]
while $\partial_s\mathbf G=\mathbf G/s$.
Hence this determinant equals
$(s/t^4)\bigl[w'w''+(z+c_1)(w'w'''-(w'')^2)\bigr]$. Inserting
$w'=-z/w$, $w''=-1/w^3$, and $w'''=-3z/w^5$ gives
\[
 \frac{s}{t^4w^6}\bigl[z(1-z^2)+(z+c_1)(3z^2-1)\bigr]
 =\frac{s}{t^4w^6}P_{c_1}(z).
\]

\subsubsection*{The second slice}
For the second slice, fix $y-x$ and use the input horizontal coordinate
$u$, so that the displacement is $\widetilde u=x-u$.  Up to a fixed
vertical translation, the graph family is
$\phi_2(u;x,t,s)=x-sF_t(\widetilde u)$, with output coordinate $x$;
this is the spatial shear used in~\cite[(3.10)--(3.13)]{CGY}. In this
slice, $z=\widetilde u/t-c_1$ and $w=\sqrt{1-z^2}$, as defined after
\eqref{eq:app-graph} with $u$ replaced by $\widetilde u$. Its
determinant is
\begin{equation}\label{eq:app-second-slicing}
 \det D_{(x,t,s)}(\phi_2,\partial_u\phi_2,\partial_u^2\phi_2)
 =-\frac{s}{t^5w^6}\bigl(tP_{c_1}(z)+sQ_{c_2}(z)\bigr), 
 \end{equation}
 where 
 \[
 Q_{c_2}(z)=c_2(2-3z^2)+2(1-z^2)^{3/2}.
\]
Indeed, define $G$ and $\mathbf G$ as above with $u$ replaced by $\widetilde u$,
so that $G_k=\partial_{\widetilde u}^kG$. Then
$(\phi_2,\partial_u\phi_2,\partial_u^2\phi_2)=(x+G_0,-G_1,G_2)$. At fixed
$u$, $\partial_x$ acts on $G$ as $\partial_{\widetilde u}$, while
$\partial_t$ and $\partial_s$ act at fixed $\widetilde u$. Changing the
sign of the middle row gives
\[
 \det D_{(x,t,s)}(\phi_2,\partial_u\phi_2,\partial_u^2\phi_2)
 =-\det\bigl(e_1+\partial_{\widetilde u}\mathbf G,\partial_t\mathbf G,\partial_s\mathbf G\bigr),
\]
where $e_1=(1,0,0)^{\mathsf T}$.
The part containing $e_1$ is the determinant \eqref{eq:app-first-slicing},
namely $s P_{c_1}(z)/(t^4w^6)$. In the remaining part, adding
$(z+c_1)\partial_{\widetilde u}\mathbf G$ to $\partial_t\mathbf G$ gives
$(s/t)(0,w'/t,2w''/t^2)^{\mathsf T}$, while
$\partial_{\widetilde u}\mathbf G=-(s/t)(w',w''/t,w'''/t^2)^{\mathsf T}$ and
$\partial_s\mathbf G=-(c_2+w,w'/t,w''/t^2)^{\mathsf T}$. Hence
\[
 \det\bigl(\partial_{\widetilde u}\mathbf G,\partial_t\mathbf G,\partial_s\mathbf G\bigr)
 =\frac{s^2}{t^5}
 \det\begin{pmatrix}
 w'&0&c_2+w\\ w''&w'&w'\\ w'''&2w''&w''
 \end{pmatrix}
 =\frac{s^2}{t^5w^6}Q_{c_2}(z).
\]
Adding the two parts and using the sign in the previous display gives
\eqref{eq:app-second-slicing}.

The factor $P_{c_1}$ determines
where the first parameter $2$-jet map is nondegenerate.  At a zero of
$P_{c_1}$, the second determinant reduces to
$-s^2Q_{c_2}(z)/(t^5w^6)$.  Thus $Q_{c_2}$ determines whether the
second parameter $2$-jet map is nondegenerate where the first one fails.
Nonvanishing of either determinant gives the cinematic condition
\eqref{eq:app-cinematic-metric} after localization, as verified below.

\subsubsection*{Common zeros}
For fixed $c$, a \emph{common zero} means a point $z_0\in(-1,1)$ such
that $P_{c_1}(z_0)=Q_{c_2}(z_0)=0$.  At the corresponding point of
$c+\mathbb S^1$, both parameter $2$-jet determinants vanish for every
$(t,s)\in\I^2$.  We treat neighborhoods of these common zeros by
rescaling in Section~\ref{sec:app-common-zeros}.

At a common zero $z_0$, we have $z_0^2\notin\{1/3,2/3\}$.
Solving the two equations gives
\begin{equation}\label{eq:app-common-zero-center}
 (c_1,c_2)=\Bigl(\frac{2z_0^3}{1-3z_0^2},\ -\frac{2(1-z_0^2)^{3/2}}{2-3z_0^2}\Bigr).
\end{equation}
In particular,
\begin{equation}\label{eq:app-common-zero}
 |c|^2-1=
 \frac{3z_0^4(1-z_0^2)^2}
 {(1-3z_0^2)^2(2-3z_0^2)^2}\ge0.
\end{equation}
Thus there are no common zeros when $|c|<1$.  For an arbitrary fixed
$c$ there is at most one. Indeed, the map $z_0\mapsto(c_1,c_2)$ in
\eqref{eq:app-common-zero-center} is injective on $(-1,1)$: $c_1$
is odd and increasing on each of the three intervals separated by
$\pm1/\sqrt3$, while $c_2$ is even and takes values in
$(-2(2/3)^{3/2},-1]$ for $|z_0|<1/\sqrt3$ and outside this interval
for $|z_0|>1/\sqrt3$. Moreover, $c_1\in(-\infty,-1)$ on $(1/\sqrt3,1)$
and $c_1\in(1,\infty)$ on $(-1,-1/\sqrt3)$.

\subsubsection*{Regular arcs}
On a regular arc, that is, an arc separated from the common zeros, one
of the determinants \eqref{eq:app-first-slicing},
\eqref{eq:app-second-slicing} is bounded away from zero, and we use the
corresponding slice, with parameters $p=(y,t,s)$ or $p=(x,t,s)$ and
with output coordinate the first entry of $p$. A translation of the
output coordinate changes the graph only by a vertical translation,
together with the horizontal translation $u\mapsto u-x_0$ in the second
slice; we may therefore localize the output coordinate to unit intervals
centered at zero, so that $p$ ranges over a compact convex set. Cover
the arc by finitely many products $J\times B$ of a $u$-interval and a
parameter box on which $w$ is bounded below. On each of them
$(u,p)\mapsto\phi_p(u)$ is smooth with bounded derivatives, so
Lemma~\ref{lem:app-nondegenerate} applies, with $\eta$ the first entry
of $p$ and $L=1$. Lemma~\ref{lem:app-selected-curves} then applies to
each resulting piece, and the fixed output coordinate is subsequently
integrated, as in~\cite[Section~3.1]{CGY}.

\subsubsection*{The case \texorpdfstring{$|c|<1$}{|c|<1}}
In the case $|c|<1$, \eqref{eq:app-common-zero} excludes all common
zeros.  Compactness gives finitely many of the regular patches just
described, and hence the required estimates for thickened curves on
the whole circle.  This includes $c=0$: then $P_0(z)=2z^3$ and
$Q_0(z)=2(1-z^2)^{3/2}$ have no common zero on $|z|<1$.
The passage to frequency-localized averages is given in
Section~\ref{sec:app-frequency-sum}.

\subsection{The case \texorpdfstring{$|c|\ge1$}{|c|>=1}}
\label{sec:app-common-zeros}
For $|c|\ge1$, Section~\ref{sec:app-slices} treats the regular arcs.
Near any common zeros, we decompose the arc into dyadic pieces
according to the distance $h$ to the zero, rescale each piece, and
verify the hypotheses of Lemma~\ref{lem:app-selected-curves} for the
resulting family of each piece, with constants independent of $h$. The
cutoffs defining the pieces are fixed, and the estimates for the pieces
are summed, in Section~\ref{sec:app-frequency-sum}.

\subsubsection*{Persistence of the degeneracy}
Changing the scalar output coordinate does not remove the degeneracy at
a common zero. Indeed, at a common zero $z_0$, substitution of
\eqref{eq:app-common-zero-center} shows that the vector field
$V=t(1-3z_0^2)\partial_t+s(2-3z_0^2)\partial_s$ in the $(t,s)$-plane
satisfies
\[
 \left.
 V(\phi_1,\partial_u\phi_1,\partial_u^2\phi_1)
 \right|_{u=(c_1+z_0)t}=0,
\]
where the parameter derivatives are taken at fixed $u$ before
evaluation. Since $V$ has no $y$-component, the same vanishing holds
for the graphs $y-sF_t(x-v)$ of the second slice at fixed $v$, which
differ from $\phi_1$ by a change of the output coordinate only. Thus
$V$ is a nonzero vector in the kernel of the parameter $2$-jet map for
every choice of the free output coordinate, and the $2$-jet map
restricted to any level set of that coordinate remains singular.

At an axis vertex we therefore use an anisotropic dilation, and at
other common zeros we construct a uniformly nondegenerate basis from
divided differences of graphs.

\subsubsection*{Axis vertices}
Since $P_{c_1}(0)=-c_1$ and $Q_{c_2}(0)=2(c_2+1)$, the common zero
$z_0=0$ forces $c=(0,-1)$; then the upper vertex of $c+\mathbb S^1$ is
the origin, which is fixed by the axis dilations. The other axis cases,
in which another vertex of $c+\mathbb S^1$ is the origin, follow by
reflection or interchange of coordinates.  Here the graph is
\[
 s\bigl(-1+\sqrt{1-u^2/t^2}\bigr).
\]
Decompose a neighborhood of the vertex into smooth dyadic pieces
$|u|\sim h$.  Under the spatial dilation
$(x,y)=(h\widetilde x,h^2\widetilde y)$, write $u=hv$ and put
\[
 (a,\beta)=(s/t^2,\ h^2/t^2).
\]
The rescaled graph and density are
\begin{equation}\label{eq:app-axis-rescaling}
 -\frac{a v^2}{1+\sqrt{1-\beta v^2}},\qquad
 \frac{h}{2\pi t}(1-\beta v^2)^{-1/2}\,dv,
 \qquad |v|\sim1.
\end{equation}
The parameter $a$ stays in a fixed compact subinterval of $(0,\infty)$,
and $\beta\sim h^2$.  For the original parameters $t\in[1,2]$, the
density in \eqref{eq:app-axis-rescaling} is bounded by $Ch\,dv$.
By positivity, we first replace the density by this bound and then
enlarge the range of $\beta$ to a fixed interval $[-\bar\beta,\bar\beta]$,
where $\bar\beta>0$ is small.
The rescaled graph extends smoothly to this larger parameter set.

For the sliced family
\[
 \phi(v;\widetilde y,a,\beta)
 =\widetilde y+\frac{a v^2}{1+\sqrt{1-\beta v^2}},
\]
the $2$-jet determinant at $\beta=0$ is $a v^3$.  Indeed, the derivatives
with respect to $a$ and $\beta$ at $\beta=0$ are $v^2/2$ and $a v^4/8$,
and the $\widetilde y$-column of the jet matrix is $(1,0,0)^{\mathsf T}$,
so the determinant is
\[
 \det\begin{pmatrix} v & a v^3/2\\ 1 & 3a v^2/2\end{pmatrix}=a v^3.
\]
Taking the fixed $\beta$-range small makes this determinant uniformly
nonzero for $|v|\sim1$. The family is smooth with bounded derivatives
on the fixed parameter set. After a fixed finite subdivision of the
graph interval and parameter set, Lemma~\ref{lem:app-nondegenerate},
with output $\widetilde y$, verifies the hypotheses of
Lemma~\ref{lem:app-selected-curves} on each piece, uniformly in $h$.

\subsubsection*{Common zeros away from the axes}
Suppose $0<|z_0|<1$, and put 
\[ w_0=\sqrt{1-z_0^2}, \quad  \alpha=c_1+z_0.\]
  Inserting \eqref{eq:app-common-zero-center}, namely
$\alpha=z_0+2z_0^3/(1-3z_0^2)$, $c_2+w_0=w_0-2w_0^3/(2-3z_0^2)$, and
$P_{c_1}'(z_0)=6z_0(c_1+z_0)=6z_0\alpha$, we obtain
\begin{equation}\label{eq:app-critical-data}
 \begin{aligned}
 \alpha&=\frac{z_0w_0^2}{1-3z_0^2}\ne0,\\
 F_t(\alpha t)&=c_2+w_0=-\frac{w_0z_0^2}{2-3z_0^2}\ne0,\\
 P_{c_1}'(z_0)&=\frac{6z_0^2w_0^2}{1-3z_0^2}\ne0.
 \end{aligned}
\end{equation}
Also $\partial_uF_t(\alpha t)=-z_0/(tw_0)\ne0$.

Decompose the arc into signed dyadic pieces
$|u/t-c_1-z_0|\sim h$, equivalently $|u-\alpha t|\sim h$.
Divide the bounded $u$-range into intervals of length $\eta_0h$ with
bounded overlap, where $\eta_0>0$ is fixed and small.  Since
$\alpha\ne0$, each interval receives parameters in a $t$-interval of
length $O(h)$.  A bounded number of subdivisions, independent of $h$,
gives rectangles on whose slight enlargements $u-\alpha t$ has one
sign and magnitude comparable to $h$.

In the following lemma the implicit constants in $|u-\alpha t|\sim h$
are those of the dyadic decomposition above, and all derivatives and
norms in $v$ are taken on $J_0=[-\eta_0/2,\eta_0/2]$.

\begin{lemma}[Graphs near a nonaxis common zero]\label{lem:app-degenerate-graphs}
Let $z_0$ be a common zero with $0<|z_0|<1$, and let $\alpha=c_1+z_0$
be as in \eqref{eq:app-critical-data}, so that $u=\alpha t$ corresponds
to $z=z_0$. There are constants $\eta_0\in(0,1)$ and $h_0>0$, depending
only on $c$ and $z_0$, with the following property. Let $0<h\le h_0$,
let $I$ be an interval of length $\eta_0h$ with center $u_*$, and let
$T\subset\I$ be a closed interval of length at most $\eta_0h$ such that
$u-\alpha t$ has one sign and $|u-\alpha t|\sim h$ on $I\times T$.
For $p=(y,t,s)$ in $\mathcal P=[-\tfrac12,\tfrac12]\times T\times\I$,
define
\[
 G_p(v)=y-sF_t(u_*+hv),\qquad v\in J_0.
\]
Then $\mathcal F_{I,T}=\{G_p:p\in\mathcal P\}$ is a compact cinematic
family in $C^2(J_0)$ whose cinematic constant $K$ and doubling constant
$D$ depend only on $c$ and $z_0$. Moreover, for $p=(y,t,s)$ and
$p'=(y',\tau,s')$ in $\mathcal P$,
\begin{equation}\label{eq:app-degenerate-output}
 |y-y'|\le Ch^{-1}\|G_p-G_{p'}\|_{C^2(J_0)},
\end{equation}
with $C$ depending only on $c$ and $z_0$. In particular, all constants
are independent of $h$, $I$, and $T$.
\end{lemma}

\begin{proof}
Throughout, constants depend only on $c$ and $z_0$, and on $\eta_0$
once it has been fixed. We first record the properties of $F_t$ on
$I\times T$, then construct a basis of graph differences adapted to the
degeneracy, and finally recover the output coordinate.

\subsubsection*{Preliminaries}
For $(u,\rho)\in I\times T$, let $z=u/\rho-c_1$ and $w=\sqrt{1-z^2}$
as defined after \eqref{eq:app-graph}. Then $z-z_0=(u-\alpha\rho)/\rho$, so
$z-z_0$ has one sign and $|z-z_0|\sim h$. Since $P_{c_1}(z_0)=0$ and
$P_{c_1}'(z_0)\ne0$ by \eqref{eq:app-critical-data}, for $h_0$
sufficiently small $P_{c_1}(z)$ has one sign and $|P_{c_1}(z)|\sim h$
on $I\times T$. Moreover, for $h_0$ sufficiently small we have
$|z_0|/2\le|z|\le(1+|z_0|)/2$ on $I\times T$, so that $|z|\sim1$ and
$w=\sqrt{1-z^2}\sim1$ with constants depending only on $z_0$; hence
$\partial_uF_\rho(u)=-z/(\rho w)$ is bounded away from zero, and all
derivatives of $F_\rho(u)$ in $(u,\rho)$ of fixed order are bounded for
the same reason. Since $\partial_u^2F_\rho/\partial_uF_\rho=1/(\rho zw^2)$
and $\partial_\rho z=-(z+c_1)/\rho$ at fixed $u$, direct differentiation
gives the Wronskian identity
\begin{equation}\label{eq:app-wronskian}
 \partial_\rho\left(\frac{\partial_u^2F_\rho}{\partial_uF_\rho}\right)
 =-\frac{P_{c_1}(z)}{\rho^4w^6(\partial_uF_\rho)^2}.
\end{equation}
Since the denominator on the right is positive and bounded above and
below by the preliminary bounds, the right-hand side has the sign of
$-P_{c_1}(z)$, which is constant on $I\times T$, and magnitude
comparable to $h$.

\subsubsection*{The divided-difference basis}
For $t,\tau\in T$ with $t\ne\tau$, define
\[
  \Gamma_{t,\tau}(u)=F_\tau(u)-aF_t(u)-b,
\]
where
\begin{equation}\label{eq:app-ab}
 (a,b)=\Bigl(\frac{\partial_uF_\tau(u_*)}{\partial_uF_t(u_*)},\ F_\tau(u_*)-aF_t(u_*)\Bigr).
\end{equation}
Then $\Gamma_{t,\tau}(u_*)=\partial_u\Gamma_{t,\tau}(u_*)=0$ by
\eqref{eq:app-ab}. Moreover
$a\,\partial_u^2F_t(u_*)=\partial_uF_\tau(u_*)\,(\partial_u^2F_t/\partial_uF_t)(u_*)$,
so the fundamental theorem of calculus gives the exact identity
\[
 \partial_u^2\Gamma_{t,\tau}(u_*)
 =\partial_uF_\tau(u_*)
   \int_t^\tau\partial_\rho
   \left(\frac{\partial_u^2F_\rho(u_*)}{\partial_uF_\rho(u_*)}\right)
   \,d\rho.
\]
By \eqref{eq:app-wronskian} the integrand has one sign and magnitude
comparable to $h$ on $T$, so the integral has no cancellation and
\begin{equation}\label{eq:app-second-difference}
 |\partial_u^2\Gamma_{t,\tau}(u_*)|\sim h|\tau-t|.
\end{equation}
Since $\partial_uF_\tau(u_*)-\partial_uF_t(u_*)=O(|\tau-t|)$ and
$\partial_uF_t(u_*)$ is bounded below, $a=1+O(|\tau-t|)$. Hence
$\partial_u^3\Gamma_{t,\tau}=(\partial_u^3F_\tau-\partial_u^3F_t)
+(1-a)\partial_u^3F_t=O(|\tau-t|)$ on $I$, by the bounds on the
derivatives of $F_\rho$. On the rescaled
interval $J_0$, define
\begin{equation}\label{eq:app-divided-basis}
 \bigl(L_t(v),B_{t,\tau}(v)\bigr)=\Bigl(\frac{F_t(u_*+hv)-F_t(u_*)}{h},\ \frac{\Gamma_{t,\tau}(u_*+hv)}{h^3(\tau-t)}\Bigr).
\end{equation}
Then $L_t(0)=0$, $L_t'(v)=\partial_uF_t(u_*+hv)$, and
$L_t''(v)=h\,\partial_u^2F_t(u_*+hv)$, so
\[
 |L_t'|\sim1,
\]
while $|L_t|\lesssim\eta_0$ and $|L_t''|\lesssim h$, by the bounds on $\partial_uF_t$ and $\partial_u^2F_t$ in the
preliminaries. Since $\Gamma_{t,\tau}(u_*)=\partial_u\Gamma_{t,\tau}(u_*)=0$,
Taylor's formula for $\Gamma_{t,\tau}$ at $u_*$, with the bounds
\eqref{eq:app-second-difference} and
$|\partial_u^3\Gamma_{t,\tau}|\lesssim|\tau-t|$, gives
\[
 B_{t,\tau}(v)=\frac{\partial_u^2\Gamma_{t,\tau}(u_*)}{2h(\tau-t)}v^2
 +O(|v|^3).
\]
Writing the remainder as $R(v)$, we have
\[
 |R^{(j)}(v)|\le C|v|^{3-j}\quad(j=0,1,2),
 \qquad \|R'''\|_{L^\infty(J_0)}\le C,
\]
with $C$ independent of $h,t,\tau$.
By \eqref{eq:app-second-difference} the coefficient of
$v^2$ is comparable to one, so choosing $\eta_0$ small enough gives
\[
 |B_{t,\tau}''|\gtrsim1,
\]
while $\|B_{t,\tau}\|_{C^2(J_0)}+\|B_{t,\tau}'''\|_{L^\infty(J_0)}\lesssim1$ and $|B_{t,\tau}'|\lesssim\eta_0$.
Fix $\eta_0$ in this way. Then
\begin{equation}\label{eq:app-basis-determinant}
 |L_t'B_{t,\tau}''-L_t''B_{t,\tau}'|\gtrsim1
\end{equation}
on $J_0$ for $h\le h_0$, after decreasing $h_0$ if necessary, since the
subtracted term is $O(h\eta_0)$.

\subsubsection*{Expansion of graph differences}
Let $p=(y,t,s)$ and $p'=(y',\tau,s')$ be in $\mathcal P$ with
$t\ne\tau$. Writing $F_\tau=aF_t+b+\Gamma_{t,\tau}$ and
$F_t(u_*+hv)=F_t(u_*)+hL_t(v)$, we obtain
\begin{equation}\label{eq:app-difference-expansion}
 G_p-G_{p'}=\beta_0+\beta_1L_t+\beta_2B_{t,\tau},
\end{equation}
with
\begin{equation}\label{eq:app-beta}
 \begin{aligned}
 \beta_0&=y-y'+(s'a-s)F_t(u_*)+s'b,\\
 \beta_1&=h(s'a-s),\\
 \beta_2&=s'h^3(\tau-t).
 \end{aligned}
\end{equation}
When $t=\tau$, the same holds with $a=1$, $b=0$, and $\beta_2=0$.

\subsubsection*{The cinematic condition and the doubling bound}
Let $f=\beta_0+\beta_1L_t+\beta_2B_{t,\tau}$. At every $v\in J_0$,
\[
 \bigl(f,f',f''\bigr)(v)
 =\beta_0(1,0,0)+\beta_1(L_t,L_t',L_t'')(v)
 +\beta_2(B_{t,\tau},B_{t,\tau}',B_{t,\tau}'')(v).
\]
The matrix of this linear map has bounded entries and determinant
$L_t'B_{t,\tau}''-L_t''B_{t,\tau}'$, which is bounded below by
\eqref{eq:app-basis-determinant}. Hence
$|\beta|\lesssim|f(v)|+|f'(v)|+|f''(v)|$ for every $v\in J_0$, while
the bounds on the basis give $\|f\|_{C^2(J_0)}\lesssim|\beta|$. Thus
\[
 \inf_{v\in J_0}\bigl(|f(v)|+|f'(v)|+|f''(v)|\bigr)
 \ge K^{-1}\|f\|_{C^2(J_0)}
\]
for some $K$ depending only on $c$ and $z_0$.
When $t=\tau$, we have $f=\beta_0+\beta_1L_t$, and
$(f,f')(v)=(\beta_0+\beta_1L_t(v),\beta_1L_t'(v))$ is a linear map of
$(\beta_0,\beta_1)$ with bounded entries and determinant $L_t'(v)$,
whose absolute value is comparable to one; the same argument therefore
applies. By \eqref{eq:app-difference-expansion}, this is
\eqref{eq:app-cinematic-metric} for $\mathcal F_{I,T}$ on $J_0$. The $C^2$ norms of the
$G_p$ are bounded because $|y|\le1/2$, so the $C^2$ diameter of
$\mathcal F_{I,T}$ is bounded as well, and $\mathcal F_{I,T}$ is
compact as the continuous image of $\mathcal P$.

For the doubling bound, consider the map
$G_p\mapsto(G_p,G_p',G_p'')(0)\in\R^3$. It is Lipschitz, and by the
preceding estimate at $v=0$ its inverse is Lipschitz on the image.
Thus $\mathcal F_{I,T}$ is bi-Lipschitz equivalent to a subset of
$\R^3$. Every subset of $\R^3$ is doubling with a dimensional constant,
and a bi-Lipschitz map changes the doubling constant by a factor
depending only on the bi-Lipschitz constants; hence $D$ depends only on
$c$ and $z_0$.

\subsubsection*{The output coordinate}
It remains to prove \eqref{eq:app-degenerate-output}. Put
\[
 \mathcal R_t(u)=\partial_tF_t(u)
       -\frac{\partial_t\partial_uF_t(u)}{\partial_uF_t(u)}F_t(u),
\]
with parameter derivatives taken at fixed $u$. Substituting
$\partial_tF_t=z(z+c_1)/(tw)$ and
$\partial_t\partial_uF_t=z/(t^2w)+(z+c_1)/(t^2w^3)$, which follow from
$\partial_tz=-(z+c_1)/t$, we obtain
\[
 \mathcal R_t(u)=\frac1t\left[\frac{z(z+c_1)}{w}
 +F_t(u)\frac{c_1+2z-z^3}{zw^2}\right].
\]
At $u=\alpha t$, that is, $z=z_0$, \eqref{eq:app-critical-data} gives
$z_0(z_0+c_1)/w_0=z_0^2w_0/(1-3z_0^2)$, while
\eqref{eq:app-common-zero-center} gives
$c_1+2z_0-z_0^3=z_0w_0^2(2-3z_0^2)/(1-3z_0^2)$, so by
\eqref{eq:app-critical-data} the second term equals
$-z_0^2w_0/(1-3z_0^2)$. Hence $\mathcal R_t(\alpha t)=0$.
Differentiating in $u$ gives
\[
 \partial_u\mathcal R_t(u)
 =\frac{F_t(u)P_{c_1}(z)}{t^4w^6(\partial_uF_t(u))^2},
\]
so $\partial_u\mathcal R_t(\alpha t)=0$ as well. Since
$\partial_u^2\mathcal R_t$ is bounded on the patch and
$|u-\alpha t|\sim h$ for $u\in I$, Taylor's formula at $\alpha t$
gives $|\mathcal R_t(u)|\lesssim h^2$ on $I\times T$. By \eqref{eq:app-ab},
$b=\partial_uF_\tau(u_*)\bigl[F_\tau/\partial_uF_\tau
-F_t/\partial_uF_t\bigr](u_*)$, and the quotient rule gives
$\partial_t(F_t/\partial_uF_t)=\mathcal R_t/\partial_uF_t$. Hence
\begin{equation}\label{eq:app-intercept-bound}
 |b|=
 \bigg|\partial_uF_\tau(u_*)
 \int_t^\tau\frac{\mathcal R_\rho(u_*)}{\partial_uF_\rho(u_*)}\,d\rho\bigg|
 \lesssim h^2|\tau-t|.
\end{equation}

Now let $p,p'\in\mathcal P$ with $t\ne\tau$, and let $f=G_p-G_{p'}$ with
the coefficients of \eqref{eq:app-difference-expansion}. Solving the
formula for $\beta_0$ in \eqref{eq:app-beta} gives
\[
 y-y'=\beta_0-h^{-1}F_t(u_*)\beta_1-s'b.
\]
By \eqref{eq:app-intercept-bound} and the formula for $\beta_2$ in
\eqref{eq:app-beta}, the
last term is at most $Ch^2|\tau-t|=Ch^{-1}|\beta_2|/s'\le Ch^{-1}|\beta_2|$.
Since $|\beta|\lesssim\|f\|_{C^2(J_0)}$ by the preceding step and
$h\le1$, this gives $|y-y'|\lesssim h^{-1}\|f\|_{C^2(J_0)}$, which is
\eqref{eq:app-degenerate-output}. When $t=\tau$, the same follows from
$y-y'=\beta_0-h^{-1}F_t(u_*)\beta_1$.
\end{proof}

\subsubsection*{Application of Lemma~\ref{lem:app-selected-curves}}
In the proof of Proposition~\ref{prop:translated-pyz-cinematic}, the
output coordinate $y$ is localized to unit intervals; a vertical
translation of the input reduces each of them to $[-\frac12,\frac12]$.
For the family $\mathcal F_{I,T}$ of Lemma~\ref{lem:app-degenerate-graphs},
assign to $G_p$ the output coordinate $\eta(G_p)=hy$. By
\eqref{eq:app-degenerate-output} this is well defined and satisfies
\eqref{eq:app-output-projection} with a constant independent of $h$,
and its range has length $h$. Lemma~\ref{lem:app-selected-curves}
therefore applies to $\mathcal F_{I,T}$ uniformly in $h$, $I$, and
$T$.

\subsection{Frequency localization and summation}\label{sec:app-frequency-sum}
We pass from thickened curves to frequency-localized averages. For
$|c|<1$ the regular arcs of Section~\ref{sec:app-slices} cover the
circle; for $|c|\ge1$ we also sum the dyadic pieces of
Section~\ref{sec:app-common-zeros} near the common zeros.
We now complete the proof of Proposition~\ref{prop:translated-pyz-cinematic}.

\begin{proof}[Proof of Proposition~\ref{prop:translated-pyz-cinematic}]
Set $g=P_\lambda f$.  Its vertical Fourier support is contained in
$|\xi_2|\le C\lambda$.

\subsubsection*{Vertical reproduction}
If a function has vertical Fourier support in
$|\xi_2|\le C\Lambda$, $\Lambda\ge1$, choose a Schwartz function
$\omega$ whose Fourier transform is one on $[-C,C]$.  With
$\omega_\Lambda(z)=\Lambda\omega(\Lambda z)$, vertical convolution
reproduces the function.  For every fixed large $N$,
\begin{equation}\label{eq:app-vertical-reproduction}
 |\omega_\Lambda(z)|
 \lesssim_N\sum_{k\ge0}2^{-Nk}\delta_k^{-1}
                  \mathbf1_{\{|z|\le\delta_k\}},
\end{equation}
where $\delta_k=2^k/\Lambda$.
This bounds graph evaluation by a rapidly convergent sum of positive
vertical averages.  Impose input restrictions only after reproduction,
in the terms with $\delta_k<1$, and dominate bounded densities and
cutoffs by a constant.  For these terms, \eqref{eq:app-selected-strong}
gives $C_{r,\varepsilon}\delta_k^{-\varepsilon}$, whose weighted sum
is at most $C_{r,\varepsilon}\Lambda^\varepsilon$.

\subsubsection*{Regular arcs and the case \texorpdfstring{$|c|<1$}{|c|<1}}
On the regular arcs of Section~\ref{sec:app-slices}, there are only
finitely many fixed localizations.  Freeze one output coordinate, apply
Lemma~\ref{lem:app-selected-curves} to the other, and then integrate
the fixed coordinate.  Restricting the free output coordinate to unit
intervals keeps the graph norms bounded.  The corresponding input
intervals have bounded overlap because the averaging curves lie in
a fixed bounded set.  This gives
$C_{c,r,\varepsilon}\lambda^\varepsilon\|g\|_r$ for the regular arcs.
For each $k$ with $\delta_k\ge1$, bounded graph displacements give
a common rectangle kernel with uniformly bounded $L^1$ norm;
the coefficients in \eqref{eq:app-vertical-reproduction} are summable.
Together with $\|P_\lambda f\|_r\lesssim_r\|f\|_r$, this proves
\eqref{eq:app-diagonal-loss} when there are no common zeros, in
particular when $|c|<1$.

\subsubsection*{Dyadic pieces}
It remains to treat the neighborhoods of the common zeros from
Section~\ref{sec:app-common-zeros}. Let $z_0$ be a common zero and
$\alpha=c_1+z_0$; thus $\alpha=0$ at an axis vertex. Fix
$\chi\in C_c^\infty(\{\tfrac12<|\zeta|<2\})$ with $\sum_h\chi(\zeta/h)=1$ for
$0<|\zeta|\le h_0$, the sum taken over dyadic $h\le h_0$. Here $h_0>0$
depends only on $c$; we take it no larger than the constant $h_0$ of
Lemma~\ref{lem:app-degenerate-graphs} for each nonaxis common zero and
with $h_0^2\le\bar\beta$, and the rest of the arc is a regular arc.
For such $h$, define
\[
 \begin{aligned}
 \A^{c,h}_{t,s}g(x,y)&=\int g\bigl(x-u,\,y-sF_t(u)\bigr)\,
   \chi\Bigl(\frac{u-\alpha t}{h}\Bigr)\,d\mu_t(u),\\
 \M^c_{\lambda,h}f&=\sup_{(t,s)\in\I^2}\bigl|\A^{c,h}_{t,s}P_\lambda f\bigr|,
 \end{aligned}
\]
where $F_t$ is as in \eqref{eq:app-graph} and
\[
 d\mu_t(u)=\frac{du}{2\pi\sqrt{t^2-(u-c_1t)^2}}
\]
is the averaging density. We bound
$\|\M^c_{\lambda,h}f\|_r$ for each $h$ and sum over $h$ at the end of
the proof.

\subsubsection*{Pieces near an axis vertex}
Let $c=(0,-1)$, so that $\alpha=0$ and $F_t(u)=-1+\sqrt{1-u^2/t^2}$.
Put $(x,y)=(h\widetilde x,h^2\widetilde y)$, $u=hv$, and
$g_h(\widetilde x,\widetilde z)=g(h\widetilde x,h^2\widetilde z)$. By
\eqref{eq:app-axis-rescaling},
\[
 \A^{c,h}_{t,s}g(h\widetilde x,h^2\widetilde y)
 =\int g_h\bigl(\widetilde x-v,\,\phi(v;\widetilde y,a,\beta)\bigr)\,
  \chi(v)\,\frac{h}{2\pi t}(1-\beta v^2)^{-1/2}\,dv,
\]
with $a=s/t^2\in[\tfrac14,2]$ and $\beta=h^2/t^2\in[0,\bar\beta]$ for
$h\le h_0$. After a vertical translation, $\widetilde y$ ranges over
$[-\tfrac12,\tfrac12]$. For each $h$, the curves therefore belong to
the single family
\[
 \begin{aligned}
 \mathcal F_{\mathrm{ax}}
 &=\bigl\{\phi(\cdot;\widetilde y,a,\beta):
   \widetilde y\in[-\tfrac12,\tfrac12],\ a\in[\tfrac14,2],\
   \beta\in[0,\bar\beta]\bigr\}\subset C^2(J_\pm),\\
 J_\pm&=\pm[\tfrac12,2],
 \end{aligned}
\]
which is independent of $h$; the two signs of $v$ are treated
separately. By Section~\ref{sec:app-common-zeros}, a fixed finite
subdivision of the graph interval and parameter set gives subfamilies
satisfying the hypotheses of Lemma~\ref{lem:app-selected-curves} with
output coordinate $\widetilde y$. The number of pieces and all constants
are independent of $h$. The
input $g_h$ has vertical Fourier support in $|\xi_2|\le C\lambda h^2$.
We apply \eqref{eq:app-vertical-reproduction} with
$\Lambda=\max\{1,\lambda h^2\}$, then
Lemma~\ref{lem:app-selected-curves} to each of these subfamilies with
input $(v,\widetilde z)\mapsto g_h(\widetilde x-v,\widetilde z)$ at
fixed $\widetilde x$, sum over the pieces, and integrate in
$\widetilde x$ as for the regular arcs. The density is at most $Ch$.
The Jacobians cancel, since
$\|G(h\,\cdot,h^2\,\cdot)\|_r=h^{-3/r}\|G\|_r$ for both the output and
the input $g_h$. Hence
\begin{equation}\label{eq:app-axis-piece}
 \|\M^c_{\lambda,h}f\|_r
 \le C_{c,r,\varepsilon}h\max\{1,\lambda h^2\}^{\varepsilon}\|g\|_r
 \le C_{c,r,\varepsilon}h\lambda^\varepsilon\|g\|_r.
\end{equation}
For $\delta_k\ge1$, the graphs in $\mathcal F_{\mathrm{ax}}$ are
bounded, so in the rescaled coordinates the corresponding average is
dominated by the convolution of $|g_h|$ with
$Ch\,\delta_k^{-1}\mathbf 1_{\{|v|\le2,\ |\widetilde z|\le C\delta_k\}}$,
whose $L^1$ norm is $O(h)$. Since the Jacobians cancel as before, the
weighted sum of these terms is also bounded by $Ch\|g\|_r$.

\subsubsection*{Pieces near a nonaxis common zero}
Let $0<|z_0|<1$ and $h\le h_0$. The support of
$\chi((u-\alpha t)/h)$ is covered by the rectangles $I\times T$
constructed before Lemma~\ref{lem:app-degenerate-graphs}; the
intervals $I$ have bounded overlap, and each $I$ meets a bounded
number of intervals $T$. Fix such a rectangle and the first output
coordinate $x$, and write $u=u_*+hv$, $p=(y,t,s)$, and
$g_{x,I}(v,z)=g(x-u_*-hv,z)$. Then, for $t\in T$, the part of
$\A^{c,h}_{t,s}g(x,y)$ with $u\in I$ equals
\[
 \int_{J_0}g_{x,I}\bigl(v,G_p(v)\bigr)\,
 \chi\Bigl(\frac{u-\alpha t}{h}\Bigr)\,d\mu_t(u_*+hv),
\]
where $G_p\in\mathcal F_{I,T}$ is the graph of
Lemma~\ref{lem:app-degenerate-graphs}. We apply
Lemma~\ref{lem:app-selected-curves} to $\mathcal F_{I,T}$ with output
coordinate $\eta(G_p)=hy$; by \eqref{eq:app-degenerate-output} it
satisfies \eqref{eq:app-output-projection}, and its range has length
$h$. Vertical reproduction acts only in the second variable, so it
leaves $I$ unchanged, and $g_{x,I}$ has vertical Fourier support in
$|\xi_2|\le C\lambda$. Three factors appear: the density contributes
$h$; $\|g_{x,I}\|_{L^r(J_0\times\R)}
=h^{-1/r}\|g(x-\cdot,\cdot)\|_{L^r(I\times\R)}$; and $d\eta=h\,dy$
gives $\|\cdot\|_{L^r(dy)}=h^{-1/r}\|\cdot\|_{L^r(d\eta)}$. Thus the
fine-tube part of $\sup_{t\in T,\,s}|\A^{c,h}_{t,s}g(x,\cdot)|$, with
$u\in I$, is bounded in $L^r(dy)$ by
\[
 C_{c,r,\varepsilon}h^{1-2/r}\lambda^\varepsilon
 \|g(x-\cdot,\cdot)\|_{L^r(I\times\R)}.
\]
Here Lemma~\ref{lem:app-selected-curves} remains uniform even when the
output interval is shorter than the tube width, because its output
measure is unnormalized.

We now assemble the fine tubes. Localize $y$ to the unit intervals
$Y_j=[j,j+1]$, $j\in\mathbb Z$. For $y\in Y_j$, the fine tubes involve
only vertical input coordinates
\[
 z\in Y_j+[-C_c-1,C_c+1],
\]
since $|sF_t(u)|\le C_c$ and $\delta_k<1$. These enlarged intervals
have bounded overlap, uniformly in $h$ and $k$.
For fixed parameters the dyadic arc
has horizontal length $O(h)$, so it meets a bounded number of
intervals $I$. Hence the maximal piece is bounded by a constant times
the $\ell^r$ sum over $I$ of the localized maximal pieces. The
intervals $I$ have bounded overlap and lie in a fixed bounded set
depending only on $c$. Integrating the preceding bound in $x$ and
applying Fubini's theorem therefore gives
\[
 \int_{\R}\sum_I\int_I\int_{\R}
       |g(x-u,z)|^r\,dz\,du\,dx
 =\Big(\sum_I|I|\Big)\|g\|_r^r \lesssim_c\|g\|_r^r.
\]

For $\delta_k\ge1$, the uniform bound $|sF_t(u)|\le C_c$ and
H\"older's inequality on $I$, with $|I|=\eta_0h$, give
\[
 \begin{aligned}
 &\biggl(\frac1{\delta_k}\int_I
       \int_{|z-y+sF_t(u)|<\delta_k}|g(x-u,z)|\,dz\,du\biggr)^r\\
 &\qquad\lesssim h^{r-1}\frac1{\delta_k}
       \int_I\int_{|z-y|<C_c\delta_k}|g(x-u,z)|^r\,dz\,du,
 \end{aligned}
\]
uniformly in the parameters. The same $\ell^r$ sum and integration in
$x,y$ give an $L^r$ bound
$C_ch^{1-1/r}\|g\|_r\le C_ch^{1-2/r}\|g\|_r$, and the rapidly
decreasing coefficients in \eqref{eq:app-vertical-reproduction} absorb
these coarse tubes. We have consequently proved the global bound
\begin{equation}\label{eq:app-nonaxis-piece}
 \|\M^c_{\lambda,h}f\|_r
 \lesssim_{c,r,\varepsilon}
 h^{1-2/r}\lambda^\varepsilon\|g\|_r
\end{equation}
for each nonaxis dyadic piece.

\subsubsection*{Summation}
For $|c|\ge1$, summing over the finitely many graph patches and the
common zeros $z_0$ in each patch, we have pointwise
\[
 \M^c_\lambda f\le\M^c_{\lambda,\mathrm{reg}}f
 +\sum_{z_0}\sum_{h\le h_0}\M^c_{\lambda,h}f,
\]
where $\M^c_{\lambda,\mathrm{reg}}$ is the maximal operator over the
regular arcs; its $L^r$ norm is at most
$C_{c,r,\varepsilon}\lambda^\varepsilon\|g\|_r$ by the bound for the
regular arcs.
The points $u=\alpha t$ themselves carry no averaging measure. The sums
of $h$ in \eqref{eq:app-axis-piece} and of $h^{1-2/r}$ in
\eqref{eq:app-nonaxis-piece} converge over dyadic $h$ for $r\ge3$; at
$r=3$ the latter factor is $h^{1/3}$. Since
$\|P_\lambda f\|_r\lesssim_r\|f\|_r$, this proves
\eqref{eq:app-diagonal-loss} for every fixed $c$.
\end{proof}

\subsection{Diagonal frequency decay}\label{sec:translated-diagonal}
Combining Proposition~\ref{prop:translated-pyz-cinematic} with local
smoothing gives the following decay.

\begin{lemma}\label{lem:translated-diagonal}
For every $3<r<\infty$ there is $\kappa_r>0$ such that, for every fixed
$c\in\R^2$ and every dyadic $\lambda\ge2$,
\[
 \|\M^c_\lambda f\|_r\lesssim_{c,r}\lambda^{-\kappa_r}\|f\|_r.
\]
The decay exponent $\kappa_r$ can be chosen independently of $c$.
\end{lemma}

\begin{proof}
We follow the proof of Lemma~\ref{lem:diagonal} with two changes.
First, the $L^3$ bound is now
Proposition~\ref{prop:translated-pyz-cinematic} with $r=3$:
\[
 \|\M^c_\lambda f\|_3
 \lesssim_{c,\varepsilon}\lambda^\varepsilon\|f\|_3.
\]
Second, for local smoothing we write
$\A^c_{t,s}f(x)=\A^0_{t,s}f\bigl(x-(tc_1,sc_2)\bigr)$. The shear
$(x,t,s)\mapsto(x-(tc_1,sc_2),t,s)$ preserves $dx\,dt\,ds$, so the
local smoothing bound for $\A^0_{t,s}P_\lambda f$ in
$L^{p_0}(dx\,dt\,ds)$ applies to $\A^c_{t,s}P_\lambda f$. The parameter
Sobolev inequality also requires the corresponding spacetime bounds for
$\lambda^{-\alpha-\beta}\partial_t^\alpha\partial_s^\beta
[\A^c_{t,s}P_\lambda f](x)$, with $\alpha,\beta\in\{0,1\}$, where the
derivatives are taken at fixed $x$. By the chain rule, with
$x'=x-(tc_1,sc_2)$,
\[
 \lambda^{-1}\partial_t\bigl[\A^c_{t,s}P_\lambda f\bigr](x)
 =\bigl[\lambda^{-1}\partial_t\A^0_{t,s}P_\lambda f\bigr](x')
 -c_1\bigl[\A^0_{t,s}\widetilde P_\lambda
   (\lambda^{-1}\partial_1P_\lambda f)\bigr](x'),
\]
where $\widetilde P_\lambda$ is a frequency projection to a slightly
larger annulus with $\widetilde P_\lambda P_\lambda=P_\lambda$. The
derivative in $s$ is analogous, with $c_2$ and $\partial_2$. The first
term is treated as in Lemma~\ref{lem:diagonal}. The second is the
undifferentiated operator applied to $\lambda^{-1}\partial_1P_\lambda f$,
whose $L^{p_0}$ norm is $O(\|f\|_{p_0})$. The normalized mixed derivative
$\lambda^{-2}\partial_t\partial_s$ is handled by applying the same chain
rule twice; all resulting centered operators satisfy the same local
smoothing estimate. The kernels of
$\A^c_{t,s}P_\lambda$, $(t,s)\in\I^2$, decay rapidly away from
$B(0,2+2|c|)$, and $\M^c_\lambda$ is translation invariant, so the
localization argument of Lemma~\ref{lem:diagonal} applies. Hence
\[
 \|\M^c_\lambda f\|_{p_0}
 \lesssim_{c,p_0,\varepsilon}
 \lambda^{-1/p_0+\varepsilon}\|f\|_{p_0},\qquad 12\le p_0<\infty.
\]
For the prescribed $3<r<\infty$, choose $p_0>\max\{12,r\}$ and
$0<\vartheta<1$ with $1/r=(1-\vartheta)/3+\vartheta/p_0$. Interpolation
gives
$\|\M^c_\lambda f\|_r\lesssim_{c,r,\varepsilon}
\lambda^{\varepsilon-\vartheta/p_0}\|f\|_r$. Taking
$\varepsilon=\vartheta/(2p_0)$ gives the lemma with
$\kappa_r=\vartheta/(2p_0)$, which depends only on $r$ and the choice
of $p_0$, not on $c$; with the same $p_0$ it equals the exponent of
Lemma~\ref{lem:diagonal}.
\end{proof}

\section{Circles with a fixed nonzero center}\label{sec:translated-case}
We now prove Theorem~\ref{thm:translated} for the translated operator
$\M^c$, with $c\ne0$ fixed as in Section~\ref{sec:translated-result}.
We adapt the frequency-localized argument to $\M^c$ and then prove
the necessary conditions.

\subsection{The frequency-localized estimate and sufficiency}\label{sec:translated-sufficiency}
In view of the diagonal decay in Lemma~\ref{lem:translated-diagonal}, the
argument of Section~\ref{sec:sufficiency} reduces sufficiency to the
following estimate.

\begin{proposition}\label{prop:translated}
Let $A>0$. For every $\varepsilon>0$ and dyadic $\lambda\ge2$,
\[
 \|\M^c_\lambda f\|_8
 \lesssim_{A,\varepsilon}\lambda^\varepsilon\|f\|_4,
 \qquad |c|\le A.
\]
\end{proposition}

\begin{proof}
We first assume that $\lambda$ is sufficiently large. The positive
branch localized to $I$ is
\[
 \A^{c,+}_{\lambda,I}g(x,t,s)
 =\A^+_{\lambda,I}g\bigl(x-(tc_1,sc_2),t,s\bigr),
\]
with phase
\[
 \varphi^c_{t,s}(\xi)=\Phi_{t,s}(\xi)-tc_1\xi_1-sc_2\xi_2.
\]
The shear $x\mapsto x-(tc_1,sc_2)$ preserves spacetime measure,
and each kernel is a spatial translate of the centered kernel.
The undifferentiated spacetime and kernel bounds are therefore
unchanged. This parameter-dependent shear does not in general
preserve the parameter supremum at fixed $x$.
We take parameter derivatives at fixed $x$ in the original
variables $(t,s)$. A fixed smooth parameter cutoff,
equal to one near $\I^2$ and supported in a slightly larger compact
rectangle in $(0,\infty)^2$, is understood below.

\subsubsection*{Separated angular supports}
Let $I,I'\subset I_b$ satisfy the hypotheses of
Proposition~\ref{prop:separated}, with
$\lambda^{-1/2}\lesssim\delta\lesssim b$.
Choose $\theta_0$ between the two supports and write
$\xi=\lambda r e(\theta)$, where $\theta=\theta_0+\delta u$ and
$H=H_{t,s}(\theta)$ is given by \eqref{eq:angular-phase}.
To compare with the centered case, recall the coordinates
\eqref{eq:sobolev-coordinates}. By
\eqref{eq:inverse-sobolev-coordinates}, $\partial_\nu\varphi^c_{t,s}$
differs from $\partial_\nu\Phi_{t,s}$ by
$\mu c_1\xi_1\sin^2\theta_0/(2t)-\mu c_2\xi_2\cos^2\theta_0/(2s)$,
which need not vanish at $\theta=\theta_0$. Thus the cancellation giving the
smaller scale $\lambda b\delta$ in
Section~\ref{sec:parameter-sobolev} need not persist in these
coordinates. In the original parameters $(t,s)$, direct
differentiation gives
\[
 \begin{aligned}
 \lambda^{-1}\partial_t\varphi^c
 &=r\left(\frac{t\cos^2\theta}{H}-c_1\cos\theta\right),\\
 (\lambda b)^{-1}\partial_s\varphi^c
 &=r\left(\frac{s\sin^2\theta}{bH}-c_2\frac{\sin\theta}{b}\right),\\
 (\lambda^2b)^{-1}\partial_t\partial_s\varphi^c
 &=-\frac{rts\cos^2\theta\sin^2\theta}{\lambda bH^3}.
 \end{aligned}
\]
These expressions and their required derivatives in $(t,s,r,u)$
are bounded uniformly for $|c|\le A$; the last expression is
$O(b/\lambda)$. In particular, the $u$ derivatives of
$\sin\theta/b$ are bounded because $\delta\lesssim b$.
The branch symbols \eqref{eq:branch-symbol}, including the smooth
angular cutoffs, have uniform derivatives in the same variables.
The product rule in \eqref{eq:differentiated-symbols}, now with
$\lambda^{-1}\partial_t$ and $(\lambda b)^{-1}\partial_s$,
therefore gives uniformly bounded differentiated amplitudes.
The same applies to a product of two branches, using the sum of
the phases and the product amplitude. Derivatives of the fixed
parameter cutoff cause no loss since $\lambda,\lambda b\gtrsim1$.

After taking these derivatives at fixed $x$, make the shear
$x\mapsto x-(tc_1,sc_2)$ in the spacetime norm. The phase is then
the centered phase of Proposition~\ref{prop:spacetime}, and the
differentiated amplitudes satisfy \eqref{eq:symbol-bounds}.
Their angular supports are unchanged. The prefactor of the
product is $\lambda^{-1}(2\pi)^{-4}$, so that proposition and
Plancherel give, for $\alpha,\beta\in\{0,1\}$,
\[
 \Big\|\lambda^{-\alpha}(\lambda b)^{-\beta}
 \partial_t^\alpha\partial_s^\beta
 \bigl(\A^{c,+}_{\lambda,I}g_1\,
       \A^{c,+}_{\lambda,I'}g_2\bigr)\Big\|_{L^2_{x,t,s}}
 \lesssim_A\lambda^{-1}b^{-1/2}\delta^{-2}
              \|g_1\|_2\|g_2\|_2.
\]
Let $\mathcal B^c_{\lambda,I,I'}$ denote the analogue of
\eqref{eq:bilinear-maximal} for the translated positive branch.
In the shell decomposition, $b\ge\lambda^{-1/2}$, so
$\lambda b\ge\lambda^{1/2}\ge1$.
Lemma~\ref{lem:parameter-sobolev}, with scales $\lambda,\lambda b$
and exponent $2$, gives
\[
 \begin{aligned}
 \|\mathcal B^c_{\lambda,I,I'}(g_1,g_2)\|_2
 &\lesssim_A
 (\lambda^2b)^{1/2}\lambda^{-1}b^{-1/2}\delta^{-2}
 \|g_1\|_2\|g_2\|_2\\
 &=\delta^{-2}\|g_1\|_2\|g_2\|_2.
 \end{aligned}
\]
The spacetime estimate has the same dependence on
$\lambda,b,\delta$ as \eqref{eq:bilinear-L2}.
Compared with \eqref{eq:bilinear-L2-result}, the additional
factor $\delta^{-1/2}$ comes from the Sobolev cost
$\lambda b^{1/2}$ in place of $\lambda(b\delta)^{1/2}$.

The kernel $L^1$ bound gives the unchanged $L^\infty$ estimate
\eqref{eq:bilinear-Linfty}. If we linearize the common parameter
supremum as in Section~\ref{sec:separated-proof} and interpolate
with equal weights, the factors $\delta^{-2}$ and $\delta^2$
cancel. Restoring the supremum gives
\begin{equation}\label{eq:translated-separated}
 \|\mathcal B^c_{\lambda,I,I'}(g_1,g_2)\|_4
 \lesssim_A\|g_1\|_4\|g_2\|_4.
\end{equation}

\subsubsection*{Terminal intervals}
For an interval of width $0<h\lesssim_{C_0}\lambda^{-1/2}$,
use the parameter scales $\lambda,\lambda$.
The phase derivatives normalized by these scales are bounded
uniformly in the center direction. After the linear cap rescaling
\eqref{eq:cap-frequency-map}, their required frequency and
parameter derivatives are also uniformly bounded for $|c|\le A$.
The same product rule therefore gives uniform seminorms for all
four differentiated branch amplitudes. Lemma~\ref{lem:kernel}\eqref{lem:kernel-small}--\eqref{lem:kernel-amplitude}
applies to these amplitudes and gives kernel
$L^1$ bounds $O_{A,C_0}(\lambda^{-1/2})$.
Young's inequality and Lemma~\ref{lem:parameter-sobolev} at
exponent $7$ now yield
\[
 \Big\|\sup_{(t,s)\in\I^2}|\A^{c,+}_{\lambda,I}g|\Big\|_7
 \lesssim_{A,C_0}\lambda^{2/7}\lambda^{-1/2}\|g\|_7
 =\lambda^{-3/14}\|g\|_7.
\]
Interpolate this with the kernel $L^1\to L^\infty$ bound
$C_{C_0}\lambda$, using
\[
 \left(\frac14,\frac18\right)
 =\frac78\left(\frac17,\frac17\right)+\frac18(1,0).
\]
After restoring the supremum as before, we obtain
\begin{equation}\label{eq:translated-short}
 \Big\|\sup_{(t,s)\in\I^2}|\A^{c,+}_{\lambda,I}g|\Big\|_8
 \lesssim_{A,C_0}\lambda^{-1/16}\|g\|_4.
\end{equation}

\subsubsection*{Angular iteration and summation}
We repeat Section~\ref{sec:frequency-proof} with
\eqref{eq:translated-separated} and \eqref{eq:translated-short} in
place of Propositions~\ref{prop:separated} and \ref{prop:short}.
Use the partitions \eqref{eq:shell-partition} and the terminal scale
\eqref{eq:terminal-scale}, and define $F_j,G_j$ by
\eqref{eq:iteration-maxima}, with the translated branch in place of
the centered branch. The separated-pair sum becomes
\[
 \|G_j\|_8^8\lesssim_{A,K}
 \Bigl(\sum_{I\in\mathcal I_j}\|f_I\|_4^4\Bigr)^2.
\]
By \eqref{eq:input-sum} and Lemma~\ref{lem:cordoba} at exponent $4$,
each $\|G_j\|_8$ is bounded by a fixed power of $1+\log\lambda$
times $\|f_{I_b}\|_4$.
For the corresponding terminal maximum $F_L$,
\eqref{eq:translated-short} gives
\[
 \begin{aligned}
 \|F_L\|_8^8
 &\lesssim_{A,K}\lambda^{-1/2}
   \sum_{I\in\mathcal I_L}\|f_I\|_4^8\\
 &\le\lambda^{-1/2}
   \Bigl(\sum_{I\in\mathcal I_L}\|f_I\|_4^4\Bigr)^2.
 \end{aligned}
\]
Thus the same square-function bound holds for $\|F_L\|_8$,
including when $L=0$. The recurrence \eqref{eq:iteration}, the
choice of $K$, and the summation over the angular shells are
unchanged. The axis interval $I_{\mathrm{ax}}$ is controlled by
\eqref{eq:translated-short} with $C_0=2$, and the sharp shell
projections are uniformly bounded on $L^4$ by the argument preceding
\eqref{eq:shell-projection-bound}.
Reflections and coordinate interchange preserve $|c|$,
so all constants remain uniform for $|c|\le A$. The negative
branch follows by the conjugation argument used at the end of
Section~\ref{sec:frequency-proof}. This proves the stated estimate
for large $\lambda$.

For the remaining bounded values of $\lambda$, the kernels of
$\A^c_{t,s}P_\lambda$ have a common Schwartz majorant, uniformly
for $(t,s)\in\I^2$ and $|c|\le A$. Its $L^{8/7}$ norm is finite,
so Young's inequality gives the required $L^4\to L^8$ bound.
\end{proof}

Finally, combine Proposition~\ref{prop:translated} with
Lemma~\ref{lem:translated-diagonal} and the unit-mass $L^\infty$ bound.
For fixed $c$, the kernels of $\A^c_{t,s}P_{\le1}$ have a common
Schwartz majorant, so Young's inequality also gives the required
low-frequency bound. The interpolation and dyadic summation in
Section~\ref{sec:sufficiency}, with $(1/4,1/8)$ replacing $(2/7,1/7)$,
give every point of $\operatorname{int}\Delta^\ast$ and the diagonal
$3<p=q\le\infty$.

\subsection{Necessary conditions}\label{sec:translated-necessity}
We first establish the condition $5/p-2/q\le1$, then adapt the tests from Section~\ref{sec:necessary-conditions}.

\subsubsection*{The condition $5/p-2/q\le1$}

The following constructions apply to every fixed $c\ne0$.
Suppose first that $c_1\ne0$. Choose $\varepsilon_0\in\{-1,1\}$ with $\varepsilon_0c_2\ge0$, so that $|c_2+\varepsilon_0|=1+|c_2|\ge1$. At the output point
\begin{equation}\label{eq:vertex-output}
 x_{\mathrm{vert}}(t,s)=(c_1t,(c_2+\varepsilon_0)s),
\end{equation}
the full average is exactly
\[
 \A^c_{t,s}f(x_{\mathrm{vert}}(t,s))
 =\frac1{2\pi}\int_0^{2\pi}
 f(-t\cos\theta,s[\varepsilon_0-\sin\theta])\,d\theta.
\]
On the semicircle where $\sin\theta$ has sign $\varepsilon_0$, use
\[
 u=-t\cos\theta,
 \quad \sin\theta=\varepsilon_0\sqrt{1-u^2/t^2},
 \quad |d\theta|=\frac{|du|}{\sqrt{t^2-u^2}}.
\]
The semicircle is parameterized one-to-one, up to its endpoints, by $-t<u<t$. For $f\ge0$, positivity gives
\begin{equation}\label{eq:vertex-average}
 \A^c_{t,s}f(x_{\mathrm{vert}}(t,s))
 \ge\frac1{2\pi}\int_{-t}^t
 f\Bigl(u,\varepsilon_0s\bigl[1-\sqrt{1-u^2/t^2}\bigr]\Bigr)
 \frac{du}{\sqrt{t^2-u^2}}.
\end{equation}

Let $0<h\le1/10$ and restrict the parameters to
\[
 s=t^2(1+v),
\]
where $5/4\le t\le4/3$ and $0\le v\le h^2$.
These values belong to $[1,2]^2$, since
$1<s\le(16/9)(1+1/100)<2$. Since $s/t^2=1+v$, rationalizing the square root gives, for $|u|<t$,
\begin{equation}\label{eq:quartic-identity}
 s\bigl(1-\sqrt{1-u^2/t^2}\bigr)-\frac{u^2}{2}
 =\frac{vu^2}{2}
  +\frac{s u^4}{2t^4(1+\sqrt{1-u^2/t^2})^2}.
\end{equation}
For $|u|\le h$ the right-hand side of \eqref{eq:quartic-identity} lies between zero and $3h^4/2$, since $vu^2/2\le h^4/2$ and $s/t^4=(1+v)/t^2<1$.

Take $\chi,\eta\in C_c^\infty((-2,2))$, bounded between zero and one and each equal to one on $[-1,1]$, and set
\begin{equation}\label{eq:quartic-input}
 f_h(u,z)=\chi(u/h)\,
 \eta\left(\frac{z-\varepsilon_0u^2/2}{2h^4}\right).
\end{equation}
For finite $p$, integration first in $z$ and then in $u$ gives
\begin{equation}\label{eq:quartic-input-norm}
 \begin{aligned}
 \|f_h\|_p^p
 &=2h^4\|\eta\|_p^p\int_{\R}|\chi(u/h)|^p\,du\\
 &=2h^5\|\chi\|_p^p\|\eta\|_p^p.
 \end{aligned}
\end{equation}
Thus $\|f_h\|_p\sim_p h^{5/p}$.
For $p=\infty$ its norm is one. By \eqref{eq:quartic-identity}, the integrand $f_h$ in \eqref{eq:vertex-average} is identically one for $|u|\le h$. Since $t\le2$,
\begin{equation}\label{eq:quartic-lower}
 \A^c_{t,t^2(1+v)}f_h
     \bigl(c_1t,(c_2+\varepsilon_0)t^2(1+v)\bigr)
 \ge\frac1{2\pi}\int_{-h}^h\frac{du}{\sqrt{t^2-u^2}}
 \ge\frac{h}{2\pi}.
\end{equation}

The output map and its Jacobian are
\[
 (t,v)\longmapsto
 \bigl(c_1t,(c_2+\varepsilon_0)t^2(1+v)\bigr),
 \quad
 \det\frac{\partial x}{\partial(t,v)}
 =c_1(c_2+\varepsilon_0)t^2.
\]
It is injective: $t=x_1/c_1$ and
$v=x_2/[(c_2+\varepsilon_0)t^2]-1$. Thus its image $E_h$ has area
\begin{equation}\label{eq:quartic-area}
 |E_h|=|c_1(c_2+\varepsilon_0)|h^2
       \int_{5/4}^{4/3}t^2\,dt\sim_c h^2.
\end{equation}
Equations \eqref{eq:quartic-input-norm}--\eqref{eq:quartic-area} show that even a weak type bound requires
\[
 h^{1+2/q}\lesssim_{c,p,q}h^{5/p}\quad\text{as }h\downarrow0.
\]
Consequently
\begin{equation}\label{eq:nec-quartic}
 5/p-2/q\le1.
\end{equation}
The support of \eqref{eq:quartic-input} has area comparable to $h^5$, and its indicator dominates $f_h$. Thus \eqref{eq:nec-quartic} is necessary even for restricted weak type.

If $c_1=0$ and $c_2\ne0$, let $S(x_1,x_2)=(x_2,x_1)$. Since $d\sigma$ is invariant under $y\mapsto Sy$, substituting $Sy$ for $y$ in the definition of $\sigma_{t,s}$ and interchanging $t$ and $s$ gives
\[
 \M^cf(Sx)=\M^{Sc}(f\circ S)(x).
\]
The transformed vector $Sc=(c_2,0)$ has nonzero first coordinate.
Since $S$ preserves Lebesgue measure, the construction for $Sc$ gives
\eqref{eq:nec-quartic} for $c$ as well. This treats every nonzero $c$.

\subsubsection*{The remaining necessary conditions}
The remaining necessary conditions follow from the tests in Section~\ref{sec:necessary-conditions}.
Every averaging curve is contained in $B(0,2(1+|c|))$, so the
large-disc test still gives $1/q\le1/p$.
For fixed $(t_0,s_0)\in\I^2$, take the indicator of the
$\delta$-neighborhood of $-(t_0c_1,s_0c_2)-E_{t_0,s_0}$,
where $E_{t_0,s_0}$ is the centered ellipse with semiaxes $t_0,s_0$.
Its support has area comparable to $\delta$, and its average at
$(t_0,s_0)$ is one on $B(0,\delta)$. Thus $1/p\le2/q$.

For the cusp obstruction, coordinate reflections allow us to assume
$c_1,c_2\le0$. Retain $r_0\in(1,2)$ from
Section~\ref{sec:necessary-conditions} and use the cusp test functions
\eqref{eq:cusp-test-function}. Replace the outputs \eqref{eq:test-centers} by
\[
 x^c(t,s)=x(t,s)+(tc_1,sc_2).
\]
Then $\A^c_{t,s}f(x^c(t,s))=\A^0_{t,s}f(x(t,s))$.
To retain a set of outputs of positive area, note that at
$(t,s)=(r_0,r_0)$ the functions $u,v$ in \eqref{eq:test-centers} satisfy
\[
 D_{(t,s)}(u,v)=\frac1{2\sqrt2}
 \begin{pmatrix}3&-1\\-1&3\end{pmatrix}.
\]
Consequently,
\[
 \det D_{(t,s)}x^c
 =\left(c_1-\frac3{2\sqrt2}\right)
  \left(c_2-\frac3{2\sqrt2}\right)-\frac18\ge1.
\]
The inverse function theorem therefore gives a fixed neighborhood of
$(r_0,r_0)$ on which $x^c$ is injective and its Jacobian is bounded
away from zero. Choose a fixed closed rectangle $\mathcal Q_0\Subset\mathcal Q$
of positive area inside this neighborhood; then
$\Omega^c=x^c(\mathcal Q_0)$ has fixed positive area. Fix the
constants and the cutoff functions $\chi_1,\chi_2,\chi_3$ in
\eqref{eq:cusp-test-function} as in
Section~\ref{sec:necessary-conditions}, independently of $\nu$. Since
$\A^c_{t,s}f_\nu(x^c(t,s))=\A^0_{t,s}f_\nu(x(t,s))$, the lower bound of
Section~\ref{sec:necessary-conditions} holds on $\Omega^c$, and these
test functions show that weak type fails for $p\le3$. Together with
\eqref{eq:nec-quartic}, the two elementary tests and this cusp test
give precisely the necessary conditions in
Theorem~\ref{thm:translated}.

\appendix

\section{Selected curves and regularized averages}\label{app:selected-curves}
We prove the selected-curves estimate used in
Section~\ref{sec:translated-diagonal-estimates} and then deduce a
regularized bound from Proposition~\ref{prop:translated-pyz-cinematic}.

\subsection{Proof of Lemma~\ref{lem:app-selected-curves}}
The proof uses the PYZ estimate \eqref{eq:app-pyz-incidence} and the
output projection \eqref{eq:app-output-projection}.

\begin{proof}
Fix $\varepsilon>0$. We first prove \eqref{eq:app-selected-strong}
for $r=3$ and $\delta$ sufficiently small.

\subsubsection*{Discrete selections}
Let $E\subset\eta(\mathcal F)$ be measurable and select $\phi_\eta$
measurably for $\eta\in E$.  For $0\le a<\delta$, the curves selected at
$a+n\delta\in E$ are distinct, and \eqref{eq:app-output-projection}
puts at most $C\rho/\delta$ of them in any $C^2$ ball of radius
$\rho\ge\delta$.  Every subset has the same nonconcentration bound.
Using \eqref{eq:app-pyz-incidence} with a sufficiently small loss gives,
for each such subset $F$ and every $\gamma>0$,
\[
 \Big\|\sum_{\phi\in F}\mathbf1_{\phi^\delta}\Big\|_{3/2}
 \lesssim_\gamma\delta^{2/3-\gamma}(\#F)^{2/3}.
\]
Here $\delta$ is sufficiently small in terms of the stated bounds and
$\gamma$ only; fixed nonconcentration constants are absorbed in the loss.

\subsubsection*{Weighted sums}
For fixed $a$, assign nonnegative weights $w_i$ to the selected curves
$\phi_i$, $1\le i\le N$, and relabel them so that $w_1\ge\cdots\ge w_N$.
Abel summation and H\"older's inequality give
\[
 \begin{aligned}
 \Big\|\sum_{i=1}^Nw_i\mathbf1_{\phi_i^\delta}\Big\|_{3/2}
 &\lesssim_\gamma\delta^{2/3-\gamma}
   \sum_{i=1}^Nw_i\bigl(i^{2/3}-(i-1)^{2/3}\bigr)\\
 &\lesssim_\gamma\delta^{2/3-\gamma}(\log(2+N))^{1/3}
   \Bigl(\sum_{i=1}^Nw_i^{3/2}\Bigr)^{2/3}.
 \end{aligned}
\]
Indeed, $i^{2/3}-(i-1)^{2/3}\lesssim i^{-1/3}$.  The bounded output
interval gives $N\lesssim1+\delta^{-1}$; choosing $0<\gamma<\varepsilon$
absorbs the loss and logarithm into $\delta^{-\varepsilon}$.

\subsubsection*{Continuous selections}
Decomposition into shifted lattices and Minkowski's and H\"older's inequalities yield
\begin{equation}\label{eq:app-continuous-incidence}
 \begin{aligned}
 \Big\|\delta^{-1}\int_Ew(\eta)\mathbf1_{\phi_\eta^\delta}\,d\eta\Big\|_{3/2}
 &\lesssim_\varepsilon\delta^{-1/3-\varepsilon}\int_0^\delta
   \Bigl(\sum_{n:a+n\delta\in E}w(a+n\delta)^{3/2}\Bigr)^{2/3}da\\
 &\le C_\varepsilon\delta^{-\varepsilon}\|w\|_{L^{3/2}(E)},
 \end{aligned}
\end{equation}
for nonnegative $w\in L^{3/2}(E)$, using the exact identity
\[
 \int_0^\delta\sum_{n:a+n\delta\in E}w(a+n\delta)^{3/2}\,da
 =\int_Ew(\eta)^{3/2}\,d\eta.
\]

\subsubsection*{Duality}
Let $g$ be continuous with compact support, and let $\kappa>0$. Since
$\mathcal F$ is compact, $\eta$ is continuous, and the tube average
depends continuously on $\phi$ in the $C^2(J)$ topology, a measurable
selection theorem gives a selection $\eta\mapsto\phi_\eta$ on
$E=\eta(\mathcal F)$ with $\eta(\phi_\eta)=\eta$ satisfying the
following inequality.
\[
 T_\delta g(\eta)
 \le\frac1{2\delta}\int_J\int_{-\delta}^{\delta}
   |g(v,\phi_\eta(v)-z)|\,dz\,dv+\kappa
 =\frac1{2\delta}\int|g|\,\mathbf1_{\phi_\eta^\delta}+\kappa.
\]
For nonnegative $w\in L^{3/2}(E)$, Fubini's theorem, H\"older's
inequality, and \eqref{eq:app-continuous-incidence} give
\[
 \begin{aligned}
 \int_Ew\,T_\delta g\,d\eta
 &\le\int|g|\,\frac1{2\delta}\int_Ew(\eta)\mathbf1_{\phi_\eta^\delta}\,d\eta
   +\kappa\|w\|_{L^1(E)}\\
 &\le C_\varepsilon\delta^{-\varepsilon}\|g\|_3\|w\|_{L^{3/2}(E)}
   +\kappa\|w\|_{L^1(E)}.
 \end{aligned}
\]
Letting $\kappa\to0$ and taking the supremum over
$\|w\|_{L^{3/2}(E)}\le1$ proves \eqref{eq:app-selected-strong} for $r=3$.

\subsubsection*{Extensions}
For fixed $\delta$
and $g_n\to g$ in $L^3$, sublinearity and H\"older's inequality give
\[
 \|T_\delta g-T_\delta g_n\|_\infty
 \le (2\delta)^{-1/3}|J|^{2/3}\|g-g_n\|_3.
\]
Since the output interval is bounded, the maximal averages converge in
$L^3(d\eta)$, and the estimate extends to general inputs.  Scales with $\delta$ bounded away
from zero follow from H\"older's inequality.  Finally,
$\|T_\delta\|_{\infty\to\infty}\le |J|$. Linearizing $T_\delta$ by a
measurable selection as above and interpolating gives
\eqref{eq:app-selected-strong} for $r>3$.
\end{proof}

\subsection{Regularized averages}
The endpoint in Proposition~\ref{prop:translated-pyz-cinematic} also
gives the following estimate for regularized averages.

\begin{corollary}\label{cor:translated-regularized}
Fix $c\in\R^2$, and let $\rho\in C_c^\infty(\R^2)$ be nonnegative
with $\int\rho=1$.  Put $\rho_\delta(x)=\delta^{-2}\rho(x/\delta)$.
For every $\varepsilon>0$,
\begin{equation}\label{eq:translated-regularized}
 \Big\|\sup_{(t,s)\in\I^2}
       \A^c_{t,s}(\rho_\delta*|f|)\Big\|_3
 \lesssim_{c,\rho,\varepsilon}\delta^{-\varepsilon}\|f\|_3,
 \qquad 0<\delta\le1.
\end{equation}
\end{corollary}

If $\rho\ge c_0>0$ on the unit ball, the left-hand side of
\eqref{eq:translated-regularized} dominates the maximal function of
averages over $\delta$-neighborhoods of the translated ellipses, which is
the physical-space form of such estimates; compare~\cite{PYZ}. We record
it for this reason; it is not used elsewhere in the paper.

\begin{proof}
We first take $g=|f|$ with $f$ smooth and compactly supported; the
general case follows at the end by approximation. Use the radial
partition and the projections $P_\lambda$ from
Section~\ref{sec:frequency-reduction}; in particular, $P_\lambda$ has
symbol $\psi(|\xi|/\lambda)$. Let $\widetilde P_\lambda$ have symbol $\widetilde\psi(|\xi|/\lambda)$, where
$\widetilde\psi\in C_c^\infty((1/4,4))$ equals one on $\supp\psi$.
Then $P_\lambda\widetilde P_\lambda=P_\lambda$.  After $\xi=\lambda\zeta$,
the symbol of $\rho_\delta*\widetilde P_\lambda$ is
$\widetilde\psi(|\zeta|)\widehat\rho(\delta\lambda\zeta)$.
For every $N>0$, its derivatives through order four are bounded by
$C_{\rho,N}\min\{1,(\delta\lambda)^{-N}\}$ on a fixed annulus.
Integration by parts gives the same bound for the kernel's $L^1$ norm.
Thus Proposition~\ref{prop:translated-pyz-cinematic}, with $r=3$, gives
\[
 \Big\|\sup_{(t,s)\in\I^2}
   |\A^c_{t,s}P_\lambda(\rho_\delta*g)|\Big\|_3
 \lesssim_{c,\rho,\varepsilon,N}
 \lambda^\varepsilon\min\{1,(\delta\lambda)^{-N}\}\|g\|_3.
\]
For the low frequencies, the bounded averaging displacements and the
Schwartz kernel of $P_{\le1}$ give a common integrable majorant for the
kernels of $\A^c_{t,s}P_{\le1}$.  Young's inequality bounds this term by
$C_c\|g\|_3$, uniformly in $0<\delta\le1$.
Apply these bounds to $g=|f|$ and sum over dyadic $\lambda\ge2$.
For $N>\varepsilon$, the sum of
$\lambda^\varepsilon\min\{1,(\delta\lambda)^{-N}\}$ is
$O_{\varepsilon,N}(\delta^{-\varepsilon})$, proving
\eqref{eq:translated-regularized} for such $f$ by the triangle inequality.
For general $f\in L^3$, approximation in $L^3$ gives uniform convergence
of $\rho_\delta*|f|$ for fixed $\delta$, and Fatou's lemma completes the
proof.
\end{proof}

\section*{Acknowledgement}
This work is supported by the National Research Foundation of Korea (Grant No.~RS-2024-00342160: Sanghyuk Lee and Seunghwan Lee; No.~RS-2024-00461749: Jin Bong Lee, Sanghyuk Lee, and Jeongtae Oh). 
During the course of this work, the authors used ChatGPT and Codex (OpenAI), and Claude (Anthropic) to extend and develop the basic ideas of the paper, to check computations, to review the mathematical arguments, and to assist with English-language editing and LaTeX preparation. All mathematical content was verified by the authors, who take full responsibility for it.

\end{document}